\documentclass[11pt,reqno]{amsart}
\usepackage[utf8]{inputenc}
\usepackage{amsmath,amssymb,amsthm,mathtools}
\usepackage{enumitem}
\usepackage[margin=1in]{geometry}
\usepackage{comment}

\newif\ifPDFout
\ifx\pdfoutput\undefined\else\ifnum\pdfoutput>0 \PDFouttrue\fi\fi
\ifPDFout
  \usepackage{hyperref}
  \hypersetup{colorlinks=true,linkcolor=blue,citecolor=blue}
\else
  \providecommand{\texorpdfstring}[2]{#1}
  
\fi

\DeclareTextSymbol{\S}{OMS}{120}
\DeclareTextSymbolDefault{\S}{OMS}

\theoremstyle{plain}
\newtheorem{theorem}{Theorem}[section]
\newtheorem{proposition}[theorem]{Proposition}
\newtheorem{lemma}[theorem]{Lemma}
\newtheorem{corollary}[theorem]{Corollary}
\theoremstyle{definition}
\newtheorem{definition}[theorem]{Definition}
\newtheorem{remark}[theorem]{Remark}
\newtheorem{example}[theorem]{Example}

\numberwithin{equation}{section}
\DeclareMathOperator{\GL}{GL}
\DeclareMathOperator{\Mat}{Mat}
\DeclareMathOperator{\Hom}{Hom}
\DeclareMathOperator{\Ind}{Ind}
\DeclareMathOperator{\tr}{tr}
\DeclareMathOperator{\val}{val}
\DeclareMathOperator{\WF}{WF}
\DeclareMathOperator{\SL}{SL}
\DeclareMathOperator{\Sp}{Sp}
\DeclareMathOperator{\SO}{SO}
\DeclareMathOperator{\GSpin}{GSpin}
\DeclareMathOperator{\vol}{vol}
\DeclareMathOperator{\pr}{pr}
\DeclareMathOperator{\PD}{PD}
\newcommand{\Oo}{\mathcal{O}}
\newcommand{\Rep}{\mathrm{Rep}}
\newcommand{\Irr}{\mathrm{Irr}}
\newcommand{\bR}{\mathbb{R}}
\newcommand{\bC}{\mathbb{C}}
\newcommand{\bZ}{\mathbb{Z}}
\newcommand{\bA}{\mathbb{A}}
\newcommand{\bF}{\mathbb{F}}
\newcommand{\cW}{\mathcal{W}}
\newcommand{\cX}{\mathcal{X}}
\newcommand{\cV}{\mathcal{V}}

\newcommand{\cT}{\mathcal{T}}
\newcommand{\diag}{\operatorname{diag}}
\newcommand{\abs}[1]{\lvert #1\rvert}
\newcommand{\norm}[1]{\lVert #1\rVert}
\newcommand{\extn}[2]{\lvert\wedge^{#1}(#2)\rvert_{\max}}
\newcommand{\ct}{\mathfrak{c}}

\newcommand{\bN}{\overline{N}}
\newcommand{\bu}{\bar{u}}
\newcommand{\bv}{\bar{v}}
\newcommand{\fd}{\mathfrak{d}}
\newcommand{\fX}{\mathfrak{X}}
\newcommand{\fY}{\mathfrak{Y}}
\newcommand{\bV}{\overline{V}}
\newcommand{\Psn}{\Psi_{N}}
\newcommand{\Psu}{\Psi_{U'}}

\begin{document}

\title[Bounds for \texorpdfstring{\lowercase{$(k,c)$}}{(k,c)} model functions]
{Bounds for \texorpdfstring{\lowercase{$(k,c)$}}{(k,c)} model functions over \texorpdfstring{\lowercase{$p$}}{p}-adic fields}
\author{Eyal Kaplan}
\address{Department of Mathematics, Bar Ilan University, Ramat Gan 5290002, Israel}
\email{kaplaney@gmail.com}

\begin{abstract}
Let $F$ be a $p$-adic field, let $k,c\ge1$, and let $\rho$ be a $(k,c)$ representation of
$\GL_{kc}(F)$: a complex representation of finite length whose wave front set has $(k^c)$ as
its unique maximal orbit, and which admits a $(k,c)$ functional that is unique up to scaling.
We prove that the functions in the
$(k,c)$ model of $\rho$ are of moderate growth, by relating this model to the Zelevinsky model
of $\rho$, and obtain as a consequence a bound for the restriction of these functions to the
standard Levi subgroup associated with the composition $(c^k)$ of $kc$. As applications, we study convergence problems 
of several integrals involving $(k,c)$ models: integrals from the generalized doubling method, and a new family of
integrals which extends the Rankin--Selberg integrals for Speh representations of Lapid and Mao.
\end{abstract}

\maketitle

\section{Introduction}\label{sec:intro}

Let $F$ be a non-archimedean local field of characteristic $0$, with ring of integers $\Oo$. Let $k,c\ge1$, and denote $n=kc$. Fix a nontrivial additive character $\psi$ of $F$.
Let $P=M\ltimes U$ be the standard parabolic subgroup of $\GL_n$ attached to the composition $(c^k)$, with $U$ its unipotent radical.
Define the character $\psi_k$ of $U(F)$ by $\psi_k(u)=\psi\big(\sum_{i=1}^{(k-1)c}u_{i,c+i}\big)$. A (complex) finite-length representation $\rho$ of $\GL_n(F)$
is called a $(k,c)$ representation if the nilpotent orbit $(k^c)$ is the unique maximal orbit in its wave front set and, moreover, $\dim\Hom_{U(F)}(\rho,\psi_k)=1$.
In this case a nonzero $\lambda\in \Hom_{U(F)}(\rho,\psi_k)$ is called a $(k,c)$ functional on $\rho$. The $(k,c)$ model is spanned by the functions \[W_{\xi}(g):=\lambda(\rho(g)\xi)\] with $\xi$ in the space $\cV_\rho$ of $\rho$. Define a height $\norm{\cdot}$ on $\GL_n(F)$ by composing the embedding $g\mapsto \diag(g,g^{-1})$
of $\GL_n(F)$ into $\SL_{2n}(F)$ with the sup-norm on $\SL_{2n}(F)$.
\begin{theorem}[see Theorem~\ref{thm:mgpsi}]\label{thm:intromg}
Let $\rho$ be a $(k,c)$ representation on a space $\cV_\rho$, and fix a $(k,c)$ functional $\lambda$ on $\rho$.
There is a constant $D\ge0$ such that for each $\xi\in\cV_\rho$ there is a constant $C(\xi)\geq0$ such that
\[\abs{W_{\xi}(g)}\leq C(\xi)\norm{g}^D,\qquad\forall g\in \GL_n(F).\]
\end{theorem}

The stabilizer of $\psi_k$ in $M(F)$ is the diagonal embedding $\GL_c^{\triangle}(F)$ of $\GL_c(F)$ in $M(F)$. The functional $\lambda$ translates on the left under
$\GL_c^{\triangle}(F)$ with respect to a character $\eta_\rho$ of $\GL_c(F)$. If we write $m\in M(F)$ in the form $m=m_k^{\triangle}m'$, $m'=(m_k^{-1}m_1,\ldots,m_k^{-1}m_{k-1},I_c)$, we obtain
\[W_{\xi}(m)=\eta_\rho(m_k)W_{\xi}(m').\]
By a well-known argument going back to asymptotics of Whittaker functions (see, e.g., \cite{JPSS79}), the support of the function
$m'\mapsto W_{\xi}(m')$ is bounded in the sense that the entries of $m'$ belong to a certain (large) compact subgroup of $F$.
When we combine Theorem~\ref{thm:intromg} with this observation we derive the following result, a certain asymptotic bound for $(k,c)$ functions
along the full Levi subgroup $M$:

\begin{theorem}[see Theorem~\ref{thm:gaugepsi}]\label{thm:introgauge}
Let $D$ be the exponent of Theorem~\ref{thm:intromg}. For any $\xi\in \cV_{\rho}$, there is a constant $\gamma(\xi)$ such that for all
$m=\diag(m_1,\ldots,m_k)\in M(F)$, $u\in U(F)$ and $\kappa\in \GL_n(\Oo)$,
\begin{align*}
\abs{W_{\xi}(um\kappa)}\leq\gamma(\xi)\abs{\eta_\rho(m_k)}\cdot\prod_{i=1}^{k-1}\abs{\det (m_k^{-1}m_i)}^{-D}.
\end{align*}
In addition, there are Schwartz--Bruhat functions $\phi_1,\ldots,\phi_{k-1}$ on the space of $c\times c$ matrices over $F$ such that for all $m,u$ and $\kappa$ as above,
\begin{align*}
\abs{W_{\xi}(um\kappa)}\leq\abs{\eta_\rho(m_k)}\cdot\prod_{i=1}^{k-1}\abs{\det (m_k^{-1}m_i)}^{-D}\phi_i(m_{i+1}^{-1}m_i).
\end{align*}
\end{theorem}

The analog of Theorem~\ref{thm:introgauge} over archimedean fields was proved in \cite[Lemma 3.14]{CFK-func} using the ideas of \cite{AGJ09} (see the discussion below).

Cai \textit{et al.} \cite{CFGK23} introduced a map $\rho_c$ from irreducible generic representations $\tau$ of $\GL_k(F)$ to $(k,c)$ representations of $\GL_n(F)$. When $\tau$ arises from a global construction, that is, $\tau$ is a local component of an irreducible cuspidal automorphic representation $\tau'$, and $\rho=\rho_c(\tau)$, it is useful to have a uniform bound for the constant $D$ in the results above, namely that $D$ applies simultaneously to $\rho_c(\tau'_v)$ for all finite places $v$. Our results express the constant $D$ using the cuspidal exponents of $\tau$, and we can thus obtain this uniformity, see
Theorems~\ref{thm:mgpsi} and \ref{thm:gaugepsi}.

In the local--global setting, that is, for $\rho=\rho_c(\tau'_v)$, where $\tau'$ is as above, we also give a
second, global proof of Theorems~\ref{thm:intromg} and \ref{thm:introgauge}, using a global
$(k,c)$ coefficient, see \eqref{eq:globalcoefficient} and Theorem~\ref{thm:globalgauge}.
This argument is softer than the local one, because in this setting we obtain the local theorem, 
Theorem~\ref{thm:intromg}, from the moderate growth of a single automorphic form. It also yields the uniformity of $D$ in $v$.

Our results also apply to certain central extensions of the group $\GL_n$; we discuss this in \S~\ref{sec:coverings}.

The elementary difficulty in showing that $W_{\xi}$ is of moderate growth is to control its restriction to $M(F)$. Our idea is to first show that functions in the Zelevinsky model are of moderate growth, then relate the two models. Let $N$ be the subgroup of $\GL_n$ of upper triangular unipotent matrices, and define a degenerate character $\Psi_N$ of $N(F)$ which is trivial on $U(F)$ but restricts to the generic character on $N(F)\cap M(F)$. Since $(k^c)$ is maximal in the wave front set of $\rho$,
$\rho$ admits a Zelevinsky model, i.e., a nonzero map $\rho\to\Ind_{N(F)}^{\GL_n(F)}(\Psi_N)$ which is unique up to scaling, see Lemmas~\ref{lem:zelexists} and \ref{lem:zele1dim}.
It is then straightforward to prove that functions in the Zelevinsky model have moderate growth, since these functions are left invariant under $U(F)$ and define Whittaker functions on $N(F)\cap M(F)$, where we can use the known asymptotic expansions for Whittaker functions, specifically \cite{LM09} and \cite{CM15}, to obtain the bound.

The two models are related by the transition map $\cT$ of Lapid and Mao \cite{LM20}, which is essentially a map
\[\mathcal{T}:\Ind_{N(F)}^{\GL_n(F)}\Psi_N\to\Ind_{U(F)}^{\GL_n(F)}\psi_k.\] We observe that $\mathcal{T}$ is injective (Lemma~\ref{lem:injective}), hence, because
$\dim\Hom_{U(F)}(\rho,\psi_k)=1$, the $(k,c)$ functional on $\rho$ can be realized by composing $\mathcal{T}$ with evaluation at the identity, see
Corollary~\ref{cor:identification}. It remains to bound $\mathcal{T}W^{\mathcal Z}_\xi$ where $W^{\mathcal Z}_\xi$ belongs to the image of $\rho$ in $\Ind_{N(F)}^{\GL_n(F)}(\Psi_N)$. The map $\mathcal{T}$ is given by an integral, and the integrand---$W^{\mathcal Z}_\xi$---is compactly supported on the domain of integration.
Our key insight here is to bound the support of $g\cdot W^{\mathcal Z}_\xi$ with explicit dependence on $g$ (Lemma~\ref{lem:keysupport}). This implies
Theorem~\ref{thm:intromg}.

It is worthwhile to compare the proof here with the archimedean proof. In the archimedean case, Theorem~\ref{thm:intromg} is immediate because
$\rho$ itself is of moderate growth and the functional is continuous. Over $p$-adic fields, since $\lambda$ is not smooth, $W_\xi$ is not a matrix coefficient of $\rho$ and there is a priori no reason for $W_\xi(g)$ to be bounded by a power of $\norm g$.
But the support argument, which allows us to deduce Theorem~\ref{thm:introgauge} from Theorem~\ref{thm:intromg}, is insufficient in the archimedean case,
because Schwartz functions need not be compactly supported; even if they were, the bounding constant would not be continuous as a function of the vector $\xi$. Instead, the proof proceeds
by differentiating with respect to the action of the enveloping algebra, in a way similar to the classical proof that the Fourier transform is a continuous map $\mathcal S(\bR)\to \mathcal S(\bR)$.

For $c=1$, the $(k,1)$ functional is simply a Whittaker functional. In this case, by the Iwasawa decomposition of $\GL_n(F)$ with respect to a Borel subgroup, one reduces to the study of Whittaker functions on the diagonal torus. Here the argument is classical and uses only the geometric properties of the functional; the results hold even if the dimension of the space of $(k,1)$ functionals is more than $1$. Very briefly, two functions $f_1$ and $f_2$ on the diagonal torus are defined to be equivalent if $f_1-f_2$ vanishes along a certain direction. The space of equivalence classes, i.e., the germ at $0$, is a quotient of a Jacquet module of the representation along the unipotent radical of a certain parabolic subgroup, and the center of the Levi part acts locally finitely. This gives the asymptotic expansion. See, e.g., \cite{JPSS79,JPSS83,CS80,LM09}.

When $k=2$, the $(2,c)$ functional is similar to the Shalika functional studied by Jacquet and Rallis \cite[\S 6.2]{JR96} (the Shalika model appeared first in \cite{JacquetShalika1990} in a global context); in that case an asymptotic expansion along the diagonal torus, similar to the case $c=1$, also leads to the bounds we obtain, and our results can be derived directly from their work.
For the case $k=2$ the asymptotic expansion uses a bit more, namely the left equivariance of $\lambda$ under $\GL_c^{\triangle}$ (which follows immediately when $\lambda$ is unique up to scaling), but still does not use the assumptions on the wave front set of $\rho$ (\cite[\S 6.2]{JR96} considered a more general class of representations).

In fact, it is straightforward to adapt \cite[\S 6.2]{JR96} to the restriction of $(k,c)$ model functions, for all $k$ and $c$, to
$\GL_c(F)$ embedded in the top left corner of $\GL_n(F)$. In a slightly more general setup, namely for the restriction of $(k,c)$ functions to
\[\{\diag(g,z_1I_c,\ldots,z_{k-1}I_c):g\in\GL_c(F),z_i\in F^\times\},\]
one can use the beautiful description of Cogdell and Matringe \cite{CM15} in order to obtain an asymptotic expansion. This was carried out
by Luo \cite{Luo24} with interesting applications.

An asymptotic expansion is stronger than the bounds we obtain, and has proved useful for studying questions of meromorphic continuation, but the modern approach to
meromorphic continuation is by Bernstein's continuation principle. In theory, one may attempt to write an asymptotic expansion
along the diagonal torus in the general case as well, for arbitrary $c$ and $k$, using certain vanishing properties which are implied by the assumption on the wave front set.
Even with the assumption on the wave front set, this turns out to be a difficult problem, and this is not surprising in light of the fact that the general setup falls outside the case of spherical varieties of Sakellaridis \cite{Yia2} and Sakellaridis and Venkatesh \cite{SakellaridisVenkatesh}.
To complicate things even further, even with a conjectural asymptotic expansion, it is not clear how to obtain Theorem~\ref{thm:intromg}, since restriction to the diagonal torus does not say much about the restriction to the full Levi subgroup $M(F)$, and
the only additional information is the equivariance under $\GL_c^{\triangle}(F)$.

The rest of this work is organized as follows.
Section~\ref{sec:prelim} fixes the notation; in particular $(k,c)$ representations are discussed in \S~\ref{ssec:kcdata}.
In \S~\ref{sec:whittaker} we obtain the moderate growth of Whittaker functions, in the precise form which will be useful later.
Section~\ref{sec:transition} contains the description of the transition map from \cite{LM20}, and then in
\S~\ref{sec:zelevinsky} we use this map to relate the Zelevinsky model of $\rho$ and the $(k,c)$ model of $\rho$, and prove the moderate growth of the Zelevinsky model functions.
In \S~\ref{sec:mg} we prove the moderate growth of the $(k,c)$ model functions, and in \S~\ref{sec:gauge} the gauge estimate.
Section~\ref{sec:global} contains the global argument. Section~\ref{sec:coverings} concerns central extensions.
Section~\ref{sec:applications} contains applications of our results: we prove convergence of certain
integrals involving $(k,c)$ models, among which are integrals from the (local) generalized doubling method of \cite{CFGK2,CFK22,CFGK23}, and a new family of integrals which extends the Rankin--Selberg integrals for Speh representations of Lapid and Mao \cite{LM20}. Finally in Appendix~\ref{sec:GLdoubling} we prove in detail the convergence of the
generalized doubling integral for $\GL_c$ from \cite{CFGK2,CFK22}.

\subsection*{Acknowledgments}
This research was supported by the ISRAEL SCIENCE FOUNDATION (grant No. 376/21).
The author also acknowledges the use of Claude Fable 5.1, a large language model developed by Anthropic, PBC (San Francisco, CA, USA),
for assistance with the mathematical computations and the writing of this manuscript. The author retains sole responsibility for
the correctness of the final results.

\setcounter{tocdepth}{1}
\tableofcontents

\section{Preliminaries}\label{sec:prelim}

\subsection{Notation}\label{ssec:notation}
Throughout, $F$ is a $p$-adic field with ring of integers $\Oo$, uniformizer $\varpi$,
residue cardinality $q$, valuation $\val$ and absolute value normalized by
$\abs\varpi=q^{-1}$. For $x=(x_1,\dots,x_m)\in(F^\times)^m$ we write $\val(x)=(\val(x_1),\dots,\val(x_m))\in\bZ^m$. Except in 
\S~\ref{sec:global}, all algebraic groups in this work are defined over $F$. 
By convention, for such a group $H$ we denote its $F$-points also by $H$.

Let $\Mat_{l\times m}$ denote the space of $l \times m$ matrices over $F$, and set $\Mat_l:=\Mat_{l\times l}$.

Fix $k,c\ge1$ and put $n=kc$, $G=\GL_n(F)$, $K=\GL_n(\Oo)$, and
$K(r)=I_n+\varpi^{r}\Mat_n(\Oo)$ for $r\ge1$; each $K(r)$ is compact and open, and a normal subgroup of $K$. Fix in $G$ the Borel
subgroup $B$ of upper triangular invertible matrices and write $B=T\ltimes N$ where
$T$ is the diagonal torus and $N$ is the group of upper unitriangular matrices. Then we have the Iwasawa decomposition
$G=BK$. Let $\bN\subset G$ be the subgroup of lower unitriangular matrices. For a composition $\underline{n}=(n_1,\ldots,n_l)$ of $n$,
let $P_{\underline{n}}=M_{\underline{n}}\ltimes V_{\underline{n}}$ be the standard parabolic subgroup attached to $\underline{n}$, with $V_{\underline{n}}$ the unipotent
radical. The image in $M_{\underline{n}}$ of an $l$-tuple $(g_1,\dots,g_l)$, $g_i\in\GL_{n_i}$, is the block-diagonal matrix
$\diag(g_1,\ldots,g_l)$. The modulus character of $P_{\underline{n}}$ is denoted $\delta_{P_{\underline{n}}}$. Denote by $\overline{P}_{\underline{n}}=M_{\underline{n}}\ltimes \overline{V}_{\underline{n}}$ the
opposite parabolic subgroup relative to $M_{\underline{n}}$.

When we consider several general linear groups simultaneously, we use the notation $\GL_l$, then $B_l=T_l\ltimes N_l$ is the fixed Borel (as above) and $K_l=\GL_l(\Oo)$.

Let $E_{ij}$ be the matrix unit, and let $e_1,\dots,e_n$ be the standard basis of the space
of row vectors, on which $G$ acts on the right.

All representations in this work are smooth complex representations of $l$-groups. For $H$ as above, denote the
category of smooth complex representations of finite length by $\Rep(H)$ and recall that when $H$ is reductive (as an algebraic group), such representations are in particular admissible. Let $\Irr(H)$ denote the subset of irreducible objects in $\Rep(H)$.

We fix an arbitrary nontrivial additive character $\psi$ of $F$, and let $e_\psi\in\bZ$ be its conductor exponent:
$\psi$ is trivial on $\varpi^{e_\psi}\Oo$ and nontrivial on $\varpi^{e_\psi-1}\Oo$. Thus, for $x\in F$
and $r\in\bZ$,
\begin{equation}\label{eq:conductor}
\psi(x\varpi^{r}\Oo)=1\quad\Longleftrightarrow\quad \abs x\le q^{r-e_\psi}.
\end{equation}
The character $\psi$ is unramified if $e_\psi=0$.

Haar measures on $F$ and on its finite-dimensional vector spaces are the self-dual ones
with respect to $\psi$, so that $\vol(\Oo)=q^{e_\psi/2}$.

Let $\mathcal{S}(\Mat_{l\times m})$ denote the space of Schwartz--Bruhat functions on $\Mat_{l\times m}$.

A vector $\xi$ in the space of a representation of $G$ has level $r$ if $K(r)$ fixes it; then it also has level $r'$ for every $r'\ge r$, and throughout we take levels $r$ with $r>e_\psi$ (recall $r\ge1$).

We use the normalized induction functor $\Ind$ (see \cite[1.8]{BZ77}). For example, if $H$ is an $l$-group, $V\subset H$ is a closed unimodular subgroup (in this
work, a unipotent subgroup) and $\psi_V$ is a character of $V$, then $\Ind_V^H\psi_V$ is the
space of smooth functions $f\colon H\to\bC$ such that $f(vh)=\psi_V(v)f(h)$ for all $v\in V$
and $h\in H$, with $H$ acting by right translation which we denote by $\cdot$.
We also use the Bernstein--Zelevinsky product functor  $\times$ (see \cite[4.1]{BZ77}).

Jacquet modules are unnormalized. For a unipotent group $U$, a character $\chi$ of $U$ and a
representation $(\pi,\cV)$ of a group containing $U$, $J_{U,\chi}(\pi)$ is the quotient of
$\cV$ by the span of the vectors $\pi(u)\xi-\chi(u)\xi$, and $J_U(\pi)=J_{U,1}(\pi)$. The
functors $J_{U,\chi}$ are exact \cite[1.9]{BZ77}. If $U$ is the unipotent radical of a parabolic subgroup $P$ of an algebraic group $H$, and $\pi\in\Rep(H)$, we call $J_{U}(\pi)$ the Jacquet module of $\pi$ along $P$. In this case, if $H$ is reductive, then $J_U(\pi)$ has finite length.

For a split reductive group $H$, if $\pi\in\Rep(H)$ admits a central character, we denote it by $\omega_\pi$. For a parabolic subgroup $P$ of $H$, 
let $M_P$ denote a Levi part of $P$ and $Z_{M_P}$ denote the identity component of the center of $M_P$. The set of cuspidal exponents of $\pi$ is the set 
of characters by which $Z_{M_P}$ acts on the irreducible subquotients of the Jacquet module of $\pi$ along
$P$, where $P$ varies over all standard parabolic subgroups of $H$.

Let $\pi\in\Rep(G)$ and let $\mathfrak g$ be the Lie algebra of $G$. The character
of $\pi$ defines a distribution on the space of smooth compactly supported functions on $G$.
By Howe \cite{How74} and Harish-Chandra \cite{HC78}, the transport of this distribution to a
small neighborhood of $0$ in $\mathfrak g$ (through the exponential map) is a finite linear
combination $\sum_{\mathcal O}c_{\mathcal O}(\pi)\widehat{\mu_{\mathcal O}}$
of the Fourier transforms of suitably normalized invariant measures $\mu_{\mathcal O}$ on
the nilpotent coadjoint orbits $\mathcal O$ of $G$. Following \cite[\S1.4]{GK23}, the
wave front set $\WF(\pi)$ is the set of orbits $\mathcal O$ with $c_{\mathcal O}(\pi)\ne0$
such that $c_{\mathcal O'}(\pi)=0$ for every orbit $\mathcal O'\ne\mathcal O$ whose closure
contains $\mathcal O$. The nilpotent orbits of $G$ are indexed by the partitions of $n$,
through the sizes of the Jordan blocks, and the closure order on orbits is the dominance
order on partitions.

\subsection{The height}\label{ssec:height}
For $g\in G$ put
\begin{equation}\label{eq:height}
\norm g\ :=\ \max_{i,j}\big(\abs{g_{ij}},\abs{(g^{-1})_{ij}}\big).
\end{equation}
Then $\norm g\ge1$, and we use the following standard properties (see, e.g., \cite[\S I.1]{Wal03}).
\begin{enumerate}[leftmargin=*,label=(\arabic*),ref=(\arabic*)]
\item\label{it:h1} $\norm{\kappa g\kappa'}=\norm g$ for $\kappa,\kappa'\in K$.
\item\label{it:h2} $\norm{gh}\le\norm g\,\norm h$.
\item\label{it:h3} If $g=ua$ with $u\in V_{\underline{n}}$ and $a\in M_{\underline{n}}$, then $\norm g\ge\norm a$,
since the diagonal blocks of $ua$ and of $(ua)^{-1}=a^{-1}u^{-1}$ are those of $a$ and of $a^{-1}$.
In particular $\norm{ut}\ge\norm t$ for $u\in N$ and $t\in T$.
\item\label{it:h4} For $t=\diag(t_1,\dots,t_n)\in T$, $\norm t=\max_i\abs{t_i}^{\pm1}$, and
$\abs{\val t_i}\le\log_q\norm t$ for every $i$.
\item\label{it:h5} If $\bu\in\bN$ has all of its entries of absolute value at most $R\ge1$,
then $\norm\bu\le R^{n-1}$, because $\bu^{-1}=\sum_{j=0}^{n-1}(-x)^j$ for $\bu=I_n+x$ with $x$ nilpotent.
\end{enumerate}
For a matrix $g$ with entries in $F$ and an integer $r\ge0$, let $\extn{r}{g}$ be the maximum of the absolute values of the $r\times r$ minors of $g$, with $\extn{0}{g}=1$.
The $r\times r$ minors of $g$ are the coordinates of the matrix of the exterior power $\wedge^r(g)$ with respect to the standard bases, which explains the notation.
By the Cauchy--Binet formula and the ultrametric inequality,
\begin{equation}\label{eq:GLCB}
\extn{r}{g_1g_2}\le\extn{r}{g_1}\extn{r}{g_2},
\end{equation}
for matrices $g_1$ and $g_2$ such that $g_1g_2$ is defined, 
and by Jacobi's identity for the minors of the inverse matrix, for $g\in\GL_l$ and $0\le r\le l$,
\begin{equation}\label{eq:GLjacobi}
\extn{r}{g}=\abs{\det g}\,\extn{l-r}{g^{-1}}.
\end{equation}
For a matrix $R$ with $s$ rows and $m\ge s$ columns, put $\ct(R):=\extn{s}{R}$, the maximum of
the absolute values of its $s\times s$ minors; we call $\ct(R)$ the content of $R$. For a
row vector $v$ (the case $s=1$), $\ct(v)$ is the maximum of the absolute values of its
entries. By \eqref{eq:GLCB}, for $x\in G$,
\begin{equation}\label{eq:cauchybinet}
\ct(Rx)\ \le\ \ct(R)\,\extn{s}{x}.
\end{equation}
Since each $s\times s$ minor of $g^{-1}$
is a sum of products of $s$ entries of $g^{-1}$, \eqref{eq:cauchybinet} applied to
$Rg$ and $x=g^{-1}$ gives
\begin{equation}\label{eq:mug}
\ct(R)=\ct\big((Rg)g^{-1}\big)\ \le\ \ct(Rg)\,\norm g^{s},\qquad g\in G .
\end{equation}
For $\kappa\in K$, all minors of $\kappa$ lie in $\Oo$, so \eqref{eq:cauchybinet} gives
$\ct(R\kappa)\le\ct(R)$, and \eqref{eq:mug} with $g=\kappa$ gives the reverse inequality,
because $\norm\kappa=1$; hence
\begin{equation}\label{eq:mukappa}
\ct(R\kappa)=\ct(R),\qquad\kappa\in K .
\end{equation}

\subsection{Iwasawa coordinates of unipotent elements}\label{ssec:iwasawa}
We formulate here a result of a preliminary nature, which will be used only in \S~\ref{sec:applications} and in Appendix~\ref{sec:GLdoubling}; the reader may skip this subsection on
first reading.

For a matrix $x$ with entries in $F$ put
\begin{equation}\label{eq:Nx}
\vartheta(x):=\max\big(1,\max_{i,j}\abs{x_{ij}}\big),
\end{equation}
so that $\vartheta(g)\le\norm g$ for an invertible $g$, and $\norm\bu\le\vartheta(\bu)^{n-1}$ for
$\bu\in\bN$ by property~\ref{it:h5}.

Let $\underline{n}=(n_1,\dots,n_d)$ be a composition of $n$ with $d\ge2$.
The subgroup $\bV_{\underline n}$ is isomorphic as a variety to the affine space of dimension
\begin{equation}\label{eq:dimVbar}
\fd_{\underline n}:=\sum_{i<j}n_in_j .
\end{equation}
For $\bv\in\bV_{\underline n}$ and $1\le l\le d$, let $\mathcal{R}_l(\bv)$ be the
matrix formed by the rows of $\bv$ in the blocks $l,\dots,d$, and put
$\ct_l(\bv):=\ct(\mathcal{R}_l(\bv))$ and $\ct_{d+1}(\bv):=1$.
The following lemma is the $p$-adic, block form of a classical observation: the
Iwasawa coordinates of a unipotent element are computed by the exterior powers 
of the submatrices formed by its last rows. See, e.g., \cite[Lemma~5.1]{Jac09} over archimedean fields for the Borel subgroup, and
\cite[\S11.15, Lemmas~1 and~2]{Sou93}.
\begin{lemma}\label{lem:iwasawaunip}
Let $\bv\in\bV_{\underline n}$ and write $\bv=um\kappa$ with $u\in V_{\underline n}$,
$m=\diag(m_1,\dots,m_d)\in M_{\underline n}$ and $\kappa\in K$ (the Iwasawa
decomposition of $G$ with respect to $P_{\underline n}$). Then:
\begin{enumerate}[leftmargin=*,label=(\roman*),ref=(\roman*)]
\item\label{it:iw1} $\abs{\det m_l}=\ct_l(\bv)/\ct_{l+1}(\bv)$ for $1\le l\le d$, and $\ct_1(\bv)=1$.
\item\label{it:iw2} $1\le\ct_l(\bv)\le \vartheta(\bv)^{n}$ for $1\le l\le d$.
\item\label{it:iw3} $\vartheta(\bv)\le\max_{2\le l\le d}\ct_l(\bv)$.
\item\label{it:iw4} $\norm m\le\norm\bv\le \vartheta(\bv)^{n-1}$.
\end{enumerate}
\end{lemma}

\begin{proof}
Throughout, $u_{ij}\in\Mat_{n_{i}\times n_{j}}$ denotes the $(i,j)$ block of $u$, so
that $u_{ii}=I_{n_{i}}$ and $u_{ij}=0$ for $i>j$, and similarly for $\bv$, whose
blocks satisfy $\bv_{ii}=I_{n_{i}}$ and $\bv_{ij}=0$ for $i<j$. For $1\le l\le d$ put
$n_{\geq l}:=n_l+\dots+n_d$, and let $\bv_{[l]}\in\Mat_{n_{\geq l}}$ be the submatrix of $\bv$ formed by the
rows and the columns in the blocks $l,\dots,d$, that is, $\bv_{[l]}=(\bv_{ij})_{l\le i,j\le d}$;
it is block lower unitriangular, in particular lower unitriangular, so $\det\bv_{[l]}=1$.

For \ref{it:iw1}, the $(i,j)$ block of $um$ is $u_{ij}m_{j}$. Fix $l$ and consider the
rows of $um$ in the blocks $i$, $i\ge l$. In these rows the blocks with $j<l$ vanish, since
$j<l\le i$, so these rows form the matrix $(0\mid L_l)$ with
\[
L_l=(u_{ij}m_{j})_{l\le i,j\le d}\in\Mat_{n_{\geq l}},
\]
a block upper triangular matrix with diagonal blocks $u_{ii}m_{i}=m_{i}$; hence
$\abs{\det L_l}=\prod_{j\ge l}\abs{\det m_j}$. Since $\bv=(um)\kappa$, the rows of $\bv$ in
the blocks $l,\dots,d$ are $\mathcal{R}_l(\bv)=(0\mid L_l)\kappa$, and by \eqref{eq:mukappa}
$\ct_l(\bv)=\ct((0\mid L_l))$. The only nonzero maximal minor of $(0\mid L_l)$ is
$\det L_l$, so $\ct_l(\bv)=\prod_{j\ge l}\abs{\det m_j}$. Dividing the identities for $l$
and $l+1$ gives $\abs{\det m_l}=\ct_l(\bv)/\ct_{l+1}(\bv)$, and for $l=1$ the identity reads
$\ct_1(\bv)=\abs{\det m}=\abs{\det\bv}=1$.

For \ref{it:iw2}, the columns of $\mathcal{R}_l(\bv)$ in the blocks $l,\dots,d$ form the matrix
$\bv_{[l]}$, whose determinant is $1$; so $\ct_l(\bv)\ge1$. Every maximal minor
of $\mathcal{R}_l(\bv)$ is a sum of products of at most $n$ entries of $\bv$, so it has absolute
value at most $\vartheta(\bv)^{n}$.

For \ref{it:iw3}, the entries of $\bv$ in the diagonal blocks are $0$ or $1$, those above
the diagonal blocks are $0$, and the remaining entries lie in rows of the blocks $l\ge2$
(the rows of the first block contain no entry outside the diagonal block). Hence it
suffices to show that $\abs{x}\le\ct_l(\bv)$ for every entry $x$ of $\bv$ lying in a row $r_l$
of a block $l\ge2$ and in a column $b$ of a block $j<l$. We exhibit a maximal minor of
$\mathcal{R}_l(\bv)$ equal to $\pm x$. Let $S\in\Mat_{n_{\geq l}}$ be the square submatrix of $\mathcal{R}_l(\bv)$ formed by the
column $b$ and by the columns $r$, for the rows $r\ne r_l$ of $\mathcal{R}_l(\bv)$, that is, the
columns of the diagonal positions of the rows other than $r_l$. Ordering the rows of $S$ so that $r_l$ comes first, and
the columns as $b$ first and then the columns $r$ in increasing order,
\[
S=\begin{pmatrix} x & 0\\ \ast & \bv'\end{pmatrix},\qquad \bv'\in\Mat_{n_{\geq l}-1},
\]
where $\bv'$ is obtained from $\bv_{[l]}$ by deleting the row and the column of $r_l$. The row
$r_l$ of $S$ vanishes outside the column $b$, because for $r\ne r_l$ in a block $\ge l$ the $(r_l,r)$ entry
of $\bv$ is either in the diagonal block $I_{n_l}$ off
its diagonal, or above the diagonal blocks. Expanding along the first row,
$\det S=x\det\bv'$, and $\bv'$ is lower unitriangular, since deleting a row and the column
with the same index from a lower unitriangular matrix leaves a lower unitriangular
matrix. Thus $\abs{\det S}=\abs x$, and $\abs x\le\ct_l(\bv)$ by the definition of $\ct$.
Together with $\ct_l(\bv)\ge1$, this gives $\vartheta(\bv)\le\max_{2\le l\le d}\ct_l(\bv)$.

For \ref{it:iw4}, $\norm\bv=\norm{um}\ge\norm m$ by properties~\ref{it:h1}
and \ref{it:h3}, and $\norm\bv\le\vartheta(\bv)^{n-1}$ by property~\ref{it:h5}, as
$\bV_{\underline n}\subset\bN$.
\end{proof}

\begin{corollary}\label{cor:iwasawaunip}
Let $\bv\in\bV_{\underline n}$ and write $\bv=um\kappa$ with $u\in V_{\underline n}$,
$m=\diag(m_1,\dots,m_d)\in M_{\underline n}$ and $\kappa\in K$. Then:
\begin{enumerate}[leftmargin=*,label=(\alph*),ref=(\alph*)]
\item\label{it:ci1} $\norm m\le\vartheta(\bv)^{n-1}$.
\item\label{it:ci2} $\prod_{j\ge l}\abs{\det m_j}\ge1$ for every $1\le l\le d$, and $\vartheta(\bv)\le\max_{2\le l\le d}\prod_{j\ge l}\abs{\det m_j}$.
\item\label{it:ci3} Let $x_1,\dots,x_d$ be real numbers and $r>0$ such that $x_i-x_{i+1}\ge r$ for all $1\le i<d$. Then
$\prod_{i=1}^{d}\abs{\det m_i}^{x_i}\le\vartheta(\bv)^{-r}$.
\end{enumerate}
\end{corollary}

\begin{proof}
Put $\ct_l:=\ct_l(\bv)$. Assertion \ref{it:ci1} is Lemma~\ref{lem:iwasawaunip}\ref{it:iw4}. By Lemma~\ref{lem:iwasawaunip}\ref{it:iw1},
$\prod_{j\ge l}\abs{\det m_j}=\ct_l$ (a telescoping product, using $\ct_{d+1}=1$), so \ref{it:ci2} is Lemma~\ref{lem:iwasawaunip}\ref{it:iw2} and \ref{it:iw3}.
For \ref{it:ci3}, by Lemma~\ref{lem:iwasawaunip}\ref{it:iw1} and $\ct_1=\ct_{d+1}=1$,
\[
\prod_{i=1}^{d}\abs{\det m_i}^{x_i}=\prod_{i=1}^{d}\Big(\frac{\ct_i}{\ct_{i+1}}\Big)^{x_i}=\prod_{i=1}^{d-1}\ct_{i+1}^{-(x_{i}-x_{i+1})}\le\prod_{i=2}^{d}\ct_i^{-r}\le\ct_{l_0}^{-r}\le\vartheta(\bv)^{-r},
\]
where $l_0\ge2$ is an index with $\ct_{l_0}=\max_{2\le l\le d}\ct_l$; the inequalities use $\ct_i\ge1$, $r>0$ and \ref{it:iw3}.
\end{proof}

\begin{remark}\label{rem:iwasawaH}
Let $H$ be a subgroup of $G$ and $\bv\in H\cap\bV_{\underline n}$. If we can write $\bv=um\kappa$ 
such that $u\in H\cap V_{\underline n}$, $m\in H\cap M_{\underline n}$ and $\kappa\in H\cap K$, then we can apply 
Lemma~\ref{lem:iwasawaunip} and Corollary~\ref{cor:iwasawaunip} to $\bv$ and $m$. 
The lemmas of Soudry \cite[\S11.15, Lemmas~1 and~2]{Sou93} are statements of this kind for orthogonal groups but with the Iwasawa decomposition always 
taken with respect to the Borel subgroup, in which case the torus coordinates are bounded both ways
by explicit powers of $\vartheta(\bv)$. 
\end{remark}
\begin{example}\label{ex:twoblock}
In \S~\ref{ssec:lambdaconv} and \S~\ref{ssec:doubling} we will use the following bounds. Let 
$X\in\Mat_{n_1}$ and $\bv=\left(\begin{smallmatrix}I_{n_1}&\\X&I_{n_1}\end{smallmatrix}\right)\in\GL_{2n_1}$. 
Then $\bv=um\kappa$ with $u\in V_{(n_1,n_1)}$, $m=\diag(m_1,m_2)\in M_{(n_1,n_1)}$ and $\kappa\in K_{2n_1}$. By Lemma~\ref{lem:iwasawaunip}\ref{it:iw1}
$\abs{\det m_2}=\ct((X\mid I_{n_1}))$, and Corollary~\ref{cor:iwasawaunip}\ref{it:ci1} and \ref{it:ci2} read
\begin{equation}\label{eq:twoblockbounds}
\abs{\det m_2}\ \ge\ \vartheta(X),\qquad \norm m\ \le\ \vartheta(X)^{2n_1-1}.
\end{equation}
Better bounds are available: by comparing the bottom rows $(X\mid I_{n_1})=(0\mid m_2)\kappa$ we see that the entries of $m_2$ are bounded by
$\vartheta(X)$ and those of $m_2^{-1}$ lie in $\Oo$, so $\norm{m_2}\le\vartheta(X)$; such improvements are not important for our applications. 
\end{example}

\subsection{Representations of type \texorpdfstring{$(k,c)$}{(k,c)}}\label{ssec:kcdata}
Let $U=V_{(c^k)}$ and define the character $\psi_k$ of $U$ by
\begin{equation}\label{eq:psikdef}
\psi_k(u)=\psi\Big(\sum_{i=1}^{k-1}\tr u_{i,i+1}\Big),\qquad u=(u_{i,j})\in U,\ u_{i,j}\in\Mat_c .
\end{equation}
($\tr$ denotes the trace map.) 
Let $M=M_{(c^k)}$. For $g\in\GL_c$ let $g^{\triangle}=\diag(g,\dots,g)\in M$; the stabilizer of $\psi_k$ in $M$ is $\GL_c^{\triangle}=\{g^{\triangle}:g\in\GL_c\}$.

\begin{definition}[{\cite[\S1.4]{GK23}}]\label{def:kc}
A representation $\rho\in\Rep(G)$ is a $(k,c)$ representation if $(k^c)$ is
the unique maximal orbit in $\WF(\rho)$ and $\dim\Hom_U(\rho,\psi_k)=1$.
\end{definition}
For $k=1$ the group $U$ and the character $\psi_1$ are trivial; the condition becomes $\dim\Hom(\rho,1)=1$ hence
$\dim \cV_\rho^*=1$, thus $\rho$ is a character $\eta\circ\det$ of $\GL_c$ where $\eta$ is a quasi-character of $F^\times$.

Let $\rho$ be a $(k,c)$ representation and let $\cV_\rho$ denote its space. A nonzero
$\lambda\in\Hom_U(\rho,\psi_k)$ is called a $(k,c)$ functional on $\rho$; it is unique up
to a scalar, and we fix one. By definition
\begin{equation}\label{eq:equivariance}
\lambda\big(\rho(u)\xi\big)=\psi_k(u)\,\lambda(\xi),\qquad
u\in U,\ \xi\in\cV_{\rho}.
\end{equation}
The $(k,c)$ model $W_\psi(\rho)$ of $\rho$ is the space of functions $W_\xi\colon G\to\bC$, defined for $\xi\in\cV_\rho$ by
\begin{equation}\label{eq:kcmodel}
W_\xi(g)=\lambda(\rho(g)\xi).
\end{equation}
The group $G$ acts on $W_\psi(\rho)$ by right translation (denoted by $\cdot$); this space does not depend on the choice of $\lambda$.
The map $\xi\mapsto W_\xi$ is a nonzero $G$-equivariant map from $\rho$ to $\Ind_U^G\psi_k$. In particular
\begin{equation}\label{eq:equivarianceKC}
W_\xi(ug)=\psi_k(u)W_\xi(g),\qquad \forall u\in U, g\in G.
\end{equation}

Since $\GL_c^{\triangle}$ stabilizes $\psi_k$, it acts on the one-dimensional space
$J_{U,\psi_k}(\rho)$, hence by a character, which we denote by $\eta_\rho$:
\begin{equation}\label{eq:etarho}
\lambda\big(\rho(g^{\triangle})\xi\big)=\eta_\rho(g)\,\lambda(\xi),\qquad g\in\GL_c,\ \xi\in\cV_\rho .
\end{equation}
In particular, for
$g=\diag(g_1,\dots,g_k)\in M$, writing
$g=g_k^{\triangle}\,\diag(g_k^{-1}g_1,\dots,g_k^{-1}g_{k-1},I_c)$,
\begin{equation}\label{eq:extractcenter}
W_\xi(g)=\eta_\rho(g_k)\,W_\xi\big(\diag(g_k^{-1}g_1,\dots,g_k^{-1}g_{k-1},I_c)\big).
\end{equation}
Clearly, $\eta_\rho(g)$ depends only on $\det g$,
and $\abs{\eta_\rho(\kappa)}=1$ for $\kappa\in K_c$. 

\begin{remark}\label{rem:genericchar}
The group $M$ acts on the set of characters of $U$ by conjugation, and a character of $U$ is called generic if it belongs to the $M$-orbit of $\psi_k$.
Let $\psi_U={}^{m^{-1}}\psi_k$ be a generic character, where $m\in M$ and ${}^{m^{-1}}\psi_k(u)=\psi_k(mum^{-1})$. Then $\lambda\mapsto\lambda\circ\rho(m)$ is an isomorphism
of $\Hom_U(\rho,\psi_k)$ onto $\Hom_U(\rho,\psi_U)$, so that $\dim\Hom_U(\rho,\psi_U)=1$; in particular Definition~\ref{def:kc} does not depend on the choice of the generic character.
Note that the stabilizer of $\psi_U$ in $M$ is $m^{-1}\GL_c^{\triangle}m$, and $\lambda\circ\rho(m)$ transforms on the left under $m^{-1}g^{\triangle}m$ by $\eta_\rho(g)$.
The functions $g\mapsto\lambda(\rho(mg)\xi)=W_\xi(mg)$, $\xi\in\cV_\rho$, form the $(k,c)$ model of $\rho$ with respect to $\psi_U$, so that
every bound proved below for the functions $W_\xi$ transfers to this model by the substitution $g\mapsto mg$ (note that $\norm{mg}\le\norm m\norm g$ by property~\ref{it:h2} of \S~\ref{ssec:height}), with the gauge estimate
expressed with respect to the stabilizer $m^{-1}\GL_c^{\triangle}m$. For example, if we replace the character $\psi$ of $F$ with $\psi'(x)=\psi(\varepsilon x)$, $\varepsilon\in F^\times$, then the character $\psi'_k$
defined by \eqref{eq:psikdef} with $\psi'$ in place of $\psi$ is ${}^{t_\varepsilon^{-1}}\psi_k$ with
\begin{equation*}
t_\varepsilon=\diag(\varepsilon^{k-1}I_c,\varepsilon^{k-2}I_c,\dots,\varepsilon I_c,I_c)\in T.
\end{equation*}
Since the blocks of $t_\varepsilon$ are scalar, $t_\varepsilon$ commutes with $\GL_c^{\triangle}$; hence in this case the stabilizer and the character $\eta_\rho$ are unchanged.
\end{remark}

\begin{remark}\label{rem:c1}
For $c=1$ the character $\psi_k$ is the standard generic character of $N$, and a
$(k,1)$ representation is precisely a representation from $\Rep(G)$ with a unique
Whittaker model (note that the orbit $(k)$ is the unique maximal nilpotent orbit).
\end{remark}

\begin{remark}\label{rem:rhoc}
The $(k,c)$ representations which occur in the generalized doubling method are constructed as follows
(\cite[\S2.2]{CFGK23}, \cite[\S1.4]{GK23}). For a generic $\tau\in\Irr(\GL_k)$ we define a representation $\rho_c(\tau)\in\Rep(G)$.
First assume that $\tau$ is unitary. Then $\rho_c(\tau)$ is the associated generalized Speh representation, namely
the unique irreducible subrepresentation of
\begin{equation*}
\tau|\det|^{(1-c)/2}\times\tau|\det|^{(3-c)/2}\times\dots\times\tau|\det|^{(c-1)/2}.
\end{equation*}
In general write $\tau=\tau_1|\det|^{a_1}\times\dots\times\tau_d|\det|^{a_d}$
with $\tau_i\in\Irr(\GL_{k_i})$ tempered (in particular, unitary) for each $i$, $k_1+\dots+k_d=k$ and
$a_1>\dots>a_d$. Then we define
\begin{equation}\label{eq:rho}
\rho_c(\tau)=\rho_c(\tau_1)|\det|^{a_1}\times\dots\times\rho_c(\tau_d)|\det|^{a_d}.
\end{equation}
In all cases $\rho_c(\tau)$ is a $(k,c)$ representation by \cite[Theorem~4]{CFGK23}; in the case of \eqref{eq:rho} it may be reducible. By \cite[Lemma~12]{CFGK23},
\begin{equation}\label{eq:etarhoc}
\eta_{\rho_c(\tau)}(g)=\tau\big((\det g)I_k\big)=\omega_\tau\big((\det g)I_k\big),\qquad g\in\GL_c
\end{equation}
($\omega_\tau$ is the central character of $\tau$).
\end{remark}

\subsection{The Zelevinsky data}\label{ssec:zeldata}
The Zelevinsky models are defined by a character of $N$. We define them in detail in \S~\ref{sec:zelevinsky}. Here we introduce the character $\Psn$ of $N$ relevant for our work, and relate the pairs $(N,\Psn)$ and $(U,\psi_k)$.

Let $Q=P_{(k^c)}$, $L=M_{(k^c)}$ and $V=V_{(k^c)}$, so that $Q=L\ltimes V$. Put $N_L=N\cap L=\diag(N_k,\ldots,N_k)$ ($N_k=V_{(1^k)}$); then $N=N_L\ltimes V$. Define a character $\Psn$ of $N$ by
\begin{equation*}
\Psn(u)=\psi\Big(\sum_{1\le i<n,\ k\,\nmid \,i}u_{i,i+1}\Big),\qquad u\in N .
\end{equation*}
It is trivial on $V$ and its restriction $\psi_L$ to $N_L$ is the standard generic
character of $N_L$, the product of the standard generic characters of the $c$ factors of $N_k$. For any $\pi\in\Rep(G)$,
\begin{equation}\label{eq:homfactor}
\Hom_N(\pi,\Psn)\ =\ \Hom_{N_L}\big(J_{V}(\pi),\psi_L\big).
\end{equation}

The $(k,c)$ data and the Zelevinsky data are related by a permutation. Let $w_0\in K$ be
the permutation matrix with
\begin{equation*}
w_0\,e_{(i-1)c+j}^{\,t}\ =\ e_{(j-1)k+i}^{\,t},\qquad 1\le i\le k,\ 1\le j\le c,
\end{equation*}
where $e^{\,t}$ denotes the column vector. Conjugation by $w_0$ moves the entry in
position $((i-1)c+j,(i'-1)c+j')$ to the position $((j-1)k+i,(j'-1)k+i')$. Consequently
$w_0Mw_0^{-1}$ is the subgroup of $L$-block matrices in which the $i$-th factor
$g_i\in\GL_c$ acts on the coordinates $(j-1)k+i$, $j=1,\dots,c$, and
\begin{equation*}
U':=w_0Uw_0^{-1}=\{u'\in G:\ u'_{jj}\in N_k,\ u'_{jj'}\ \text{strictly upper triangular for }j\ne j'\},
\end{equation*}
where $u'_{jj'}\in\Mat_k$ denotes the $(j,j')$ block of $u'$.
Note that $U'$ is contained in the subgroup
\begin{equation}\label{eq:D}
D:=\{g\in G:\ e_{jk}\,g=e_{jk}\ \text{for }j=1,\dots,c\}
\end{equation}
of matrices whose rows $k,2k,\dots,ck$ are the corresponding standard vectors: the last
row of the block $u'_{jj}$ is $(0,\dots,0,1)$ and the last row of $u'_{jj'}$, $j\ne j'$, is
zero.

Define
$\Psu(u')=\psi(\sum_{k\,\nmid \,i}u'_{i,i+1})$ for $u'\in U'$; this is a character of $U'$ and
\begin{equation}\label{eq:psik}
\psi_k(v)=\Psu(w_0vw_0^{-1}),\qquad v\in U,
\end{equation}
because the $(l,l+1)$ entries of $w_0vw_0^{-1}$ with $l=(j-1)k+i$, $1\leq i<k$, are the diagonal
entries $(v_{i,i+1})_{jj}$, whose sum over $j$ is $\tr v_{i,i+1}$.

The pair $(U',\Psu)$ is
the ``Shalika'' datum of \cite[\S3.1]{LM20} (for $(m,n)=(c,k)$), and $(N,\Psn)$ is the
Zelevinsky datum there.

Finally,
\begin{equation*}
U'\cap\bN=\{\bu\in\bN:\ \bu_{jj}=I_k,\ \bu_{jj'}\ \text{strictly upper triangular for }j>j',\ \bu_{jj'}=0\ \text{for }j<j'\},
\end{equation*}
a unipotent group, isomorphic as a variety to the affine space of dimension
\begin{equation}\label{eq:delta}
\fd_{U'\cap\bN}:=\binom c2\binom k2,
\end{equation}
with Haar measure the product of the Haar measures of $F$ on each entry. For $R\in q^{\bZ}$
(an integer power of $q$), the set of $\bu\in U'\cap\bN$ all of whose entries have absolute
value at most $R$ has volume $q^{e_\psi\fd_{U'\cap\bN}/2}R^{\fd_{U'\cap\bN}}$.

\section{Moderate growth of Whittaker functions}\label{sec:whittaker}

In this section we obtain the moderate growth of Whittaker functions on $L$.
The height $\norm\cdot$ of \eqref{eq:height} restricts to $L\subset G$. Let $B_L=L\cap B$ and
$K_L=L\cap K$. Let $\Delta$ denote the set of $c(k-1)$ simple roots of $T$ in $L$. We use the set $\Delta$ to parametrize the
direct product $(F^\times)^{|\Delta|}$. Define the map $\ell:T\to(F^\times)^{|\Delta|}$ by $\ell(t)_\alpha=\alpha(t)$. Since $Z_L$ is the kernel of $\ell$,
we have the split exact sequence
\[
1\to Z_L\to T\xrightarrow{\ell}(F^\times)^{|\Delta|}\to1.
\]
Let $\theta:(F^\times)^{|\Delta|}\to T$ be a splitting of $\ell$, that is, a homomorphism with $\ell\circ\theta=\mathrm{id}$.
Thus we can write $t=z\theta(a)$ with $z\in Z_L$ and $\ell\circ\theta(a)=a$.
Let $\nu\colon Z_L\to(F^\times)^c$ be the isomorphism $\nu(\diag(z_1I_k,\dots,z_cI_k))=(z_1,\dots,z_c)$, and put
$\val_\theta(t):=(\val(\nu(z)),\val(\ell(t)))\in \bZ^{kc}$; note that $z=t\,\theta(\ell(t))^{-1}$, so that $\val_\theta(t)$ depends on $\theta$.

Let $P^{\alpha}=M^\alpha\ltimes U^\alpha$ be the standard maximal parabolic subgroup of $L$ whose Levi part $M^\alpha$ is associated with $\Delta-\{\alpha\}$.
Define $H^\alpha$ to be the image under $\theta$ of the $\alpha$-th coordinate subgroup $F^\times$. 
Then $H^\alpha$ is a subgroup of the center of $M^{\alpha}$ and $\theta((F^\times)^{|\Delta|})=\prod_{\alpha}H^\alpha$.

For simplicity, we fix $\theta$ by requiring that $\theta((F^\times)^{|\Delta|})$ be the subgroup of elements of $T$ whose last coordinate in each $\GL_k$-block is $1$.
Then for $t=\diag(t_1,\dots,t_{kc})\in T$, the element $z=t\,\theta(\ell(t))^{-1}$ is $\diag(t_kI_k,t_{2k}I_k,\dots,t_{kc}I_k)$; in particular $\norm z\le\norm t$.

We use the following asymptotic expansion for Whittaker functions on $T$, which in our setting is straightforward to derive from
Lapid and Mao \cite{LM09} (see Remark~\ref{rem:LM09canbeUsed} below); for earlier results on asymptotic expansion see \cite{JPSS79,JPSS83,CS80}. The setup of \cite{LM09} was more general in that it applied to
all split connected reductive groups. We provide more details in our special case, using results of Cogdell and Matringe \cite{CM15}, to make this argument available to central extensions of $G$ (see also \cite[Remark~3.9]{LM09}).
\begin{lemma}[asymptotic expansion]\label{lem:expansion}
Let $\cW\subset\Ind_{N_L}^{L}\psi_L$ be a subrepresentation of finite length. There
is a finite set $\Gamma$ of finite functions on $T$ depending only on $\cW$, such that for every $W\in\cW$, there are $\phi_\gamma\in\mathcal{S}(F^{|\Delta|})$ for each
$\gamma\in\Gamma$, such that for all $t\in T$,
\begin{equation}\label{eq:expansion}
W(t)=\sum_{\gamma\in\Gamma}\gamma(t)\,\phi_\gamma(\ell(t)).
\end{equation}
Each $\gamma\in\Gamma$ is of the following form:
\begin{equation*}
\gamma(t)=\gamma(z\theta(a))=\omega(z)\mathcal{Q}(\val_\theta(t))\prod_{\alpha\in\Delta}\chi_{\alpha}(a_\alpha),
\end{equation*}
where
\begin{enumerate}[leftmargin=*,label=(\roman*),ref=(\roman*)]
\item\label{it:expi} $\omega$ is a character of $Z_L$ (not necessarily unitary) which is the central character of
an irreducible subquotient of $\cW$;
\item\label{it:expii} the degree of $\mathcal{Q}$ in the first $c$ variables (in which we substitute $\val(\nu(z))$) is less than the length of $\cW$; 
in each of the last $c(k-1)$ variables we substitute $\val(\ell(t)_{\alpha})$ for some $\alpha\in\Delta$, and the degree of that variable 
is less than the length of $J_{U^\alpha}(\cW)$;

\item\label{it:expiii} $\chi_\alpha$ is the restriction to $H^\alpha$ of the central character of an irreducible subquotient of $J_{U^\alpha}(\cW)$, or $\chi_\alpha$ is trivial if $J_{U^\alpha}(\cW)=0$.
\end{enumerate}
\end{lemma}

\begin{proof}
Let $W\in\cW$. The center $Z_L$ preserves the finite-dimensional space of vectors of
$\cW$ fixed by a compact open subgroup of $L$ fixing $W$, so the $Z_L$-span of $W$ is
finite-dimensional, and the Jordan decomposition of the commuting operators $W'\mapsto z\cdot W'$,
$z\in Z_L$, on it gives $z\cdot W=\sum_{l=1}^{r}c_l(z)W_l$ with $W_1,\dots,W_r\in\cW$ and
$c_l(z)=\sum_\omega\omega(z)\mathcal Q_{l,\omega}(\val(\nu(z)))$, where $\omega$ runs over the
generalized eigencharacters of $Z_L$ on $\cW$, which are the central characters of its
irreducible subquotients (finitely many, since $\cW$ has finite length), and $\deg \mathcal Q_{l,\omega}$ is less than the multiplicity of
$\omega$. Hence $W(z\theta(a))=\sum_lc_l(z)W_l(\theta(a))$, and it remains to expand the
functions $f_W:(F^\times)^{|\Delta|}\to\bC$ defined by $f_W(a)=W(\theta(a))$, $W\in\cW$.

We proceed by applying \cite[Lemma~3.4]{CM15} to the space $\mathcal{F}=\{f_{W}:W\in\cW\}$ of
functions on $(F^\times)^{|\Delta|}=\prod_{\alpha\in\Delta}H^\alpha$. Its hypotheses are verified as in
\cite[\S3]{CM15}: $\mathcal{F}$ is stable under translations, since $f_{W}(aa')=f_{\theta(a')\cdot W}(a)$ because $\theta$ is a homomorphism;
the functions $f_W$ are locally constant; and for each $\alpha$, let $\mathcal{F}^{\alpha}\subset\mathcal{F}$ be the subspace of functions $f_W$ for which there is a constant
$c_W>0$ such that $f_W$ vanishes on
\[
\{a:|a_{\alpha'}|\le1\ \forall\alpha'\ne\alpha,\, |a_\alpha|<c_W\}.
\]
Then, $\mathcal{F}/\mathcal{F}^\alpha$ is a quotient of $J_{U^\alpha}(\cW)$ as an $H^\alpha$-module, because for
$u\in U^\alpha$ and for $a$ such that $|a_{\alpha'}|\leq1$ for all $\alpha'\ne\alpha$,
\[
f_{u\cdot W-W}(a)=\big(\psi_L(\theta(a)u\theta(a)^{-1})-1\big)W(\theta(a))=\big(\psi(a_\alpha u_{\alpha})-1\big)W(\theta(a)),
\]
which vanishes for $|a_\alpha|$ small; here $u_{\alpha}$ is the projection of $u$ into the root subgroup $N^\alpha$ of $N$ associated with $\alpha$. The
representation $J_{U^\alpha}(\cW)$ of $M^\alpha$ has finite length and $H^\alpha$ acts on
each of its irreducible subquotients by a character, so $\mathcal{F}/\mathcal{F}^\alpha$ has a finite
filtration by $H^\alpha$-submodules, of length at most the length of $J_{U^\alpha}(\cW)$,
on whose successive quotients $H^\alpha$ acts by the characters described in part \ref{it:expiii} of the
statement.

Then \cite[Lemma~3.4]{CM15} gives, on the region $|a_\alpha|\le1$ for all
$\alpha\in\Delta$, an expansion \[f_W(a)=\sum_{\gamma\in\Gamma}\gamma(a)\phi_\gamma(a).\]
Here $\phi_\gamma\in\mathcal{S}(F^{|\Delta|})$ vanishes outside $\Oo^{|\Delta|}$; and $\Gamma$ is a finite set of finite functions $\gamma$ on $(F^\times)^{|\Delta|}$. Moreover, each $\gamma$ is of the form
\[
\gamma(a)=\prod_{\alpha\in\Delta}\chi_\alpha(a_\alpha)\val(a_\alpha)^{m_\alpha},
\]
with $\chi_\alpha$ among the characters of the filtration and $m_\alpha\geq0$ integers; in addition, $m_\alpha$ is less than the length of
the filtration, since for every $h\in H^\alpha$ the finite product of $(h-\chi(h))$ over the characters $\chi$ of the filtration, each taken with its multiplicity, 
annihilates $\mathcal{F}/\mathcal{F}^\alpha$.

Next we have $W(\theta(a))=0$ unless $|a_\alpha|\le q^{r'}$ for all $\alpha$, where $r\ge1$ is such that $W$ is fixed by
$N^\alpha\cap K(r)$ for all $\alpha$, and $r':=r-e_\psi$: indeed, for $x\in\varpi^r\Oo$, the element $u\in N^\alpha\cap K(r)$ with $u_\alpha=x$ satisfies
$\theta(a)u\theta(a)^{-1}\in N^{\alpha}$ with coordinate $a_{\alpha}x$, hence
$W(\theta(a))=W(\theta(a)u)=\psi(a_\alpha x)W(\theta(a))$, and if $|a_\alpha|>q^{r'}$ we can choose $x$ with $\psi(a_\alpha x)\ne1$, by \eqref{eq:conductor}.

Finally, for $y\in F^\times$ let $\underline{y}\in(F^\times)^{|\Delta|}$ be such that $\underline{y}_\alpha=y$ for all $\alpha\in\Delta$.
Denote $W'=\theta(\underline{\varpi}^{-r'})\cdot W\in\cW$. Since $\theta$ is a homomorphism,
$f_{W'}(a)=W(\theta(a)\theta(\underline{\varpi}^{-r'}))=f_W(\underline{\varpi}^{-r'}a)$, so that $f_{W'}(a)=0$ if $|a_{\alpha}|>1$ for some $\alpha$.
Applying the above expansion to $f_{W'}$, we obtain for $b\in(F^\times)^{|\Delta|}$ such that $|b_{\alpha}|\leq q^{r'}$ for all $\alpha$,
\[
f_W(b)=f_{W'}(\underline{\varpi}^{r'}b)=\sum_{\gamma\in\Gamma}\gamma(\underline{\varpi}^{r'}b)\phi_\gamma(\underline{\varpi}^{r'}b),\qquad
\gamma(\underline{\varpi}^{r'}b)=\prod_{\alpha\in\Delta}\chi_\alpha(\varpi^{r'})\chi_\alpha(b_\alpha)(r'+\val(b_\alpha))^{m_\alpha}.
\]
The functions $b\mapsto \phi_\gamma(\underline{\varpi}^{r'}b)$ are Schwartz--Bruhat functions on $F^{|\Delta|}$, supported in
$\varpi^{-r'}\Oo^{|\Delta|}$, and $\gamma(\underline{\varpi}^{r'}b)$ expands into a linear combination, with coefficients depending on $r'$, of the finite functions
\[
\prod_{\alpha\in\Delta}\chi_\alpha(b_\alpha)\val(b_\alpha)^{m'_\alpha}
\]
with $0\leq m'_{\alpha}\leq m_{\alpha}$, which are independent of $r'$ and which we may assume belong to $\Gamma$. Both sides of the resulting expansion of $f_W(b)$ vanish when $|b_{\alpha}|>q^{r'}$ for some $\alpha$, the left side by the vanishing established above and the right side because of the support of the rescaled Schwartz--Bruhat functions. Hence the expansion of $f_W$ holds on all of $(F^\times)^{|\Delta|}$, with the finite set $\Gamma$ depending only on $\cW$. This completes the proof.
\end{proof}

\begin{corollary}\label{cor:whittakermg}
Let $\cW$ and $\Gamma$ be as in Lemma~\ref{lem:expansion}, and put
\begin{equation}\label{eq:DW}
D_\cW:=\max_{\gamma\in\Gamma}\Big(\abs{\Re \omega}+2\sum_{\alpha\in\Delta}\abs{\Re \chi_{\alpha}}+\deg \mathcal{Q}\Big),
\end{equation}
where $\abs{\Re \omega}=\sum_j\abs{\Re \omega_j}$ for
$\abs{\omega(z)}=\prod_j\abs{z_j}^{\Re \omega_j}$, and
$\abs{\chi_{\alpha}(x)}=\abs x^{\Re \chi_{\alpha}}$. For every $W\in\cW$ there is
$C_W\ge0$ with
\[
\abs{W(b)}\ \le\ C_W\,\norm b^{D_\cW}\qquad\forall b\in L .
\]
\end{corollary}

\begin{proof}
Let $t=z\theta(a)\in T$. By property~\ref{it:h4} of \S~\ref{ssec:height}, $\norm z\le\norm t$,
$\abs{a_\alpha}^{\pm1}\le\norm t^{2}$ and $\abs{\val_\theta(t)_i}\le2\log_q\norm t\le2\norm t$ for every coordinate $i$ of $\val_\theta(t)$,
hence each $\gamma\in\Gamma$ satisfies
\[
|\gamma(t)|\le C_\gamma\norm t^{\abs{\Re \omega}+2\sum\abs{\Re \chi_{\alpha}}+\deg \mathcal Q},
\]
and \eqref{eq:expansion} gives $\abs{W(t)}\le C\norm t^{D_\cW}$, with $C$ depending on
$W$ through the bounded functions $\phi_\gamma$.

For $b\in L$ write $b=ut\kappa$ with
$u\in N_L$, $t\in T$ and $\kappa\in K_L$. Then $\abs{W(b)}=|(\kappa\cdot W)(t)|$. Since $W$ is smooth,
$\kappa\cdot W$ ranges over a finite subset of $\cW$ as $\kappa$ ranges over $K_L$. In addition $\norm t\le\norm{ut}=\norm{ut\kappa}=\norm b$ by
properties~\ref{it:h3} and \ref{it:h1} of \S~\ref{ssec:height}. The result follows.
\end{proof}

\begin{remark}\label{rem:LM09canbeUsed}
The expansion \eqref{eq:expansion}, together with parts \ref{it:expi} and \ref{it:expiii} of the lemma, follows from \cite[Theorem~3.1]{LM09} applied to $\cW$. 
Note that Lapid and Mao \cite{LM09} did not assume uniqueness of the Whittaker model or irreducibility.
A slightly weaker form of part \ref{it:expii}, which suffices for this work, is that 
$\deg\mathcal Q$ is bounded by the maximum over the lengths of the Jacquet modules of $\cW$ along the standard parabolic subgroups of $L$. This version follows from the proof in \cite{LM09} because the polynomials arise there from the nilpotent part of the action of
the center of a standard Levi subgroup on the generalized eigenspaces
of the Jacquet modules, and the index of nilpotency of this action is at most the
multiplicity of the eigencharacter. 
\end{remark}

\section{The transition map of Lapid and Mao}\label{sec:transition}

We recall the construction of \cite[\S3]{LM20}, in the case $(m,n)=(c,k)$ of that paper,
keeping only what we use. Let $\Psi(g)=\psi(\sum_{k\,\nmid\, i}g_{i,i+1})$ for $g\in G$. For
$i=0,\dots,(c-1)k$, \cite[\S3.1]{LM20} defined a cocharacter $\phi_i$ of $T$ ``of type
$(c,k)$'' and the unipotent radical $U_i$ of the semistandard parabolic subgroup
$P_i=M_iU_i$ attached to it, in such a way that $U_0=N$, $U_{(c-1)k}=U'$, and the
restriction $\Psi_{U_i}$ of $\Psi$ to $U_i$ is a character for every $i$; thus
$\Psi_{U_0}=\Psn$ and $\Psi_{U_{(c-1)k}}=\Psu$. By \cite[Lemma~3.3]{LM20}, for
$0\le i<(c-1)k$ the groups $U_i$ and $U_{i+1}$ normalize each other, the quotients
$A_i:=U_{i+1}/(U_i\cap U_{i+1})$ and $B_i:=U_i/(U_i\cap U_{i+1})$ are abelian, the
commutator $[u',u]=u'uu'^{-1}u^{-1}$ of $u'\in U_{i+1}$ and $u\in U_i$ lies in $U_i\cap U_{i+1}$, and the
map
\begin{equation*}
c_i\colon B_i\to\PD(A_i),\qquad c_i(u)(u')=\Psi([u',u]),
\end{equation*}
into the Pontryagin dual of $A_i$ is a well-defined surjective homomorphism. (The other
convention for the commutator, $u'^{-1}u^{-1}u'u$, gives the element $u'^{-1}u^{-1}[u',u]uu'$, which
differs from $[u',u]$ by conjugations by elements of $U_i$ and of $U_{i+1}$; since the
restriction of $\Psi$ to $U_i\cap U_{i+1}$ is the restriction of the character $\Psi_{U_i}$
of $U_i$ and of the character $\Psi_{U_{i+1}}$ of $U_{i+1}$, it is invariant under both
conjugations, and $c_i$ does not depend on the convention.)

Let $M^\Psi_{i}$ be the stabilizer of $\Psi_{U_{i}}$ in $M_{i}$ and $P^\Psi_{i}=M^\Psi_{i}\ltimes U_{i}$. By \cite[(6)]{LM20},
\[U_i\cap P^\Psi_{i+1}=(U_i\cap M^\Psi_{i+1})\ltimes(U_i\cap U_{i+1}).\]
By \cite[Lemma~3.3(ii)]{LM20}, the kernel of $c_i$ is the image of
$U_i\cap P^\Psi_{i+1}$ in $B_i$.

For $W\in\Ind_{U_i}^G\Psi_{U_i}$ and $g\in G$, the
function $\varphi_g(u'):=\Psi_{U_{i+1}}(u')^{-1}W(u'g)$ on $U_{i+1}$ factors through $A_i$ and is
compactly supported and locally constant there: this is proved in
\cite[proof of Proposition~3.5]{LM20} for an arbitrary $W$ in the induced space, using only
the smoothness of $W$ and the surjectivity of $c_i$ (and the open mapping theorem). The map
\[\cT_i\colon\Ind_{U_i}^G\Psi_{U_i}\to\Ind_{U_{i+1}}^G\Psi_{U_{i+1}}\]
given by
\[\cT_iW(g)=\int_{A_i}\varphi_g(u')\,du'=\int_{U_i\cap U_{i+1}\backslash U_{i+1}}\!\!\!W(u'g)\,\Psi_{U_{i+1}}(u')^{-1}du'\]
is well defined and $G$-equivariant, and by \cite[(10)]{LM20},
\begin{equation}\label{eq:fourier}
\Psi_{U_i}(u)^{-1}\,\cT_iW(ug)\ =\ \widehat{\varphi_g}\big(c_i(u)\big),\qquad u\in U_i,
\end{equation}
where $\widehat{\varphi}(\chi)=\int_{A_i}\varphi(x)\chi(x)\,dx$ is the Fourier transform of a function $\varphi$ on
$A_i$: indeed, for any $u''\in U_{i+1}$ we have
\[
\Psi_{U_{i+1}}(uu''u^{-1})=\Psi([u,u''])\Psi_{U_{i+1}}(u'')=\Psi([u'',u])^{-1}\Psi_{U_{i+1}}(u'');
\]
writing $u'u=u(u^{-1}u'u)$ we have $u''=u^{-1}u'u\in U_{i+1}$, and using the invariance of the Haar measure of $A_i$ under conjugation by $u$,
\[
\cT_iW(ug)=\int_{A_i}W(uu''g)\,\Psi_{U_{i+1}}(uu''u^{-1})^{-1}du''
=\Psi_{U_i}(u)\int_{A_i}\varphi_g(u'')\,c_i(u)(u'')\,du'' .
\]
Put
\begin{equation}\label{eq:T}
\cT:=\cT_{(c-1)k-1}\circ\dots\circ \cT_0\colon\ \Ind_N^G\Psn\to\Ind_{U'}^G\Psu .
\end{equation}

\begin{proposition}[{\cite[Lemma~3.8]{LM20}}]\label{prop:LM}
For every $W\in\Ind_N^G\Psn$ and $g\in G$ the function $\bu\mapsto W(\bu g)$ on
$U'\cap\bN$ is compactly supported, and
\begin{equation}\label{eq:Tformula}
\cT W(g)=\int_{U'\cap\bN}W(\bu g)\,d\bu .
\end{equation}
\end{proposition}

Proposition~3.5 of \cite{LM20} asserted that $\cT_i$ is an isomorphism between the models of an
irreducible representation, and constructed a left inverse of $\cT_i$ by integration over
$U_i\cap P^\Psi_{i+1}\backslash U_i$, which is Fourier inversion on
$\PD(A_i)\cong B_i/\ker c_i$; the following lemma proves the injectivity of $\cT$ on $\Ind_N^G\Psn$, which is all we
need.

\begin{lemma}\label{lem:injective}
Each $\cT_i$, and hence $\cT$, is injective.
\end{lemma}

\begin{proof}
Let $W\in\Ind_{U_i}^G\Psi_{U_i}$ with $\cT_iW=0$. Fix $g\in G$. By \eqref{eq:fourier},
$\widehat{\varphi_g}$ vanishes on $c_i(B_i)$, which is all of $\PD(A_i)$ by the surjectivity of $c_i$. Since $\varphi_g$ is a compactly supported locally constant function on the
abelian group $A_i$, Fourier inversion gives $\varphi_g=0$, and in particular
$W(g)=\varphi_g(I_n)=0$.
\end{proof}

\begin{remark}\label{rem:Tprevious}
The maps $\cT_i$ were also considered in \cite[\S2.4]{CFK22}, as noted in \cite[\S3.2]{LM20}.
The operator $\cT$, from the Zelevinsky model to the $(k,c)$ model,
can be obtained by successively applying \cite[Lemma~9]{CFGK23} along the compositions $(1,c-1)$, $(1,1,c-2)$, \dots, $(1^c)$ of $c$.
This was described in \cite{CFK22,CFGK23} when $\rho=\rho_c(\tau)$ with $\tau$ essentially unitary.
\end{remark}

\section{The Zelevinsky functional}\label{sec:zelevinsky}

Throughout this section $\rho$ is a $(k,c)$ representation of $G$, with space $\cV_\rho$
and $(k,c)$ functional $\lambda$. We show that $\dim\Hom_N(\rho,\Psn)=1$, that the Zelevinsky model
functions have moderate growth, and that the transition map of \S~\ref{sec:transition}
carries the Zelevinsky model onto the $(k,c)$ model.

Recall that by M\oe glin and Waldspurger \cite[II.2]{MW87} (for odd
residual characteristic; Varma \cite{Var14} removed this restriction), any $\pi\in\Irr(G)$ has a unique maximal orbit in its wave front set, and by
\cite[Theorems~E and F]{GGS17} this orbit corresponds to the partition $\beta=(\beta_1,\beta_2,\dots,\beta_l)$ in
which $\beta_1$ is the order of the highest derivative of $\pi$, $\beta_2$ that of
the highest derivative of $\pi^{(\beta_1)}$, and so on, in the sense of Bernstein and
Zelevinsky \cite{BZ77}; in terms of the Zelevinsky classification, if $\pi$ corresponds to
the multisegment $(\Delta_1,\dots,\Delta_d)$ with $\Delta_i$ a segment of length $m_i$ of
irreducible supercuspidal representations of $\GL_{r_i}$, this is the orbit induced from the orbits
$(r_i^{m_i})$ of $\GL_{m_ir_i}$, $1\le i\le d$ \cite[II.2, (4)]{MW87}. Zelevinsky
\cite[8.3, Corollary]{Zel80} proved that $\pi$ embeds, in a unique way up to a scalar,
into $\Ind_N^G\psi_\beta$ where $\psi_\beta$ is the character of $N$ defined by
\[
\psi_\beta(u)=\psi\Big(\sum_{\{1\le i<n:\ i\ne n-\sum_{j'=1}^{j}\beta_{j'},\ 1\leq j<l\}}u_{i,i+1}\Big).
\]
The image of this embedding is what we call the Zelevinsky model of $\pi$.

\begin{lemma}\label{lem:zelexists}
$\Hom_N(\rho,\Psn)\ne0$.
\end{lemma}

\begin{proof}
By \cite[Theorem~E]{GGS17}, for $\pi\in\Rep(G)$ and a nilpotent orbit $\mathcal O$, the generalized Whittaker model of $\pi$ with respect to $\mathcal O$ is nonzero if and only if
$\mathcal O$ is contained in the closure of the union of the orbits of $\WF(\pi)$. By \cite[Lemma~2.8]{GS19}, this model is the dual of a twisted Jacquet module of $\pi$ with respect to a unipotent subgroup and a character of it, both depending only on $\mathcal O$; hence the functor taking $\pi$ to this model is exact.
Since $\WF(\rho)=\{(k^c)\}$, the model of $\rho$ with respect to $(k^c)$ is nonzero, and the model with respect to any orbit $\mathcal O>(k^c)$ vanishes. By exactness, some irreducible
subquotient $\pi$ of $\rho$ has a nonzero model with respect to $(k^c)$, and all the models of $\pi$ with respect to orbits $\mathcal O>(k^c)$ vanish. Hence $\WF(\pi)=\{(k^c)\}$.
Thus, in the above notation, $\beta=(k,\dots,k)$ for $\pi$, so that $l=c$ and $\{n-\sum_{j'=1}^{j}\beta_{j'}\}_{1\leq j<c}=\{jk\}_{1\le j<c}$, whence
$\psi_{(k^c)}=\Psn$ and $\dim\Hom_N(\pi,\Psn)=1$, equivalently $\dim J_{N,\Psn}(\pi)=1$. Since the functor $J_{N,\Psn}$ is exact,
we deduce that $J_{N,\Psn}(\rho)\ne0$ and thus $\Hom_N(\rho,\Psn)\ne0$.
\end{proof}

\begin{remark}\label{rem:zelexistsrhoc}
When $\rho=\rho_c(\tau)$ (see Remark~\ref{rem:rhoc}), Lemma~\ref{lem:zelexists} can also be seen through Frobenius reciprocity when $\tau$ is unitary, and in general through the 
geometric lemma of Bernstein and Zelevinsky \cite{BZ77} applied to \eqref{eq:rho} and the parabolic subgroup $Q$. In both cases one of the
subquotients of $J_V(\rho_c(\tau))$ is induced from the tensor product over $i$ of the 
quotients $\tau_i|\det|^{(1-c)/2}\otimes\dots\otimes\tau_i|\det|^{(c-1)/2}$ of the Jacquet
modules of the $\rho_c(\tau_i)$, hence admits a Whittaker functional by the heredity theorem of Rodier \cite{Rod73}.
\end{remark}

Before proving the uniqueness, we note that the space $\Hom_N(\rho,\Psn)$ is
finite-dimensional a priori. Indeed, $J_{N,\Psn}=J_{N_L,\psi_L}\circ J_V$, $J_V(\rho)\in\Rep(L)$, and for 
$\sigma\in\Irr(L)$, $\dim J_{N_L,\psi_L}(\sigma)\leq1$. Since $J_{N_L,\psi_L}$ is exact, $\dim\Hom_N(\rho,\Psn)$ is the number 
of generic irreducible subquotients of $J_V(\rho)$ counted with multiplicity. The following lemma shows that for a $(k,c)$ representation this dimension is exactly one.

\begin{lemma}\label{lem:zele1dim}
For $\lambda_{\mathcal Z}\in\Hom_N(\rho,\Psn)$ and $\xi\in\cV_\rho$ let $W^{\mathcal Z}_\xi(g)=\lambda_{\mathcal Z}(\rho(g)\xi)$, and put
$\Lambda_{\lambda_{\mathcal Z}}(\xi)=(\cT W^{\mathcal Z}_\xi)(w_0)$, with $\cT$ the transition map \eqref{eq:T}. Then
$\lambda_{\mathcal Z}\mapsto\Lambda_{\lambda_{\mathcal Z}}$ is an injective linear map from $\Hom_N(\rho,\Psn)$ to $\Hom_U(\rho,\psi_k)$.
In particular, $\dim\Hom_N(\rho,\Psn)=1$.
\end{lemma}
\begin{proof}
For $c=1$, we have $N=U$, $\Psi_N=\psi_k$ and the assertions follow immediately from Definition~\ref{def:kc}. We proceed for all $c\geq1$.
The map $\xi\mapsto W^{\mathcal Z}_\xi$ is a $G$-equivariant map $\rho\to\Ind_N^G\Psn$, which vanishes if and only if $\lambda_{\mathcal Z}=0$.
Put $\Lambda(\xi)=(\cT W^{\mathcal Z}_\xi)(I_n)$. Since $\cT$ and $\xi\mapsto W^{\mathcal Z}_\xi$ are
$G$-equivariant, $\Lambda(\rho(g)\xi)=(\cT W^{\mathcal Z}_\xi)(g)$, and since $\cT W^{\mathcal Z}_\xi\in\Ind_{U'}^G\Psu$,
$\Lambda\in\Hom_{U'}(\rho,\Psu)$. By \eqref{eq:psik}, $\Lambda_{\lambda_{\mathcal Z}}=\Lambda\circ\rho(w_0)\in\Hom_U(\rho,\psi_k)$.
The map $\lambda_{\mathcal Z}\mapsto\Lambda_{\lambda_{\mathcal Z}}$ is linear, and it is injective: if $\Lambda_{\lambda_{\mathcal Z}}=0$, then $\Lambda=0$,
so that $\cT W^{\mathcal Z}_\xi=0$ for all $\xi$, hence $W^{\mathcal Z}_\xi=0$ for all $\xi$ by Lemma~\ref{lem:injective}, and $\lambda_{\mathcal Z}=0$.
Since $\Hom_U(\rho,\psi_k)$ is one-dimensional by Definition~\ref{def:kc}, we obtain $\dim\Hom_N(\rho,\Psn)\le1$, and equality holds by Lemma~\ref{lem:zelexists}.
\end{proof}

\begin{remark}\label{rem:GGS}
The upper bound $\dim\Hom_N(\rho,\Psn)\le1$ also follows from the comparison of degenerate Whittaker models of
Gomez, Gourevitch and Sahi \cite{GGS17}. By \cite[\S3.1]{LM20}, $J_{N,\Psn}(\rho)$ and
$J_{U',\Psu}(\rho)$ are degenerate Whittaker models in the sense of M\oe glin and Waldspurger
\cite[\S II.2]{MW87}. In the terminology of \cite[\S2.5]{GGS17}, $J_{N,\Psn}(\rho)$ is defined by
a Whittaker pair and $J_{U',\Psu}(\rho)$ is defined by a neutral Whittaker pair with the same nilpotent element, and
\cite[Theorem~A]{GGS17} provided a $G$-equivariant epimorphism from the neutral model
$\operatorname{ind}_{U'}^G\Psu$ onto $\operatorname{ind}_N^G\Psn$ (both pairs have
$\mathfrak g_1=0$ there, so that no oscillator representation intervenes). Applying
$\Hom_G(\,\cdot\,,\rho^{\vee})$ and Frobenius reciprocity for compact induction
(\cite[2.29]{BZ76}) gives an injection of $\Hom_N(\rho,\Psn^{-1})$ into
$\Hom_{U'}(\rho,\Psu^{-1})$, which is the bound (for $\psi^{-1}$, hence for $\psi$). For
$c>1$ the pair defining $J_{N,\Psn}(\rho)$ is not neutral, which is why the multiplicity one theorem of
\cite{MW87} for neutral pairs does not apply to $(N,\Psn)$ directly. The proof given above
replaces the epimorphism of \cite{GGS17} by the injectivity of $\cT$.
\end{remark}

\begin{remark}\label{rem:zelirred}
When $\rho$ is irreducible, Lemma~\ref{lem:zele1dim} follows directly from \cite[8.3]{Zel80} together with the identification of
$\beta$ with $(k^c)$ explained above.
\end{remark}

We fix a nonzero $\lambda_{\mathcal Z}\in\Hom_N(\rho,\Psn)$, and put
\begin{equation}\label{eq:WZ}
W^{\mathcal Z}_\xi(g):=\lambda_{\mathcal Z}(\rho(g)\xi),\qquad \xi\in\cV_\rho,\ g\in G .
\end{equation}
Then $\xi\mapsto W^{\mathcal Z}_\xi$ is a nonzero $G$-equivariant map from $\rho$ to
$\Ind_N^G\Psn$, and the definitions imply that
\begin{alignat}{2}\label{eq:WZequivN}
W^{\mathcal Z}_\xi(ug)&=\Psn(u)W^{\mathcal Z}_\xi(g),&\qquad& u\in N,\ g\in G,\\
\label{eq:WZequivG}
W^{\mathcal Z}_{\rho(h)\xi}(g)&=W^{\mathcal Z}_\xi(gh),&& g,h\in G .
\end{alignat}
We call the space $\{W^{\mathcal Z}_\xi:\xi\in\cV_\rho\}$ the Zelevinsky model of $\rho$; by
Lemma~\ref{lem:zele1dim} it does not depend on the choice of $\lambda_{\mathcal Z}$.

\begin{lemma}\label{lem:zelmg}
There is $D_{\mathcal Z}\ge0$, depending only on $\rho$, such that for every $\xi\in\cV_\rho$ there is
$C_{\mathcal Z}(\xi)\ge0$ with
\[
\abs{W^{\mathcal Z}_\xi(g)}\ \le\ C_{\mathcal Z}(\xi)\,\norm g^{D_{\mathcal Z}}\qquad\forall g\in G .
\]
\end{lemma}

\begin{proof}
Let $\cX$ be the space of $J_{V}(\rho)$, a representation of $L$ of finite length, and let
$\xi\mapsto\bar\xi$ be the $Q$-equivariant projection $\cV_\rho\to\cX$. By
\eqref{eq:homfactor}, $\lambda_{\mathcal Z}=\bar\lambda\circ(\xi\mapsto\bar\xi)$ for a unique
$\bar\lambda\in\Hom_{N_L}(\cX,\psi_L)$. For $\eta\in\cX$ put $W^\cX_\eta(b)=\bar\lambda(J_V(\rho)(b)\eta)$,
$b\in L$. The space $\cW=\{W^\cX_\eta:\eta\in\cX\}$ is an $L$-invariant subspace of
$\Ind_{N_L}^L\psi_L$ and a quotient of $\cX$, hence of finite length, and
Corollary~\ref{cor:whittakermg} provides $D_{\mathcal Z}:=D_\cW$ and constants $C_\eta$. The
space $\cW$ is determined by $\bar\lambda$, hence by $\lambda_{\mathcal Z}$, and by
Lemma~\ref{lem:zele1dim} it does not change when $\lambda_{\mathcal Z}$ is multiplied by a
nonzero scalar; thus $\cW$, and with it $D_{\mathcal Z}$, depends only on $\rho$.

Let $\xi\in\cV_\rho$. The set $\{\rho(\kappa)\xi:\kappa\in K\}$ is finite, since $\xi$ is
smooth and $K$ is compact; let $\eta_1,\dots,\eta_l\in\cX$ be the images of its elements
and $C_{\mathcal Z}(\xi):=\max_mC_{\eta_m}$. Let $g\in G$ and write $g=vb\kappa$ with $v\in V$,
$b\in L$, $\kappa\in K$. Then, by \eqref{eq:WZequivN}, \eqref{eq:WZequivG}, the triviality of $\Psn$ on $V$ and
the $L$-equivariance of the projection,
\[
W^{\mathcal Z}_\xi(g)=\lambda_{\mathcal Z}\big(\rho(v)\rho(b)\rho(\kappa)\xi\big)
=\bar\lambda\big(J_V(\rho)(b)\overline{\rho(\kappa)\xi}\big)=W^\cX_{\eta_m}(b)
\]
for the index $m$ with $\eta_m=\overline{\rho(\kappa)\xi}$. Hence
$\abs{W^{\mathcal Z}_\xi(g)}\le C_{\mathcal Z}(\xi)\norm b^{D_{\mathcal Z}}\le C_{\mathcal Z}(\xi)\norm g^{D_{\mathcal Z}}$, using
$\norm b\le\norm{vb}=\norm{vb\kappa}$ by properties~\ref{it:h3} and \ref{it:h1} of \S~\ref{ssec:height}.
\end{proof}

\begin{corollary}\label{cor:identification}
Let $\lambda_{\mathcal Z}$ be as in \eqref{eq:WZ}. There is $c_0\in\bC^\times$ such that
\begin{equation}\label{eq:identification}
W_\xi(g)\ =\ c_0^{-1}\int_{U'\cap\bN}W^{\mathcal Z}_\xi(\bu\,w_0g)\,d\bu
\qquad\forall \xi\in \cV_\rho,\ g\in G,
\end{equation}
the integrand being compactly supported.
\end{corollary}

\begin{proof}
By Lemma~\ref{lem:zele1dim}, $\Lambda_{\lambda_{\mathcal Z}}$ is a nonzero element of $\Hom_U(\rho,\psi_k)$, since $\lambda_{\mathcal Z}\ne0$. Hence
$\Lambda_{\lambda_{\mathcal Z}}=c_0\lambda$ with $c_0\ne0$, because $\dim\Hom_U(\rho,\psi_k)=1$ by Definition~\ref{def:kc}, and
$W_\xi(g)=\lambda(\rho(g)\xi)=c_0^{-1}\Lambda_{\lambda_{\mathcal Z}}(\rho(g)\xi)=c_0^{-1}(\cT W^{\mathcal Z}_\xi)(w_0g)$,
and \eqref{eq:Tformula} gives \eqref{eq:identification}.
\end{proof}

\section{Moderate growth of \texorpdfstring{$(k,c)$}{(k,c)} model functions}\label{sec:mg}

Recall that $n=kc$, and recall the level of a vector together with the convention that levels are greater than $e_\psi$ (\S~\ref{ssec:notation}). 
Throughout this section $\rho$ is a $(k,c)$ representation of $G$, with space $\cV_\rho$ and $(k,c)$ functional $\lambda$.
We also use the map $\xi\mapsto W^{\mathcal Z}_\xi$ of \S~\ref{sec:zelevinsky}, a nonzero $G$-equivariant map from $\rho$ to $\Ind_N^G\Psi_N$.

We start with a bound on the support of $g\cdot W^{\mathcal Z}_\xi$, with explicit dependence on $g$.
\begin{lemma}\label{lem:keysupport}
Let $\xi\in \cV_\rho$ have level $r$, let $g\in G$, and let $\bu\in U'\cap\bN$ satisfy
$W^{\mathcal Z}_\xi(\bu g)\ne0$. Then every entry of $\bu$ has absolute value at most
\begin{equation}\label{eq:R}
R(g):=\big(q^{(k-1)(r-e_\psi)}\,\norm g^{2}\big)^{n}.
\end{equation}
\end{lemma}

\begin{proof}
Put $h=\bu g$ and write $h=n_ha\kappa$ with $n_h\in N$, $a=\diag(a_1,\dots,a_n)\in T$ and
$\kappa\in K$.

\underline{The torus variable inside the blocks.} By \eqref{eq:WZequivN} and \eqref{eq:WZequivG},
$W^{\mathcal Z}_\xi(h)=\Psn(n_h)W^{\mathcal Z}_{\xi'}(a)$ with $\xi'=\rho(\kappa)\xi$, which has level $r$ since
$K(r)$ is normal in $K$. Let $1\le i<n$ with $k\nmid i$ and $x\in\varpi^r\Oo$. Then
$I_n+xE_{i,i+1}\in K(r)$ fixes $\xi'$, and $a(I_n+xE_{i,i+1})a^{-1}=I_n+a_ia_{i+1}^{-1}xE_{i,i+1}$,
so \eqref{eq:WZequivG} gives
\[
W^{\mathcal Z}_{\xi'}(a)=W^{\mathcal Z}_{\xi'}\big(a(I_n+xE_{i,i+1})\big)=\psi(a_ia_{i+1}^{-1}x)\,W^{\mathcal Z}_{\xi'}(a).
\]
Since $W^{\mathcal Z}_{\xi'}(a)\ne0$, $\psi$ is trivial on $a_ia_{i+1}^{-1}\varpi^r\Oo$, and
\eqref{eq:conductor} gives $\abs{a_i}\le q^{r-e_\psi}\abs{a_{i+1}}$. Iterating inside the block
$\{(j-1)k+1,\dots,jk\}$,
\begin{equation}\label{eq:inblock}
\abs{a_{(j-1)k+l}}\ \le\ q^{(k-l)(r-e_\psi)}\abs{a_{jk}},\qquad 1\le l<k,\ 1\le j\le c .
\end{equation}

\underline{The rows $jk$.} Since $U'\subset D$ (see \eqref{eq:D}), $e_{jk}\bu=e_{jk}$, so the
row $jk$ of $h$ is $e_{jk}g$, the row $jk$ of $g$, and $\ct(e_{jk}h)\le\norm g$. On the
other hand $e_{jk}h=(e_{jk}n_ha)\kappa$ and
$e_{jk}n_ha=a_{jk}e_{jk}+\sum_{l>jk}(n_h)_{jk,l}\,a_l\,e_l$, so that
$\ct(e_{jk}n_ha)\ge\abs{a_{jk}}$; by the $K$-invariance of $\ct$, $\abs{a_{jk}}\le\norm g$. With
\eqref{eq:inblock},
\begin{equation}\label{eq:allai}
\abs{a_i}\ \le\ q^{(k-1)(r-e_\psi)}\,\norm g\qquad\forall 1\le i\le n .
\end{equation}

\underline{The minors.} For $1\le i\le n$ let $s=n-i+1$ and let $\mathcal{R}_i$ be the $s\times n$ matrix
formed by the rows $i,\dots,n$ of $\bu$. The rows $i,\dots,n$ of $h$ form the matrix $\mathcal{R}_ig$, and they are also the rows of
$(n_ha)\kappa$. The rows $i,\dots,n$ of $n_ha$ form the matrix $(0\mid B)$ with $B$
upper triangular of size $s$ with diagonal $a_i,\dots,a_n$; its only nonzero maximal minor
is $\det B=a_i\cdots a_n$. By the $K$-invariance of $\ct$ and \eqref{eq:allai},
\[
\ct(\mathcal{R}_ig)=\abs{a_i\cdots a_n}\ \le\ \big(q^{(k-1)(r-e_\psi)}\norm g\big)^{s},
\]
and by \eqref{eq:mug},
$\ct(\mathcal{R}_i)\le\ct(\mathcal{R}_ig)\norm g^{s}\le(q^{(k-1)(r-e_\psi)}\norm g^{2})^{s}\le R(g)$.
Finally let $l<i$. The minor of $\mathcal{R}_i$ on the columns $l,i+1,\dots,n$ equals $\bu_{il}$:
its first row is $(\bu_{il},0,\dots,0)$, because $\bu$ is lower triangular, and its
remaining rows and columns form the lower unitriangular matrix of the rows and columns
$i+1,\dots,n$ of $\bu$. Hence $\abs{\bu_{il}}\le\ct(\mathcal{R}_i)\le R(g)$. The entries of $\bu$
above the diagonal vanish and those on the diagonal equal $1\le R(g)$.
\end{proof}

\begin{remark}\label{rem:keysupport}
Taking $g=I_n$ in Lemma~\ref{lem:keysupport} recovers the compact support of $W^{\mathcal Z}_\xi$ on
$U'\cap\bN$ (\cite[Lemma~3.8]{LM20}, restricted from $D\cap\bN$). The point of the lemma
is the dependence on $g$: the exponent $2n$ of $\norm g$ in \eqref{eq:R} is universal, and
the level $r$ of $\xi$ enters only the constant $q^{(k-1)(r-e_\psi)n}$. Both features are needed
in Theorem~\ref{thm:mgpsi}, where the exponent must not depend on $\xi$.
\end{remark}

For brevity, we put $\fd':=\fd_{U'\cap\bN}=\binom c2\binom k2$ (see \eqref{eq:delta}).
Let $D_{\mathcal Z}$ be the exponent of Lemma~\ref{lem:zelmg} and put
\begin{equation}\label{eq:Dformula}
D:=D_{\mathcal Z}+2n\big(\fd'+(n-1)D_{\mathcal Z}\big).
\end{equation}
The exponent $D$ depends only on $\rho$. We first record its uniformity for the representations $\rho_c(\tau)$.

\begin{lemma}\label{lem:boundednessD}
Let $\tau\in\Irr(\GL_k)$ be generic and $\rho=\rho_c(\tau)$ as defined in Remark~\ref{rem:rhoc}. Assume that the real parts of the cuspidal exponents of $\tau$ are bounded in absolute value
by $A_\tau\ge0$. Then the constant $D$ of \eqref{eq:Dformula} can be taken so that it is bounded by a function of $A_\tau$, $k$ and $c$.
\end{lemma}
\begin{proof}
The constant $D_{\mathcal Z}$ is taken in Lemma~\ref{lem:zelmg} to be $D_\cW$, where $\cW$ is a quotient of $J_V(\rho_c(\tau))$.
Looking at the definition \eqref{eq:Dformula} of $D$, it is enough to show that $D_\cW$ is bounded by a function of $A_\tau$, $k$ and $c$.

There are unitary supercuspidal representations $\sigma'_i\in\Irr(\GL_{n_i})$, $1\leq i\leq l'$, and $d_i\in\bR$ such that
$\tau$ is a subrepresentation of $\sigma'_1|\det|^{d_1}\times\dots\times\sigma'_{l'}|\det|^{d_{l'}}$.
The assumption on $\tau$ implies that $|d_i|$ is bounded by $A_\tau$ plus a function of $k$, for all $i$ (the Jacquet modules being unnormalized).
By the definitions $\rho_c(\tau)$ is a subquotient of $\Pi=\sigma_1|\det|^{h_1}\times\dots\times\sigma_l|\det|^{h_l}$
where each $\sigma_i$ is $\sigma'_{j_i}$ for some $1\le j_i\le l'$, $l\geq l'$, and $h_i=d_{j_i}+r_i\in\bR$ with $r_i$ half-integers depending only on $k$ and $c$.
By the geometric lemma of Bernstein and Zelevinsky \cite{BZ77}, or its coordinate form \cite[1.6]{Zel80}, and \cite[2.8, Theorem]{BZ77},
for any parabolic subgroup $P_1=M_1\ltimes U_1$ of $G$ with unipotent radical $U_1$, the length of $J_{U_1}(\Pi)$, a fortiori of $\rho_c(\tau)$ and of
$J_V(\rho_c(\tau))$, is at most $(n!)^2$. This means that the degree of the polynomials $\mathcal Q$ of Lemma~\ref{lem:expansion} is bounded by a function of $k$ and $c$.
Again by \cite[1.6]{Zel80}, we deduce that the characters of the center of $M_1$ acting on the constituents of $J_{U_1}(\Pi)$ are of the form $t\mapsto\prod_i\omega_{\sigma_i}(t_i)|\det t_i|^{h_i}$ up to $\delta_{P_1}^{-1/2}$ ($\delta_{P_1}$ is the modulus character of $P_1$). Thus the real parts of the cuspidal exponents of $\Pi$, a fortiori those of
$\rho_c(\tau)$ and of $J_V(\rho_c(\tau))$, are bounded in absolute value by $A_\tau$ plus a function of $k$ and $c$.
We conclude that the constant $D_\cW$ of Corollary~\ref{cor:whittakermg} is bounded by a function of $A_\tau$, $k$ and $c$ alone, as required.
\end{proof}

\begin{remark}\label{rem:Dbound}
When $\rho$ is a ladder representation, for example when $\rho=\rho_c(\tau)$ for a square-integrable $\tau$, the Jacquet modules of $\rho$ were computed by
Kret and Lapid \cite{KL12}, so that its cuspidal exponents, and with them the constant $D_{\mathcal Z}$, can in principle be obtained.
\end{remark}

We can now state the first main result of this work.
\begin{theorem}\label{thm:mgpsi}
Let $\rho$ be a $(k,c)$ representation of $G$ on the space $\cV_\rho$, and let $D$ be given by \eqref{eq:Dformula}.
\begin{enumerate}[leftmargin=*,label=(\arabic*),ref=(\arabic*)]
\item\label{it:mgpsi1} For every $\xi\in \cV_\rho$ there is $C(\xi)\ge0$ such that
\[
\abs{W_\xi(g)}\ \le\ C(\xi)\,\norm g^{D}\qquad\forall g\in G .
\]
Explicitly, if $\xi$ has level $r$, one may take
$C(\xi)=\abs{c_0}^{-1}C_{\mathcal Z}(\xi)\,q^{e_\psi\fd'/2+(k-1)(r-e_\psi)n(\fd'+(n-1)D_{\mathcal Z})}$, with $c_0$ as in
\eqref{eq:identification} and $C_{\mathcal Z}(\xi)$ as in Lemma~\ref{lem:zelmg}.
\item\label{it:mgpsi2} Assume that $\rho=\rho_c(\tau)$, where $\tau\in\Irr(\GL_k)$ is generic, and that the real parts of the cuspidal exponents
of $\tau$ are bounded in absolute value by $A_\tau\ge0$. Then $D$ may be taken so that it is bounded by a function of $A_\tau$, $k$ and $c$.
\end{enumerate}
\end{theorem}
\begin{proof}
Assertion \ref{it:mgpsi2} is Lemma~\ref{lem:boundednessD}. We prove \ref{it:mgpsi1}.
Let $\xi$ have level $r$ and let $g\in G$. By \eqref{eq:identification},
\[
\abs{W_\xi(g)}\ \le\ \abs{c_0}^{-1}\vol(S)\,\sup_{\bu\in S}\abs{W^{\mathcal Z}_\xi(\bu w_0g)},
\qquad S:=\{\bu\in U'\cap\bN:\ W^{\mathcal Z}_\xi(\bu w_0g)\ne0\}.
\]
By Lemma~\ref{lem:keysupport}, applied to $w_0g$, and by property~\ref{it:h1} of \S~\ref{ssec:height}, every entry of every
$\bu\in S$ has absolute value at most $R:=R(w_0g)=R(g)$, a power of $q$; hence
$\vol(S)\le q^{e_\psi\fd'/2}R^{\fd'}$ by \eqref{eq:delta}. For $\bu\in S$, property~\ref{it:h5} gives
$\norm\bu\le R^{n-1}$, and properties~\ref{it:h2} and \ref{it:h1} give
$\norm{\bu w_0g}\le R^{n-1}\norm g$. By Lemma~\ref{lem:zelmg},
\[
\abs{W_\xi(g)}\ \le\ \abs{c_0}^{-1}q^{e_\psi\fd'/2}R^{\fd'}\,C_{\mathcal Z}(\xi)\big(R^{n-1}\norm g\big)^{D_{\mathcal Z}}
=\abs{c_0}^{-1}q^{e_\psi\fd'/2}C_{\mathcal Z}(\xi)\,R^{\fd'+(n-1)D_{\mathcal Z}}\,\norm g^{D_{\mathcal Z}}.
\]
Substituting $R=q^{(k-1)(r-e_\psi)n}\norm g^{2n}$ gives the assertion.
\end{proof}
Theorem~\ref{thm:intromg} is assertion \ref{it:mgpsi1} of Theorem~\ref{thm:mgpsi}. Since $\norm g\ge1$ for all $g\in G$, the bound \ref{it:mgpsi1} remains valid when $D$ is replaced by any larger number. Note that $D\ge0$ by \eqref{eq:Dformula}.

\section{The gauge estimate}\label{sec:gauge}

Let $\rho$ be a $(k,c)$ representation of $G$, on the space $\cV_\rho$, and fix a $(k,c)$ functional $\lambda$ on $\rho$.
In this section we use the results of \S~\ref{sec:mg} to derive the gauge estimate on the $(k,c)$ model of $\rho$.
Recall that $M=M_{(c^k)}$ and $U=V_{(c^k)}$ (see \S~\ref{ssec:kcdata}),  the model functions are defined by \eqref{eq:kcmodel},
$W_\xi(g)=\lambda(\rho(g)\xi)$ for $\xi\in\cV_\rho$, and the character $\eta_\rho$ of $\GL_c$ is defined by \eqref{eq:etarho}.

By \eqref{eq:extractcenter}, it is enough to bound $W_\xi$ on $y\in M_{(c^k)}$ whose $k$-th block is $I_c$.
\begin{lemma}\label{lem:support}
Let $\xi$ have level $r$, let $y=\diag(y_1,\dots,y_{k-1},I_c)\in M$ and suppose $W_\xi(y)\ne0$. Then, with
$y_k:=I_c$,
\begin{equation}\label{eq:support}
y_{i+1}^{-1}y_i\in\varpi^{-(r-e_\psi)}\Mat_c(\Oo)\qquad(1\le i\le k-1).
\end{equation}
Consequently every $y_i$ lies in $\varpi^{-(k-1)(r-e_\psi)}\Mat_c(\Oo)$, and
$\abs{\det y_i}\le q^{c(k-1)(r-e_\psi)}$.
\end{lemma}

\begin{proof}
Let $E_{i,i+1}(X)\in U$ be the identity matrix with $X\in\Mat_c$ in the block
$(i,i+1)$, let $X\in\varpi^r\Mat_c(\Oo)$, so that $\rho(E_{i,i+1}(X))\xi=\xi$. Then $yE_{i,i+1}(X)y^{-1}=E_{i,i+1}(y_iXy_{i+1}^{-1})$
and \eqref{eq:equivariance} gives
\[
W_\xi(y)=\psi\big(\tr(y_iXy_{i+1}^{-1})\big)W_\xi(y)=\psi\big(\tr(y_{i+1}^{-1}y_iX)\big)W_\xi(y).
\]
Nonvanishing forces $\psi(\tr(y_{i+1}^{-1}y_iX))=1$ for all such $X$, and
\eqref{eq:conductor} applied entry by entry gives \eqref{eq:support}. The telescoping
product $y_i=(y_k^{-1}y_{k-1})(y_{k-1}^{-1}y_{k-2})\cdots(y_{i+1}^{-1}y_i)$ has at most
$k-1$ factors with entries of valuation at least $-(r-e_\psi)$, so the entries of $y_i$ have
valuation at least $-(k-1)(r-e_\psi)$, and $\abs{\det y_i}\le q^{c(k-1)(r-e_\psi)}$.
\end{proof}

\begin{lemma}\label{lem:cramer}
Let $h\in\GL_c\cap\varpi^{-l}\Mat_c(\Oo)$ with $l\ge0$. Then
$\norm h\le q^{(c+1)l}\abs{\det h}^{-1}$. Consequently, for
$y=\diag(y_1,\dots,y_{k-1},I_c)\in M$ with all $y_i\in\varpi^{-l}\Mat_c(\Oo)$,
\begin{equation}\label{eq:heighttodet}
\norm y\ \le\ q^{C(k,c)\,l}\prod_{i=1}^{k-1}\abs{\det y_i}^{-1},\qquad C(k,c):=(c+1)+c(k-2).
\end{equation}
\end{lemma}

\begin{proof}
The entries of $h$ have valuation at least $-l$, so $\abs{\det h}\le q^{cl}$; each entry
of $h^{-1}$ is a cofactor divided by $\det h$, and a cofactor has absolute value at most
$q^{(c-1)l}$. Hence $\abs{h_{ij}}\le q^l\le q^{(c+1)l}\abs{\det h}^{-1}$ and
$\abs{(h^{-1})_{ij}}\le q^{(c-1)l}\abs{\det h}^{-1}\le q^{(c+1)l}\abs{\det h}^{-1}$.

For the second claim, $\norm y=\max_i\norm{y_i}$, and multiplying the bound for the maximizing
index by $\prod_{j\ne i}\abs{\det y_j}^{-1}q^{c(k-2)l}\ge 1$ gives \eqref{eq:heighttodet}.
\end{proof}

We can now state the second main result of this work.
\begin{theorem}\label{thm:gaugepsi}
Let $D$ be the exponent of Theorem~\ref{thm:mgpsi}.
\begin{enumerate}[leftmargin=*,label=(\arabic*),ref=(\arabic*)]
\item\label{it:gaugepsi1} For any $\xi\in \cV_{\rho}$ there is a constant $\gamma(\xi)\geq0$ such that for all $y=\diag(y_1,\ldots,y_k)\in M$, $u\in U$ and $\kappa\in K$,
\begin{align}\label{eq:thegauge1psi}
\abs{W_\xi(uy\kappa)}\leq\gamma(\xi)\abs{\eta_\rho(y_k)}\cdot\prod_{i=1}^{k-1}\abs{\det (y_k^{-1}y_i)}^{-D}.
\end{align}
\item\label{it:gaugepsi2} For any $\xi\in\cV_\rho$ there are $\phi_1,\ldots,\phi_{k-1}\in\mathcal{S}(\Mat_c)$ such that for all $y,u$ and $\kappa$ as in \ref{it:gaugepsi1},
\begin{align}\label{eq:thegauge2psi}
\abs{W_\xi(uy\kappa)}\leq\abs{\eta_\rho(y_k)}\cdot\prod_{i=1}^{k-1}\abs{\det (y_k^{-1}y_i)}^{-D}\phi_i(y_{i+1}^{-1}y_i).
\end{align}
\item\label{it:gaugepsi3} Assume that $\rho=\rho_c(\tau)$, where $\tau\in\Irr(\GL_k)$ is generic, and that the real parts of the cuspidal exponents
of $\tau$ are bounded in absolute value by $A_\tau\ge0$. Then the constant $D$ above may be taken so that it is bounded by a function of $A_\tau$, $k$ and $c$.
\end{enumerate}
\end{theorem}
\begin{proof}
Assertion \ref{it:gaugepsi3} follows from Theorem~\ref{thm:mgpsi}\ref{it:mgpsi2}.
For \ref{it:gaugepsi1}, by \eqref{eq:equivarianceKC} it suffices to prove \eqref{eq:thegauge1psi} with $u=I_n$; moreover $W_\xi(y\kappa)=W_{\rho(\kappa)\xi}(y)$ and the vector
$\rho(\kappa)\xi$ has the same level as $\xi$ (the group $K(r)$ is normal in $K$), so it suffices to prove \eqref{eq:thegauge1psi} with $u=\kappa=I_n$, with a constant $\gamma(\xi)$ depending on $\xi$ only through its level and
the constant $C(\xi)$ of Theorem~\ref{thm:mgpsi} (which we then replace by its maximum over the finitely many vectors $\rho(\kappa)\xi$, $\kappa\in K$).
Let $\xi$ have level $r$ and consider $y=\diag(y_1,\dots,y_{k-1},I_c)\in M$ such that $W_\xi(y)\ne0$. By
Lemma~\ref{lem:support}, $y_i\in\varpi^{-l}\Mat_c(\Oo)$ with $l=(k-1)(r-e_\psi)\ge0$ and
$\abs{\det y_i}\le q^{cl}$, for each $1\leq i<k$. By Theorem~\ref{thm:mgpsi} and Lemma~\ref{lem:cramer},
\[
\abs{W_\xi(y)}\ \le\ C(\xi)\,\norm y^{D}\ \le\ C(\xi)\,q^{D\,C(k,c)l}\prod_{i=1}^{k-1}\abs{\det y_i}^{-D}.
\]
This, with $\gamma(\xi)=C(\xi)q^{D\,C(k,c)l}$, together with \eqref{eq:extractcenter}, implies \eqref{eq:thegauge1psi}.
The second bound \eqref{eq:thegauge2psi} also follows, since we can take $\phi_i$ to be the characteristic function of
$\varpi^{-l}\Mat_c(\Oo)$, with one of these functions scaled by $\gamma(\xi)$.
\end{proof}

Theorem~\ref{thm:introgauge} consists of assertions \ref{it:gaugepsi1} and \ref{it:gaugepsi2} of Theorem~\ref{thm:gaugepsi}.

\section{The global argument for local components of cuspidal representations}\label{sec:global}

In this section we present the global argument which proves Theorem~\ref{thm:intromg}
for the representations $\rho_c(\tau_v)$, where $\tau_v$ is the component at a
non-archimedean place $v$ of an irreducible cuspidal automorphic representation $\tau$.

\subsection{Notation}\label{ssec:globalnotation}
Let $\bF$ be a number field with ring of adeles $\bA$. Let $n=kc$ as in \S~\ref{ssec:notation}.
We fix a maximal compact subgroup $K_{\bA}=\prod_vK_v$ of $\GL_n(\bA)$ as in \cite[I.1.1]{MW95} (at the finite places this is the group $K$ of \S~\ref{ssec:notation}).

Define the height on $\GL_n(\bA)$ following \cite[I.2.2]{MW95}:
\begin{equation*}
\norm g_{\bA}:=\prod_v\norm{g_v}_v,\qquad
\norm{g_v}_v:=\max_{i,j}\big(\abs{g_{v,ij}}_v,\abs{(g_v^{-1})_{ij}}_v\big).
\end{equation*}
Here at the infinite places $|\cdot|_v$ is the standard absolute value when $\bF_v=\bR$, and the square of the standard absolute value when $\bF_v=\bC$.
The height is bounded from below: at the infinite places $\norm g_v\ge1/n$ hence $\norm g_{\bA}\ge n^{-[\bF:\mathbb Q]}$ for all $g$.
The global analog of the submultiplicativity at the finite places (\S~\ref{ssec:height}, \ref{it:h2}) is
\begin{equation}\label{eq:adelicsubmult}
\norm{gh}_{\bA} \le n^{[\bF:\mathbb Q]}\,\norm g_{\bA}\,\norm h_{\bA},\qquad g,h\in\GL_n(\bA).
\end{equation}

Fix a nontrivial character $\psi=\otimes'_v\psi_v$
of $\bF\backslash\bA$ and let $\psi_k$ be the character of $U(\bA)$ defined by
\eqref{eq:psikdef} with $\psi$; then $(\psi_v)_k$ is the character \eqref{eq:psikdef} of $U(\bF_v)$ defined by $\psi_v$.

A function $\varphi$ on $\GL_n(\bA)$ has moderate growth, in the sense of
\cite[I.2.3]{MW95}, if there are $C_\varphi\ge0$ and $D_\varphi\in\bR$ with
\begin{equation}\label{eq:adelicmg}
\abs{\varphi(g)}\ \le\ C_\varphi\,\norm g_{\bA}^{D_\varphi}\qquad\forall g\in\GL_n(\bA),
\end{equation}
and every automorphic form on $\GL_n(\bF)\backslash\GL_n(\bA)$ has moderate growth (see \cite[I.2.17]{MW95}).
Since $\norm g_{\bA}$ is bounded from below, we take $D_\varphi\ge0$ (at the cost of possibly increasing $C_\varphi$, depending on $D_\varphi$).

\subsection{The global argument}\label{ssec:globalargument}
Let $\tau$ be an irreducible automorphic unitary cuspidal representation of $\GL_k(\bA)$.
Write $\tau$ as a restricted tensor product, $\tau=\otimes_v'\tau_v$.
Let $\rho_c(\tau)$ be the generalized Speh representation of $\GL_n(\bA)$ associated with $\tau$ (\cite[\S 2.4]{Jac84}).
The representation $\rho_c(\tau_v)$ in this case is also the Langlands quotient
of \[\tau_v|\det|_v^{(c-1)/2}\times\dots\times\tau_v|\det|_v^{(1-c)/2}\] (see \cite[\S2.2]{CFGK23}) and by \cite[\S 2.4]{Jac84} and \cite[I.2, I.11, II.3]{MW89},
it is the component at $v$ of $\rho_c(\tau)$, that is, $\rho_c(\tau)=\otimes_v'\rho_c(\tau_v)$. Let $\cV$ denote the abstract space of $\rho_c(\tau)$ and
$\cV_v$ the space of $\rho_c(\tau_v)$, so that $\cV=\otimes'_v\cV_v$. For $\xi\in\cV$ denote the corresponding automorphic form by $\varphi_{\xi}$.

We define the global $(k,c)$ coefficient on $\rho_c(\tau)$ by
\begin{equation}\label{eq:globalcoefficient}
\mathcal{L}(\varphi_{\xi}):=\int_{U(\bF)\backslash U(\bA)}\varphi_{\xi}(u)\,\psi_k^{-1}(u)\,du .
\end{equation}
(This is the coefficient \cite[(0.1)]{CFK22}.)
By definition, $\mathcal{L}(u\cdot \varphi_{\xi})=\psi_k(u)\mathcal{L}(\varphi_{\xi})$ for all $\xi\in\cV$ and $u\in U(\bA)$.

The coefficient \eqref{eq:globalcoefficient} is not identically zero by the results of Jiang and Liu \cite[Theorem 5.6(1)]{JL13}.
Thus, there is a pure tensor $\xi=\otimes'_v\xi_v\in\cV$ such that $\mathcal{L}(\varphi_\xi)\ne0$.

For each place $v$ of $\bF$ let $0\ne\lambda_v\in\Hom_{U(\bF_v)}(\rho_c(\tau_v),(\psi_v)_k)$. 

\begin{theorem}\label{thm:globalgauge}
There is a constant $D$ depending only on $\tau$, $c$ and the choice of $\xi$ above, such that for each finite place $v$ of $\bF$, 
for each $\xi'_v\in\cV_v$, there is a constant $C(\xi'_v)\geq0$ such that
for all $g_v\in\GL_n(\bF_v)$,
\begin{equation}\label{eq:boundLclGbl}
\abs{\lambda_v(\rho_c(\tau_v)(g_v)\xi'_v)}\leq C(\xi'_v){\norm g_v}^D.
\end{equation}
Consequently, the assertions \eqref{eq:thegauge1psi} and \eqref{eq:thegauge2psi} of Theorem~\ref{thm:gaugepsi} hold
for $\rho_c(\tau_v)$, with the constant $D$ of that theorem uniform across all places $v$.
\end{theorem}

\begin{proof}
First observe that \eqref{eq:boundLclGbl} implies \eqref{eq:thegauge1psi} and \eqref{eq:thegauge2psi} for $\rho_c(\tau_v)$. Indeed \eqref{eq:boundLclGbl} is the bound of Theorem~\ref{thm:mgpsi}\ref{it:mgpsi1}, and the proof of Theorem~\ref{thm:gaugepsi} applies with this $D$.

Let $v$ be a finite place of $\bF$. We have an embedding $\rho_c(\tau_v) \hookrightarrow \rho_c(\tau)$ given by \[\xi'_v\mapsto \xi':=\xi'_v\otimes(\otimes_{w\ne v} \xi_w).\]
Composing this embedding with $\mathcal{L}$ we obtain a $(k,c)$ functional $\mathcal{L}_v$ on $\rho_c(\tau_v)$, specifically $\mathcal{L}_v:\cV_v\to\bC$ is defined by
\[\mathcal{L}_v(\xi'_v)=\mathcal{L}(\varphi_{\xi'});\] $\mathcal{L}_v$ is not identically zero because $\mathcal{L}_v(\xi_v)=\mathcal{L}(\varphi_{\xi})$.
Since $\dim \Hom_{U(\bF_v)}(\rho_c(\tau_v),(\psi_v)_k)=1$, it is enough to prove the theorem for $\mathcal{L}_v$.

Let $\xi'_v\in\cV_v$. Since $\rho_c(\tau_v)$ is irreducible, there are
$h_1,\ldots,h_l\in\GL_n(\bF_v)$ and $c_1,\ldots ,c_l\in\bC$ such that
\[\xi'_v=\sum_{i=1}^lc_i(\rho_c(\tau_v)(h_i))\xi_v.\]
Then $\varphi_{\xi'}=\sum_{i=1}^lc_i(h_i\cdot\varphi_{\xi})$.

Let $C:=C_{\varphi_\xi}$ and $D:=D_{\varphi_{\xi}}$ be the constants given by \eqref{eq:adelicmg}, i.e.,
\begin{equation*}
\abs{\varphi_{\xi}(g)}\ \le\ C\,\norm g_{\bA}^{D}\qquad\forall g\in\GL_n(\bA).
\end{equation*}

Let
\[C(\xi'_v):=\sum_i\abs{c_i}\,C\big(n^{[\bF:\mathbb Q]}\norm{h_i}_{v}\big)^{D}.\]
By \eqref{eq:adelicmg} and \eqref{eq:adelicsubmult}, and since $\norm{h_i}_{v}=\norm{h_i}_{\bA}$ ($h_i\in\GL_n(\bF_v)$),
for all $g\in\GL_n(\bA)$,
\begin{equation}\label{eq:varphiprime}
\abs{\varphi_{\xi'}(g)}\le\sum_i\abs{c_i}\abs{\varphi_{\xi}(gh_i)}
\le\sum_i\abs{c_i}\,C\big(n^{[\bF:\mathbb Q]}\norm g_{\bA}\norm{h_i}_{\bA}\big)^{D}=C(\xi'_v){\norm g_{\bA}}^D.
\end{equation}

Since $U(\bF)\backslash U(\bA)$ is compact, there is a compact subset $\mathcal C\subset U(\bA)$ which surjects onto it. By \eqref{eq:varphiprime} and \eqref{eq:adelicsubmult},
\begin{equation}\label{eq:suponcoefficient}
\sup_{u\in\mathcal C}\abs{\varphi_{\xi'}(ug)} \le C(\xi'_v)\sup_{u\in\mathcal C}\norm{ug}_{\bA}^D  \le C(\xi'_v)n^{[\bF:\mathbb Q]D}(\sup_{u\in\mathcal C}\norm{u}_{\bA}^{D}){\norm g_{\bA}}^D.
\end{equation}
Since $\mathcal C$ is compact and the norm is continuous, $\sup_{u\in\mathcal C}\norm{u}_{\bA}^{D}$ is finite. Now we can bound
\eqref{eq:globalcoefficient} for $g\cdot\varphi_{\xi'}$ using \eqref{eq:suponcoefficient} and obtain
\[
\abs{\mathcal{L}(g\cdot\varphi_{\xi'})}\leq\vol\big(U(\bF)\backslash U(\bA)\big)\,\sup_{u\in\mathcal C}\abs{\varphi_{\xi'}(ug)}\le C_0 C(\xi'_v)\norm{g}_{\bA}^{D},
\]
where $C_0:=\vol\big(U(\bF)\backslash U(\bA)\big)n^{[\bF:\mathbb Q]D}\sup_{u\in\mathcal C}\norm{u}_{\bA}^{D}$.

We conclude that
\[
\abs{\mathcal{L}_v(\rho_c(\tau_v)(g_v)\xi'_v)}\leq C_0 C(\xi'_v)\norm{g_v}_v^{D},\qquad \forall g_v\in\GL_n(\bF_v).\]
This completes the proof of \eqref{eq:boundLclGbl} and with it the theorem.
\end{proof}
\begin{remark}\label{rem:nonunitary}
Let $v$ be a finite place of $\bF$.
By the definitions, for any $b\in\bC$, $\rho_c(|\det|_v^b\tau_v)=|\det|_v^b\rho_c(\tau_v)$. Since
$|\det g_v|^{\pm1}_v\leq {\norm {g_v}}_v^n$ for all $g_v\in\GL_n(\bF_v)$, the results of Theorem~\ref{thm:globalgauge} hold also when $\tau$ is (cuspidal) nonunitary, i.e.,
$|\det|_{\bA}^{-b}\tau$ is unitary.
\end{remark}

\section{Central extensions}\label{sec:coverings}
Theorems~\ref{thm:intromg} and \ref{thm:introgauge} apply to certain central extensions of general linear groups. In this section we formulate the results in this setting
and describe the modifications of the proofs. 

Let $m\geq1$ be an integer. Assume that $F$ contains a primitive $m$-th root of unity, and let $\mu_m\subset F^\times$ be the group of $m$-th roots of unity.
Define $F^{\times m}:=\{x^m:x\in F^\times\}$. As in the previous sections, $n=kc$ and $G=\GL_n$. Let $\boldsymbol{G}$ be a topological central extension of $G$ by $\mu_m$, so that 
\[
1\to\mu_m\to\boldsymbol{G}\xrightarrow{\ p\ }G(F)\to1.
\]
For a subset $X\subset G$ put $\boldsymbol X=p^{-1}(X)$. A $p$-section of $X$ is a map $\sigma:X\to\boldsymbol{G}$ such that $p\circ\sigma=\mathrm{id}_X$; $\sigma$ is called $Y$-equivariant, for a subgroup $Y$ of $G$, if $y\sigma(x)y^{-1}=\sigma(p(y)xp(y)^{-1})$ for every $x\in X$ and $y\in\boldsymbol Y$. For a subgroup $X$ of $G$, a splitting of $X$ is a continuous homomorphism $\sigma:X\to\boldsymbol{G}$ which is a $p$-section, and we say that $\boldsymbol{G}$ splits over $X$ if such a homomorphism exists. Fix a faithful character of $\mu_m$; a representation of $\boldsymbol X$ is called genuine if $\mu_m$ acts by this character.

The height is defined on $\boldsymbol{G}$ by composing with $p$.

Let $B=T\ltimes N$ denote the Borel subgroup of $G$ as in \S~\ref{ssec:notation}.
We use the following properties of $\boldsymbol{G}$, see \cite{KaplanSzpruch} and the references therein. Let $\underline{n}$ be a composition of $n$.
\begin{enumerate}[leftmargin=*,label=(C\arabic*),ref=(C\arabic*)]
\item\label{it:cov1unipotents} The group $\boldsymbol{G}$ splits uniquely over the group of $F$-points of every unipotent subgroup $X$ of $G$. Thus there is a well-defined $G$-equivariant $p$-section $x\mapsto x^\star$ on the set of unipotent elements of $G$, whose restriction to each unipotent subgroup $X$ is the splitting of $X$.
\item\label{it:cov0identity2} There are $j_0\geq1$ and a splitting $\varsigma:K(j_0)\to\boldsymbol{G}$ such that for all $j\geq j_0$ and $\underline{n}$, $\varsigma|_{K(j)}$ is $K$-equivariant and moreover $\varsigma$ agrees with the map $x\mapsto x^\star$ on $K(j)\cap V_{\underline{n}}$ and on $K(j)\cap\overline{V}_{\underline{n}}$ (\cite[Lemma 2.2]{KaplanSzpruch}). In particular $\boldsymbol{G}$ is an $l$-group.
   
\item\label{it:cov2Ak} There is an open subgroup $A$ of finite index in $T$ such that $\boldsymbol{A}$ is abelian and belongs to the center of $\boldsymbol{T}$.
For example, \[A=T^m:=\{t^m:t\in T\}=\{t=\diag(t_1,\dots,t_n)\in T:t_i\in F^{\times m}\ \text{for all }i\}\] has these properties. 
Note that if $a\in A$ belongs to the center of $M_{\underline{n}}$, then every $a'\in p^{-1}(a)$ belongs to the center of $\boldsymbol{M_{\underline{n}}}$: 
indeed $a'$ commutes with $X^\star$ for each unipotent subgroup $X$ of $M_{\underline{n}}$ by \ref{it:cov1unipotents}, because $a$ commutes with $X$, and 
together with $\boldsymbol{T}$ these generate $\boldsymbol{M_{\underline{n}}}$.
\item\label{it:finLengthAdm} Genuine finite-length representations of $\boldsymbol{M_{\underline{n}}}$ and in particular of $\boldsymbol{G}$ are admissible.
\item\label{it:cov5Jacquet} For a genuine finite-length representation $\pi$ of $\boldsymbol{M_{\underline{n}}}$ and the unipotent radical $U_1$ of a parabolic subgroup $P_1=M_1\ltimes U_1$ of $M_{\underline{n}}$, $J_{U_1}(\pi)$ is a genuine finite-length representation of $\boldsymbol{M_1}$. 
\end{enumerate}
We add the following assumption on the subgroup $A$ of \ref{it:cov2Ak}.
\begin{enumerate}[leftmargin=*,label=(C\arabic*),ref=(C\arabic*),start=6]
\item\label{it:cov6GLtriangle} $\boldsymbol{A}\cong A\times\mu_m$ (a trivial central extension).
\end{enumerate}
For $A=T^m$ this holds whenever the restriction of the $2$-cocycle of $\boldsymbol{G}$ to $T\times T$ is a product of $m$-th order Hilbert symbols $(t_i,t'_j)_m$, which are trivial on $m$-th powers.
Note also that \ref{it:cov2Ak} and \ref{it:cov6GLtriangle} are inherited by any subgroup of finite index in $A$ (any such subgroup is automatically open by Hensel's lemma), and we will use this freedom to shrink $A$ in the proof of Theorem~\ref{thm:covmain} below.

By \ref{it:cov1unipotents} we identify unipotent subgroups $X$ of $G$ with their images in $\boldsymbol{G}$ under their unique splitting, and thus we may view 
a representation $\pi$ of $\boldsymbol{X}$ also as a representation of $X$.

The groups $P$, $M$, $U=V_{(c^{k})}$ and the character $\psi_k$ of $U$ are those of \S~\ref{ssec:kcdata}, and $Q$, $L$, $V$, $N_L$, $\Psn$, $\psi_L$, $w_0$, $U'$ and $\Psu$ are those of \S~\ref{ssec:zeldata}. Recall that the stabilizer of $\psi_{k}$ in $M$ is $\GL_c^{\triangle}$. 

Let $\rho$ be a genuine finite-length representation of $\boldsymbol{G}$ on a space $\cV_\rho$. If $\dim\Hom_{U}(\rho,\psi_{k})=1$, then we let $0\ne\lambda\in \Hom_{U}(\rho,\psi_{k})$ and define, as in the linear case, the homomorphism 
$\cV_{\rho}\to\Ind_U^{\boldsymbol{G}}\psi_k$ by $\xi\mapsto  W_\xi$, where $W_\xi(g)=\lambda(\rho(g)\xi)$ for all $g\in\boldsymbol{G}$. Since $\dim J_{U,\psi_{k}}(\rho)=1$, 
$\boldsymbol{\GL}_c^{\triangle}$ acts on $J_{U,\psi_k}(\rho)$ by a genuine character $\eta_{\rho}$ and 
$|\eta_{\rho}|$ can be regarded as a character of $\GL_c$ (it is trivial on $\mu_m$). Thus 
for $g\in \boldsymbol{\GL}_c^{\triangle}$ such that $p(g)=h^{\triangle}$ and $g'\in\boldsymbol{G}$, 
$\abs{W_\xi(gg')}=\abs{\eta_\rho(h)}\cdot\abs{W_\xi(g')}$.

The following is the analog of Theorems~\ref{thm:intromg} and \ref{thm:introgauge} for $\boldsymbol{G}$.

\begin{theorem}\label{thm:covmain}
Let $\rho$ be a genuine finite-length representation of $\boldsymbol{G}$ on a space $\cV_\rho$ such that
\begin{equation}\label{eq:covuniq}
\dim\Hom_{U}(\rho,\psi_{k})=1.
\end{equation}
Assume that
\begin{equation}\label{eq:covzel}
\Hom_N(\rho,\Psn)\ne0.
\end{equation}
\begin{enumerate}[leftmargin=*,label=(\arabic*),ref=(\arabic*)]
\item\label{it:covmain1} There is a constant $D\ge0$ such that for each $\xi\in\cV_\rho$ there is a constant $C(\xi)$ such that $\abs{W_\xi(g)}\le C(\xi)\norm{g}^{D}$ for all $g\in\boldsymbol{G}$.
\item\label{it:covmain2} For each $\xi\in\cV_\rho$ there are a constant $\gamma(\xi)$ and $\phi_1,\dots,\phi_{k-1}\in\mathcal S(\Mat_c)$ such that for all
$y\in\boldsymbol{M}$, $u\in U$ and $\kappa\in\boldsymbol K$, with $p(y)=\diag(y_1,\dots,y_{k})\in M$,
\begin{align*}
\abs{W_\xi(u^{\star}y\kappa)}&\le\gamma(\xi)\,\abs{\eta_\rho(y_{k})}\prod_{i=1}^{k-1}\abs{\det(y_{k}^{-1}y_i)}^{-D},\\
\abs{W_\xi(u^{\star}y\kappa)}&\le\abs{\eta_\rho(y_{k})}\prod_{i=1}^{k-1}\abs{\det(y_{k}^{-1}y_i)}^{-D}\,\phi_i(y_{i+1}^{-1}y_i).
\end{align*}
\end{enumerate}
Moreover $\dim\Hom_N(\rho,\Psn)=1$.
\end{theorem}

\begin{proof}
The result follows by adjusting the arguments of \S~\ref{sec:whittaker}--\S~\ref{sec:gauge}. We describe the modifications.

\underline{Section~\ref{sec:whittaker}:} We use the splitting $\theta$ fixed in \S~\ref{sec:whittaker}. Since by \ref{it:cov2Ak} the subgroup 
\[(Z_L\cap A)\,\theta\big(\prod_{\alpha\in\Delta}(H^\alpha\cap A)\big)\] 
of $A$ is again open and of finite index in $T$, and satisfies \ref{it:cov2Ak} and \ref{it:cov6GLtriangle}, we may assume that
\[A=(Z_L\cap A)\,\theta\big(\prod_{\alpha\in\Delta}(H^\alpha\cap A)\big);\] 
for $A=T^m$ this already holds because $\ell(t^m)=\ell(t)^m$. Hence for $t\in A$ the decomposition $t=z\theta(a)$ of \S~\ref{sec:whittaker} has $z\in Z_L\cap A$ and $\theta(a)\in A$.  

In the statement of Lemma~\ref{lem:expansion} we replace $T$ with $\boldsymbol A$; accordingly, $H^\alpha$ and $Z_L$ are replaced with the intersections of their preimages with 
$\boldsymbol A$. 
By \ref{it:finLengthAdm}, for $W\in\cW$ the $(\boldsymbol{Z_L}\cap \boldsymbol{A})$-span of $W$ is finite-dimensional. 

For $u\in U^\alpha$, $a^{-1}u^{\star}a=(a^{-1}ua)^{\star}$ for $a\in\boldsymbol{A}$ by \ref{it:cov1unipotents}. The space 
$\mathcal{F}/\mathcal{F}^\alpha$ is a quotient of $J_{U^\alpha}(\cW)$ as a module over $\boldsymbol{H^\alpha}\cap \boldsymbol{A}$. Since by \ref{it:cov2Ak} 
$\boldsymbol{H^\alpha}\cap \boldsymbol{A}$ is central in $\boldsymbol{M^{\alpha}}$, and by \ref{it:cov5Jacquet} the representation $J_{U^\alpha}(\cW)$ is of finite length, we deduce that $\boldsymbol{H^\alpha}\cap \boldsymbol{A}$ acts on its irreducible subquotients by genuine characters.

We apply \cite[Lemma~3.4]{CM15} with the torus replaced with $\prod_\alpha(\boldsymbol{H^\alpha}\cap\boldsymbol{A})$, which by \ref{it:cov6GLtriangle} is $\prod_\alpha(H^\alpha\cap A)\times\mu_m$. Each $H^\alpha\cap A$ is an open subgroup of finite index in $H^\alpha\cong F^\times$, hence of the form $\varpi_\alpha^{\bZ}\times\Oo_\alpha$ with $\varpi_\alpha\in F^\times$ such that $\val(\varpi_\alpha)>0$ and $\Oo_\alpha$ a compact open subgroup of $\Oo^\times$, and the proof of \textit{loc. cit.}, which uses only this structure of $F^\times$, applies to such groups verbatim. 

Corollary~\ref{cor:whittakermg} is valid with $L$ replaced by $\boldsymbol{L}$. In detail, by \ref{it:cov2Ak} 
the index of $\boldsymbol{A}$ in $\boldsymbol{T}$ is finite and $\boldsymbol A$ is central in $\boldsymbol T$, and also 
$t\cdot W\in\cW$ for each $t\in\boldsymbol T$, hence the expansion on $\boldsymbol A$ for finitely many translates of $W$ bounds $W$ on $\boldsymbol T$.

\underline{Section~\ref{sec:transition}:} We consider the spaces $\Ind_{U_i}^{\boldsymbol{G}}\Psi_{U_i}$ of smooth functions on $\boldsymbol{G}$. The result 
\cite[Lemma~3.3]{LM20} and the proof of \cite[Proposition~3.5]{LM20} involve only integrations over unipotent subgroups, conjugations of unipotent elements, the smoothness of the functions, and Fourier inversion on the abelian groups $A_i$ and $B_i$. These results are applicable here in $\boldsymbol{G}$ by \ref{it:cov1unipotents}. 
The argument on the support in the proof of \cite[Lemma~3.8]{LM20} also follows here using \ref{it:cov1unipotents}.
Thus \eqref{eq:fourier} and Proposition~\ref{prop:LM} hold, and the proof
of Lemma~\ref{lem:injective} applies verbatim, so that $\cT$ is injective.

\underline{Section~\ref{sec:zelevinsky}:} The assumption \eqref{eq:covzel} replaces Lemma~\ref{lem:zelexists}. 
In Lemma~\ref{lem:zele1dim} we rely on \eqref{eq:covuniq}. For the conjugation we take some element $w'\in p^{-1}(w_0)$, and use \ref{it:cov1unipotents}; note 
that $\Lambda\circ\rho(w')$ depends on the choice of $w'$ but only up to a scalar in $\mu_m$.
For Lemma~\ref{lem:zelmg}, $J_V(\rho)$ is a genuine representation of $\boldsymbol{L}$ of finite length by \ref{it:cov5Jacquet}, and Corollary~\ref{cor:identification} holds with the same proof.

\underline{Section~\ref{sec:mg}:} The level of a vector $\xi\in\cV_\rho$ is defined using \ref{it:cov0identity2}, namely $\xi$ has level $j$, for $j\geq j_0$, if $\rho(\varsigma(x))\xi=\xi$ for every $x\in K(j)$. Every $\xi$ has a level because $\{\varsigma(K(j))\}_{j\geq j_0}$ is a basis of neighborhoods of the identity in $\boldsymbol G$. Since $K(j)$ is normal in $K$ and $\varsigma|_{K(j)}$ is $K$-equivariant, the vectors 
$\rho(\kappa)\xi$ where $\kappa$ varies in $\boldsymbol{K}$ all have the same level. 
In the proofs of Lemmas~\ref{lem:keysupport} and \ref{lem:support} the level of $\xi$ is used only through the invariance of $\xi$ under $x^\star$ for $x\in N\cap K(j)$, respectively $x\in U\cap K(j)$, and $x^\star=\varsigma(x)$ for such $x$ by \ref{it:cov0identity2}.
For Lemma~\ref{lem:keysupport} we use \ref{it:cov1unipotents}: for every $t\in\boldsymbol{T}$ such that $p(t)=\diag(t_1,\ldots,t_n)$, $t(I_n+xE_{i,i+1})^{\star}t^{-1}
=(I_n+t_it_{i+1}^{-1}xE_{i,i+1})^{\star}$. So the first step of the proof applies, with the modified notion of the level.
The remaining steps concern a computation in $G$ and they apply without change.
The proof of Theorem~\ref{thm:mgpsi} then applies verbatim.

\underline{Section~\ref{sec:gauge}:} 
Lemma~\ref{lem:support} and Theorem~\ref{thm:gaugepsi} hold with justifications similar to the above. Note that the analog of \eqref{eq:extractcenter} holds 
since we are taking the absolute value of $W_\xi$. 
This completes the proof of Theorem~\ref{thm:covmain}.
\end{proof}

For the rest of this section let $\boldsymbol{G}$ be the following extension. Let $\Sp_{2n}^{(m)}$ be the $m$-fold covering group of $\Sp_{2n}$ of Matsumoto \cite{Mat69}, defined with respect to the Steinberg symbol $(\cdot,\cdot)_m^{-1}$. For $g\in G$, we have an embedding $g\mapsto \diag(g,g^*)\in\Sp_{2n}$ where $g^*$ is defined by the symplectic form. Let $\boldsymbol{G}$ denote the central extension of $G$ obtained by restriction of $\Sp_{2n}^{(m)}$ to this image of $G$.
In this case it is easy to see from the formulas of \cite{BLS} (see \cite[\S2.7]{Kaplan2025})
that \ref{it:cov6GLtriangle} holds for $A=T^m$.

For $l\leq n$, let $\boldsymbol{G}_l$ denote the central extension of $\GL_l$ obtained by restriction from $\boldsymbol{G}$ with respect to
the embedding $\GL_l\to G$ in the top left corner (any embedding as a direct factor of a Levi subgroup would work); $\boldsymbol{G}_l$ is also the central extension obtained using $\Sp_{2l}^{(m)}$.

Put $r=m$ if $m$ is odd, and $r=m/2$ otherwise. Following \cite[\S3.1]{GK23}, we assume that $r$ divides $k$ and write $k=rk_0$ with an integer $k_0\ge1$; thus $n=rk_0c$. Note that the integer $k$ of \cite{Kaplan2023,Kaplan2025} is $k_0$ here (cf. \cite[Remark~3.4]{GK23}): the representations considered in \cite{Kaplan2023,Kaplan2025} are attached to genuine representations of $\boldsymbol{G}_{k_0}$.

In \cite[Definition~11]{Kaplan2023}, a genuine finite-length representation $\rho$ of $\boldsymbol{G}$ was called an $(rk_0,c)$ representation, i.e., in our notation a $(k,c)$ representation, if
$\Hom_{V(\beta')}(\rho,\psi')=0$ for every partition $\beta'$ of $n$ which is greater than or not comparable with $(k^c)$ and every generic character $\psi'$ of the unipotent subgroup $V(\beta')$ attached to $\beta'$, and
\begin{equation*}
\dim\Hom_{U}(\rho,\psi_{k})=1.
\end{equation*}
See \cite[\S2.1]{Kaplan2023} for the details (cf. Patel \cite{Patel2015} for the wave front set of representations of $\boldsymbol{G}$). The vanishing condition is not assumed in Theorem~\ref{thm:covmain}
(in the linear case this condition is relevant for the Zelevinsky model); only the last condition, which is \eqref{eq:covuniq}, is.
The difficulty is to construct such representations, i.e., to define an analog of the map $\tau\mapsto\rho_c(\tau)$ for genuine irreducible generic representations $\tau$ of 
$\boldsymbol{G}_{k_0}$; this was carried out in \cite{Kaplan2023} for $k_0=1$. We verify \eqref{eq:covzel} for these representations below.

The case $m=1$ is the linear case ($\boldsymbol{G}=G$).

When $m=2$ (then $r=1$ and $k=k_0$) the genuine representations of $\boldsymbol{G}$ are in a (noncanonical) bijection with the representations of $G$ \cite[\S1.7]{Kaplan2023}, which preserves the models we consider up to
multiplication by a complex-valued function of absolute value $1$ (a Weil factor). The notion of a $(k,c)$ representation is then the same as in the linear case; hence 
\eqref{eq:covzel} holds and Theorem~\ref{thm:covmain} is applicable (Theorems~\ref{thm:intromg} and \ref{thm:introgauge} can also be applied, through the bijection). In the rest of this section we assume $r>1$.

Assume in addition that $F$ contains a primitive $2m$-th root of unity, as in \cite{Kaplan2023}.

Let $k_0=1$, so that $k=r$ and $n=rc$. We recall the construction of \cite[\S2.2]{Kaplan2023}. Let $\tau$ be a genuine irreducible representation of $\boldsymbol{G}_1$; it is constructed by Stone--von Neumann theory from a genuine character of the center of $\boldsymbol{G}_1$,
and this character defines a quasi-character $\chi$ of $F^\times$ uniquely when $r=m$ and noncanonically by choosing a Weil factor
when $r=m/2$. For $1\le d\le n$ let $\Theta_{d,m}$ be the exceptional genuine irreducible representation of $\boldsymbol{G}_d$ of \cite[\S2.2]{Kaplan2023}; the exceptional representations depend on an auxiliary parameter (denoted $\vartheta$ in \textit{loc. cit.}) which we fix once and for all and omit from the notation. 
Then
\begin{equation}\label{eq:covrhoc}
\rho_c(\tau)=(\chi\circ\det\circ p)\Theta_{rc,m},
\end{equation}
and $\rho_c(\tau)$ is an $(r,c)$ representation of $\boldsymbol{G}=\boldsymbol{G}_{rc}$ by \cite[Proposition~15 and Theorem~17]{Kaplan2023}. In particular $\Theta_{r,m}$ is an $(r,1)$ representation, so that it admits a unique Whittaker model.

\begin{theorem}\label{thm:covk1}
Let $\rho_c(\tau)$ be the representation \eqref{eq:covrhoc}. Then $\dim\Hom_N(\rho_c(\tau),\Psn)=1$. Consequently Theorem~\ref{thm:covmain} holds for $\rho_c(\tau)$.
\end{theorem}
\begin{proof}
Here $L=M_{(r^c)}$ and $V=V_{(r^c)}$. For the proof we can assume that $\chi=1$. By \cite[(2.10)]{Kaplan2023} (where the unnormalized Jacquet functor was used, as we do here), for $0\le l\le d$,
\begin{equation}\label{eq:covjacquet}
J_{V_{(d-l,l)}}(\Theta_{d,m})=\delta_{P_{(d-l,l)}}^{(r-1)/(2r)}\big(\Theta_{d-l,m}\otimes\Theta_{l,m}\big),
\end{equation}
where the tensor product is given by the standard construction (see, e.g., \cite{BZ76}) because for $\boldsymbol{G}$,
the preimages of the direct factors of $M_{\underline{n}}$ commute in $\boldsymbol{M_{\underline{n}}}$
(see \cite[\S2.7]{Kaplan2025}). Repeatedly applying \eqref{eq:covjacquet} we obtain
\begin{equation*}
J_{V}(\Theta_{rc,m})=\delta\cdot\big(\Theta_{r,m}\otimes\dots\otimes\Theta_{r,m}\big),
\end{equation*}
with $c$ factors, where $\delta$ is a product of the modulus characters appearing in \eqref{eq:covjacquet}. Hence
\[
J_{N,\Psn}(\Theta_{rc,m})=J_{N_L,\psi_L}\circ J_V(\Theta_{rc,m})=\delta\bigotimes_{i=1}^{c}J_{N_r,\psi_r}(\Theta_{r,m}),
\]
where $\psi_r$ is the standard generic character of $N_r$ (see \S~\ref{ssec:zeldata}). Each factor on the right-hand side is one-dimensional, therefore $\dim\Hom_N(\rho_c(\tau),\Psn)=1$ (which is stronger than \eqref{eq:covzel}).
\end{proof}

\begin{remark}\label{rem:covboundednessD}
For the representations $\rho_c(\tau)=\chi\Theta_{rc,m}$ above ($k_0=1$) the analog of Lemma~\ref{lem:boundednessD} holds, because the only dependence on $\tau$ is through the quasi-character $\chi$ and we can take the bound $A_\tau=\abs{\Re(\chi)}$. 
When $\chi$ varies over the local components of a quasi-character of $\bF^\times\backslash\bA^\times$, the constant $D$ for $\rho_c(\tau)$ is bounded uniformly over the finite places.
\end{remark}

\begin{remark}\label{rem:covglobal}
The global argument of \S~\ref{sec:global} also extends, as follows. Automorphic forms on finite central extensions are of moderate growth (\cite[I.2.17]{MW95} included these).
The proof of Theorem~\ref{thm:globalgauge} then applies to the component at a finite place $v$ of a genuine irreducible automorphic representation $\mathcal E$ of a central extension of $\GL_{kc}(\bA)$ by $\mu_m$, granted that
the global $(k,c)$ coefficient \eqref{eq:globalcoefficient} does not vanish on $\mathcal E$, and that the component $\mathcal E_v$ satisfies \eqref{eq:covuniq}. The argument does not require \eqref{eq:covzel}. When $\bF$ contains a primitive $2m$-th root of unity, representations $\mathcal E$ with these properties were constructed for $k_0=1$, with
\eqref{eq:covuniq} holding at all places (\cite[Proposition~52]{Kaplan2025}, \cite[\S2.2, Proposition~15 and Theorem~17]{Kaplan2023}).
\end{remark}

\section{Applications: absolute convergence of integrals of \texorpdfstring{$(k,c)$}{(k,c)} model functions}\label{sec:applications}

In this section we apply Theorems~\ref{thm:mgpsi} and \ref{thm:gaugepsi} to prove the absolute convergence, in
explicit cones and half-planes, of four integrals whose integrands are values of $(k,c)$
model functions. The first application, in \S~\ref{ssec:twomodels}, concerns a new family of integrals which extends the integrals of
\cite{LM20} (and \cite{JPSS83}), but as opposed to those works we merely prove convergence.

In the following subsections and in Appendix~\ref{sec:GLdoubling} we revisit several integrals
from \cite{CFGK2,CFK22,CFGK23} that are fundamental to the generalized doubling method, and prove their convergence in detail. It is straightforward to adapt the proofs of
this section and Appendix~\ref{sec:GLdoubling}, using \S~\ref{sec:coverings} and Theorem~\ref{thm:covk1}, to the generalized doubling method for covering groups (\cite{Kaplan2025}),
in the cases where the construction there is unconditional\footnote{One can require the conjectural $(k,c)=(rk_0,c)$ representations to have Zelevinsky models; with
this additional requirement convergence will again follow as it does here.}; indeed, the absolute value of a genuine function factors through the projection $p$. 

While our focus is $k>1$, all $k\geq1$ are treated, and uniformly so except in the appendix, where the case $k=1$ requires a separate argument. For $k=1$ the convergence of the
doubling integrals was proved by Lapid and Rallis \cite{LR} using multiplicativity arguments. Our proof is more elementary in nature.

The scheme of most of the proofs is the one Soudry used for the unipotent integrations appearing in the local Rankin--Selberg integrals for
$\SO_{2l+1}\times\GL_m$ \cite[Lemma~4.5]{Sou93}. It has three steps: a gauge bounds the model function on the Levi
part, uniformly in the unipotent part and in the maximal compact subgroup; the Iwasawa
decomposition expresses the Levi part of a unipotent element through its coordinates; and
the two combine into an integral of a power of those coordinates. The first step is
supplied by \eqref{eq:thegauge2psi} for the integrals of \S~\ref{ssec:twomodels}, and by Theorem~\ref{thm:mgpsi} for
the remaining integrals, in the same role as in \cite[Lemma~4.4]{Sou93}.

Throughout, $\rho$ is an arbitrary $(k,c)$ representation of $\GL_{kc}$ (in \cite{CFGK2,CFK22,CFGK23,CFK-func} only the case $\rho=\rho_c(\tau)$ was relevant);
for a $(k,c)$ representation $\rho$ we denote its model by $W_{\psi}(\rho)$ (see \S~\ref{ssec:kcdata}).
This means, in particular, that the character $\psi_k$ is defined by \eqref{eq:psikdef} with respect to $\psi$.
We write $D(\rho)$ for the constant $D$ of Theorems~\ref{thm:mgpsi} and \ref{thm:gaugepsi}.

Recall the character $\eta_\rho$ of $\GL_c$ attached to a
$(k,c)$ representation $\rho$ (see \eqref{eq:etarho}), and let $e_\rho\in\bR$ be the real part of $\eta_\rho$, i.e., $\abs{\eta_\rho(g)}=\abs{\det g}^{e_\rho}$ for $g\in\GL_c$.

For $k=1$, we have $\rho=\eta\circ\det=\eta_\rho$ for a quasi-character $\eta$ of $F^\times$ (see \S~\ref{ssec:kcdata}) and we can take
$D(\eta\circ\det):=c\abs{e_\rho}$, because $\abs{\det g}^{\pm1}\le\norm g^{c}$.

Recall the function $\vartheta$ of \eqref{eq:Nx}. The unipotent groups over which we integrate are identified with affine spaces, and
we denote their dimensions by the letter $\fd$, with a subscript indicating the group. For an integer $\fd\ge1$ let $F^{\fd}=F\times\dots\times F$ ($\fd$ times).
We use the elementary integral
\begin{equation}\label{eq:Iea}
I_{\fd}(a):=\int_{F^{\fd}}\vartheta(x)^{-a}\,dx\ <\ \infty,\qquad a>\fd ,
\end{equation}
where $\vartheta(x)=\max(1,\abs{x_1},\dots,\abs{x_{\fd}})$ for $x=(x_1,\dots,x_{\fd})$: indeed
$\vartheta(x)^{-a}\le\prod_{j=1}^{\fd}\max(1,\abs{x_j})^{-a/\fd}$, and
$\int_F\max(1,\abs x)^{-b}\,dx=q^{e_\psi/2}\big(1+(1-q^{-1})\sum_{m\ge1}q^{m(1-b)}\big)$ is finite for $b>1$.
For fixed $\fd$, $I_{\fd}(a)$ decreases as $a$ increases, since $\vartheta\ge1$. When $\fd=0$ we formally define $I_{\fd}(a)$ to be the constant $1$.

\begin{remark}\label{rem:uniformity}
In the applications which motivate this section, $\rho=\rho_c(\tau)$ for $\tau\in\Irr(\GL_k)$ generic,
and in the local--global setting $\tau$ is a local component of an irreducible cuspidal automorphic representation.
In this case the real parts of the cuspidal exponents of $\tau$ are bounded, in absolute value, by a constant $A_\tau$ which does not depend on the place, and
by Theorem~\ref{thm:mgpsi}\ref{it:mgpsi2} and Theorem~\ref{thm:gaugepsi}\ref{it:gaugepsi3}, the constant $D(\rho)$ may be taken to be bounded by a function of $A_\tau$, $k$ and $c$.
Moreover $\eta_\rho(g)=\omega_\tau((\det g)I_k)$ by \eqref{eq:etarhoc}, and $\omega_\tau$ is a cuspidal exponent of $\tau$, so that $\abs{e_\rho}\le A_\tau$. We make no attempt to render the constants optimal: all that we retain of them is that the cone, or half-plane, of convergence
depends on $\rho$ only through $k$, $c$, $D(\rho)$ and $e_\rho$, and, in the
case of the doubling integral, also on an irreducible representation $\pi$ but only through a bound for the real parts of its cuspidal exponents (see Lemma~\ref{lem:mcbound}).
This is important for the local--global arguments of \cite{CFGK2,CFK22} involving an Euler product of infinitely many local integrals.
\end{remark}
In this section and in the appendix, $G$ and $H$ denote split reductive groups (over $F$), defined explicitly at the beginning of each subsection.

\subsection{Rankin--Selberg integrals for \texorpdfstring{$(k,c)$}{(k,c)} representations}\label{ssec:twomodels}
Let $k'\geq k\geq 1$ and $c\geq1$, $G:=\GL_{kc}$, $K_G=K_{kc}$, $M=M_{(c^k)}$, $U=V_{(c^k)}$, and $H:=\GL_{k'c}$, $K_H=K_{k'c}$.
Let $\rho$ be a $(k,c)$ representation of $G$ and let $\rho'$
be a $(k',c)$ representation of $H$. We let $\mathfrak{e}:G\hookrightarrow  H$ be the embedding
\[
\mathfrak{e}(g)=\diag\big(g,I_{(k'-k)c}\big).
\]
Let $\pr(g)\in\Mat_{c\times kc}$ denote the last $c$ rows of
$g\in G$.

We define the following families of integrals.

If $k<k'$, then for any $W\in W_{\psi}(\rho)$ and $W'\in W_{\psi^{-1}}(\rho')$,
\begin{equation}\label{eq:twomodelnophi}
Z(s,W,W')=\int_{U\backslash G}W(g)\,W'(\mathfrak{e}(g))\,\abs{\det g}^{s}\,dg.
\end{equation}
When $k=k'$, for any $W$ and $W'$ as above and also $\Phi\in\mathcal{S}(\Mat_{c\times kc})$,
\begin{equation}\label{eq:twomodel}
Z(s,W,W',\Phi)=\int_{U\backslash G}W(g)\,W'(\mathfrak{e}(g))\,\Phi(\pr(g))\,\abs{\det g}^{s}\,dg.
\end{equation}

The integrals are well defined: since for $u\in U$, $\pr(u)=(0,\dots,0,I_c)$, $\pr(ug)=\pr(g)$; moreover $\psi_{k'}^{-1}(\mathfrak{e}(u))=\psi_k^{-1}(u)$ by \eqref{eq:psikdef},
because the blocks of $\mathfrak{e}(u)$ in the positions $(i,i+1)$ for $i\geq k$ vanish, and thus
\[W(ug)W'(\mathfrak{e}(ug))=\psi_k(u)\psi_{k'}^{-1}(\mathfrak{e}(u))W(g)W'(\mathfrak{e}(g)) =W(g)W'(\mathfrak{e}(g)).\]

For $c=1$ these are the Rankin--Selberg integrals of Jacquet, Piatetski-Shapiro and Shalika
\cite{JPSS83}, and for $c>1$ and $k=k'$  they are the integrals of Lapid and Mao \cite[\S5]{LM20}.

\begin{theorem}\label{thm:twomodelconv}
There is a real number $R$, depending only on the representations $\rho$, $\rho'$ and on $c$, such that for
every $s$ with $\Re(s)>R$ and every $W$, $W'$ and $\Phi$, the integrals \eqref{eq:twomodelnophi} and \eqref{eq:twomodel} are absolutely convergent.
\end{theorem}
\begin{proof}
We use the following integration formula, which follows from the Iwasawa decomposition $G=UMK_G$:
for every nonnegative measurable function $\varphi$ on $U\backslash G$,
\begin{equation*}
\int_{U\backslash G}\varphi(g)\,dg=\int_{K_G}\int_M\delta_P^{-1}(m)\,\varphi(m\kappa)\,dm\,d\kappa.
\end{equation*}

Write $m=\diag(m_1,\dots,m_k)$, then $\delta_P^{-1}(m)=\prod_{i=1}^{k}\abs{\det m_i}^{c(2i-k-1)}$ and $\mathfrak{e}(m)=\diag(m,I_{(k'-k)c})$. By \eqref{eq:thegauge2psi} applied to $W$ and $W'$,
with $D=D(\rho)$ and $D'=D(\rho')$, there are $\phi_1,\dots,\phi_{k-1}\in\mathcal{S}(\Mat_c)$ and $\phi'_1,\dots,\phi'_{k-1}\in\mathcal{S}(\Mat_c)$ such that
\begin{equation}\label{eq:tm1}
\abs{W(m\kappa)} \le \abs{\det m_k}^{e_\rho}\prod_{i=1}^{k-1}\abs{\det(m_k^{-1}m_i)}^{-D}\,\phi_i(m_{i+1}^{-1}m_i),
\end{equation}
and, noticing that $\mathfrak{e}(\kappa)\in K_H$,
\begin{equation}\label{eq:tm2}
\abs{W'(\mathfrak{e}(m\kappa))}\le\begin{dcases}
\prod_{i=1}^{k}\abs{\det m_i}^{-D'}\,
\prod_{i=1}^{k-1}\phi'_i(m_{i+1}^{-1}m_i)\,\phi'_k(m_k)&k<k',\\
\abs{\det m_k}^{e_{\rho'}}\prod_{i=1}^{k-1}\abs{\det(m_k^{-1}m_i)}^{-D'}\,\phi'_i(m_{i+1}^{-1}m_i)
&k=k'.\end{dcases}
\end{equation}
In \eqref{eq:tm2}, when $k'>k+1$, the remaining $k'-k-1$ functions $\phi'_i$ are evaluated at $I_c$ and contribute a constant, which we absorb into $\phi_1'$.

When $k<k'$ it is now clear that the integrand vanishes unless $m$ belongs to a bounded subset $M\cap\varpi^{-l}\Mat_{kc}(\Oo)$ for some $l>0$.
If $k=k'$, then since $\Phi(\pr(m\kappa))=\Phi((0,\dots,0,m_k)\kappa)$, we first deduce that
$(0,\dots,0,m_k)\kappa\in\varpi^{-l_0}\Mat_{c\times kc}(\Oo)$, and because $\kappa\in K$, obtain $(0,\dots,0,m_k)\in\varpi^{-l_0}\Mat_{c\times kc}(\Oo)$, hence $m_k\in\varpi^{-l_0}\Mat_c(\Oo)$. Then again by \eqref{eq:tm1}
$m\in M\cap\varpi^{-l}\Mat_{kc}(\Oo)$ for some $l>0$.

Therefore, the integral is bounded by a product of integrals
\begin{equation*}
\prod_{j=1}^{k}\,\int\limits_{m_j\in\GL_c\cap \varpi^{-l}\Mat_c(\Oo)}\abs{\det m_j}^{s+A_j}\,dm_j,
\end{equation*}
where $dm_j$ is the Haar measure of $\GL_c$,
\begin{align*}
&A_j=-D-D'+c(2j-k-1)\qquad j<k,\\
&A_k=\begin{cases}e_{\rho}+D(k-1)-D'+c(k-1)&k<k',\\e_{\rho}+e_{\rho'}+(D+D')(k-1)+c(k-1)&k=k'.\end{cases}
\end{align*}
By a result of Godement and Jacquet \cite[Proposition 1.1]{GJ} each integral in this product is convergent when
$\Re(s)>R$ where $R$ depends only on $\rho$ through $k,D$ and $e_{\rho}$, on $\rho'$ through $k',D'$ and $e_{\rho'}$ (when $k=k'$), and on $c$.
\end{proof}

\begin{remark}\label{rem:LMsharper}
For $k=k'$, Lapid and Mao obtained a sharper half-plane, together with the meromorphic continuation of the integral
\cite[Theorem~5.1]{LM20}. Their idea was to replace the $(k,c)$ functions with Zelevinsky model functions, then restrict to the Levi
part where these functions are Whittaker functions. The main point is that within the integral, the transition map $\mathcal{T}$ can be absorbed,
see \cite[Proposition 5.2]{LM20}.
\end{remark}

\subsection{Jacquet type integrals for \texorpdfstring{$(k,c)$}{(k,c)} functionals}\label{ssec:jacquetintegral}
Cai \textit{et al.} \cite[\S3.1]{CFGK23} described a Jacquet type integral realization for a $(k,c)$ functional, when the representation $\rho$
is induced from a family of $(k_i,c)$ representations $\rho_i$, $k_1+\dots+k_d=k$. (The discussion in \textit{loc. cit.} was limited to representations $\rho_c(\tau)$.)
In this subsection we prove the absolute convergence of this integral.

Assume that $\rho=\rho_1\times\dots\times\rho_d$, parabolic induction from $P_{\underline{k}c}$, $\underline{k}c=(k_1c,\dots,k_dc)$, to $\GL_{kc}$, where $\rho_i$ is a $(k_i,c)$ representation of $\GL_{k_ic}$. The representation
$W_{\psi}(\rho_1)\times\dots\times W_{\psi}(\rho_d)$, induced from the models, is a quotient of
$\rho$. As observed in \textit{loc. cit.}, in order to ensure convergence of the Jacquet integral we need to incorporate a twist in the inducing data.
Let $\zeta\in\bC^d$ and consider the representation $\Pi(\zeta)=\abs{\det}^{\zeta_1}W_{\psi}(\rho_1)\times\dots\times \abs{\det}^{\zeta_d}W_{\psi}(\rho_d)$. Then $\Pi(0)$ is a quotient of $\rho$.
Denote the space of $\Pi(\zeta)$ by $\cV_\zeta$. A vector $\xi\in \cV_\zeta$ is regarded as a complex-valued function by evaluating at the identity.

Let $\underline{k}'=(k_d,k_{d-1},\dots,k_1)$ and let $w\in\GL_{kc}$ denote the permutation matrix such that $wM_{\underline{k}'c}w^{-1}=M_{\underline{k}c}$ and $w(N\cap M_{\underline{k}'c})w^{-1}=N\cap M_{\underline{k}c}$.
The Jacquet integral (3.1) of \cite{CFGK23} is defined for each $\xi\in \cV_{\zeta}$ by
\begin{equation}\label{eq:jacquetintegral}
\int_{V_{\underline{k}'c}}\xi(wv)\,\psi_k^{-1}(v)\,dv.
\end{equation}
Formally, it is a $(k,c)$ functional on $\Pi(\zeta)$, and
\cite{CFGK23} asserted its absolute convergence for $\Re(\zeta_i-\zeta_{i+1})\gg0$, with a
reference to \cite[Lemma~2.1]{Sou00}. In this subsection we prove this result.

\begin{theorem}\label{thm:jacquetconv}
There is a real number $R$, depending only on $\rho_1,\ldots,\rho_d$ and $c$, such that if $\Re(\zeta_i-\zeta_{i+1})>R$ for all $1\le i< d$, then
the integral \eqref{eq:jacquetintegral} converges absolutely for every $\xi\in \cV_{\zeta}$.
\end{theorem}

\begin{proof}
By the definition of $w$ we have $wV_{\underline{k}'c}w^{-1}=\bV_{\underline{k} c}$. Thus
it suffices to prove that
\[\int\limits_{\bV_{\underline{k}c}}\abs{\xi(\bv)}\,d\bv<\infty,\qquad\forall \xi\in \cV_{\zeta}.
\]
Since $\xi$ is smooth, the functions $M_{\underline{k}c}\to\bC$ defined by
\[m\mapsto\delta_{P_{\underline{k}c}}^{-1/2}(m)\prod_i\abs{\det m_i}^{-\zeta_i}\xi(m\kappa)\] form a finite subset of
$\otimes_{i=1}^dW_{\psi}(\rho_i)$ when $\kappa$ varies in $K_{kc}$. Each such function is a sum of functions $m\mapsto\prod_{i=1}^dW_i(m)$ with $W_i\in W_{\psi}(\rho_i)$.
Let $D=D(\rho_1)+\dots+D(\rho_d)$. Applying Theorem~\ref{thm:mgpsi} to each of the finitely many model functions $W_i$ which occur,
and using $\norm{m_i}\le\norm m$ for $m=\diag(m_1,\dots,m_d)$, there is a constant $C_\xi\ge0$ such that
\begin{equation}\label{eq:Wkappabound}
\abs{\delta_{P_{\underline{k}c}}^{-1/2}(m)\prod_i\abs{\det m_i}^{-\zeta_i}\xi(m\kappa)}\ \le\ C_\xi\,\norm m^{D}\qquad\forall\kappa\in K,\ m\in M_{\underline{k}c}.
\end{equation}

Let $\bv\in\bV_{\underline{k}c}$ and write $\bv=um\kappa$ with $u\in V_{\underline{k}c}$, $m=\diag(m_1,\dots,m_d)\in M_{\underline{k}c}$ and $\kappa\in K$.
By the definition of the induced representation $\Pi(\zeta)$, the computation of
$\delta_{P_{\underline{k}c}}^{1/2}$ and \eqref{eq:Wkappabound},
\[
\abs{\xi(\bv)}\ \le\ C_\xi\,\norm m^{D}\prod_{i=1}^{d}\abs{\det m_i}^{x_i},
\qquad x_i:=\Re\zeta_i+\tfrac c2\Big(\sum_{j>i}k_j-\sum_{j<i}k_j\Big),
\]
and $x_i-x_{i+1}=\Re(\zeta_{i}-\zeta_{i+1})+\tfrac c2(k_i+k_{i+1})$.
Let $\fd:=\fd_{\underline{k}c}=c^2\sum_{i<j}k_ik_j$, as in \eqref{eq:dimVbar}, and
\[R:=\fd+(kc-1)D.\]
Put $r:=\min_{1\le i<d}(x_i-x_{i+1})$. Since $k_{i}+k_{i+1}>0$, the assumption $\Re(\zeta_i-\zeta_{i+1})>R$ gives $r>R$,
so by Corollary~\ref{cor:iwasawaunip}\ref{it:ci1} and \ref{it:ci3},
\[
\abs{\xi(\bv)}\ \le\ C_\xi\,\vartheta(\bv)^{(kc-1)D}\,\vartheta(\bv)^{-r}=C_\xi\,\vartheta(\bv)^{-a},\qquad a:=r-(kc-1)D>\fd .
\]
Hence
$\int_{\bV_{\underline{k}c}}\abs{\xi(\bv)}\,d\bv\le C_\xi\,I_{\fd}(a)$, which is finite by
\eqref{eq:Iea}, because $\fd<a$.
\end{proof}

\subsection{The functional \texorpdfstring{$\lambda(s,c,\tau,\psi)$}{lambda(s,c,tau,psi)}}\label{ssec:lambdaconv}
We use the notation of \cite[\S2--\S3]{CFK22}, which we recall to the extent needed. Let
$H$ be $\Sp_{2kc}$ or $\SO_{2kc}$, realized in $\GL_{2kc}$ with respect to the
standard antidiagonal forms of \cite{CFK22}, and let $\epsilon=1$ for $H=\Sp_{2kc}$ and
$\epsilon=-1$ for $H=\SO_{2kc}$. Fix the Borel subgroup $B_H=H\cap B_{2kc}$. Let $P=M_P\ltimes U_P$ be the Siegel parabolic subgroup
of $H$, with
\begin{gather*}
M_P=\{m(a)=\diag(a,a^*):a\in\GL_{kc}\},\qquad a^*=J_{kc}\,{}^ta^{-1}J_{kc},\\
U_P=\Big\{u(Y)=\begin{pmatrix}I_{kc}&Y\\&I_{kc}\end{pmatrix}\in H\Big\},\qquad
\overline{U}_P=\Big\{v(X)=\begin{pmatrix}I_{kc}&\\X&I_{kc}\end{pmatrix}\in H\Big\},
\end{gather*}
where for any $l\geq1$, $J_{l}\in\GL_l$ is the permutation matrix such that $(J_l)_{i,l-i+1}=1$ for each $i$, ${}^ta$ denotes the transpose of $a$, and $X$ ranges over the
$F$-subspace $\fX$ of $\Mat_{kc}$ of dimension
\[
\fd_P:=\dim U_P=kc(kc+\epsilon)/2
\]
defined by the condition that $J_{kc}X$ be symmetric ($\epsilon=1$) or antisymmetric ($\epsilon=-1$).
A set of $\fd_P$ entries of $X$, for instance the entries on and above the antidiagonal (above it only, when $\epsilon=-1$),
forms a system of linear coordinates on $\fX$, and each of the remaining entries is $\pm$ one of these coordinates ($0$ on the antidiagonal when $\epsilon=-1$); thus
$\vartheta(X)$ of \eqref{eq:Nx}, the maximum of $1$ and of the absolute values of all the entries of $X$, is also the maximum of $1$ and of the absolute values of the coordinates.
The Haar measure $dv$ of $\overline{U}_P$ is the product of the measures $dx$ over the coordinates.

Let $K_H=H\cap K_{2kc}$; then $H=U_PM_PK_H$ is the Iwasawa decomposition of $H$ with respect to $P$.
We apply Corollary~\ref{cor:iwasawaunip} to $H$, as in Remark~\ref{rem:iwasawaH}. Let
$v=v(X)\in \overline{U}_P$ and write $v=u\,m(a)\,\kappa$ with $u\in U_P$, $a\in\GL_{kc}$ and
$\kappa\in K_H$. Since $U_P\subset V_{(kc,kc)}$, $m(a)=\diag(a,a^*)\in M_{(kc,kc)}$ and
$K_H\subset K_{2kc}$, this is the case of Example~\ref{ex:twoblock} with $n_1=kc$, $m_1=a$ and $m_2=a^*$,
and $\abs{\det a^*}=\abs{\det a}^{-1}$, $\norm{a^*}=\norm a$ (\S~\ref{ssec:height}, property~\ref{it:h1}). Thus \eqref{eq:twoblockbounds} reads
\begin{equation}\label{eq:siegelbounds}
\abs{\det a}\le \vartheta(X)^{-1},\qquad \norm a\le \vartheta(X)^{2kc-1}.
\end{equation}

Let $\rho$ be a $(k,c)$ representation of $\GL_{kc}$. For $s\in\bC$, consider the induced representation $\Ind_P^H(W_{\psi}(\rho)|\det|^{s-1/2})$. Functions in this space take values in $W_{\psi}(\rho)$, and after evaluating at $I_{kc}$ we obtain\footnote{The map $\varphi\mapsto (h\mapsto \varphi(h)(I_{kc}))$ is an isomorphism from the induced space onto $V(s,\rho)$:
for $f\in V(s,\rho)$ the function $W_h$ in \eqref{eq:fIwasawa} is unique and the inverse map sends $f$ to the function $\varphi\colon H\to W_\psi(\rho)$, $\varphi(h)=W_h$.}
the space $V(s,\rho)$ of the smooth functions $f\colon H\to\bC$ such that for each $h\in H$, there is $W_h\in W_\psi(\rho)$ such that
\begin{equation}\label{eq:fIwasawa}
f(s,um(a)h)=\delta_P^{1/2}(m(a))\,\abs{\det a}^{s-1/2}W_h(a),\qquad a\in\GL_{kc},\ u\in U_P .
\end{equation}
Note that $\delta_P(m(a))=\abs{\det a}^{kc+\epsilon}$.

\begin{remark}\label{rem:sections}
The two-variable notation is suggestive: we call a function $f:\bC\times H\to \bC$ an entire section of $V(\rho)$ if for each $s\in\bC$, $f(s,\cdot)\in V(s,\rho)$,
and for each $h\in H$, the function $s\mapsto f(s,h)$ is a Laurent polynomial in $q^{-s}$. The definition implies that there is a compact open subgroup $K_f$ of $H$ such that $f(s,h\kappa)=f(s,h)$ for every $(s,h)$ and $\kappa\in K_f$. Moreover, for an entire section $f$ and $h\in H$, the functions $W_{\kappa h}$ of \eqref{eq:fIwasawa}, $\kappa\in K_H$, vary in a finite-dimensional subspace of $W_\psi(\rho)$ independent of $s$, and their coefficients with respect to a basis of this subspace are Laurent polynomials in $q^{-s}$. An entire section $f$ such that $f(s,K_H)$ is independent of $s$ is called a standard section. See \cite{CFK22}.
\end{remark}

Let $P'$ be the conjugate of $P$ by $\diag(I_{kc-1},\left(\begin{smallmatrix}&1\\1\end{smallmatrix}\right),I_{kc-1})^{kc}$, a parabolic subgroup, and let $U_{P'}$ denote its unipotent radical; when $kc$ is even, $P'=P$ and $U_{P'}=U_P$.
Let $\delta_0\in H$ be a representative for a Weyl group element such that
\begin{equation}\label{eq:delta0}
{}^{\delta_0}U_{P'}=\overline{U}_P.
\end{equation}
Let $\psi_{U_{P'}}$ be the character of $U_{P'}$ defined in \cite[(3.3)]{CFK22}.
We fix the Haar measure $du$ on $U_{P'}$ so that the isomorphism $u\mapsto{}^{\delta_0}u$ of \eqref{eq:delta0}
carries $du$ to $dv$. For $f\in V(s,\rho)$ put
\begin{equation}\label{eq:lambda}
\lambda(s,c,\rho,\psi)f=\int_{U_{P'}}f(s,\delta_0u)\,\psi_{U_{P'}}^{-1}(u)\,du .
\end{equation}
Similar constructions for $H=\GSpin_{2kc}$ and $H=\GL_{kc}$ were given in \cite[(3.3)]{CFK22}.
For $\rho=\rho_c(\tau)$ this is the functional $\lambda(s,c,\tau,\psi)$ of \textit{ibid.}

\begin{theorem}\label{thm:lambdaconv}
Let $\rho$ be a $(k,c)$ representation of $\GL_{kc}$ and $D=D(\rho)$. There is a real number $R$,
depending only on $k$, $c$, $H$ and $D$, such that for every $s$ with $\Re(s)>R$ and every
$f\in V(s,\rho)$, the integral $\int_{U_{P'}}\abs{f(s,\delta_0u)}\,du$ is
finite. In particular \eqref{eq:lambda} converges absolutely for $\Re(s)>R$.
\end{theorem}

\begin{proof}
We prove the assertion with
\begin{equation}\label{eq:R0}
R:=\tfrac12-\tfrac{kc+\epsilon}2+(2kc-1)D+\fd_P ,
\end{equation}
and in the more precise form
\begin{equation}\label{eq:lambdabound}
\int_{U_{P'}}\abs{f(s,\delta_0u)}\,du\ \le\ C_f\,I_{\fd_P}\big(\Re(s)-R+\fd_P\big),
\end{equation}
where $C_f$ is a constant depending on $f(s,\cdot)|_{K_H\delta_0}$ (the restriction of $f(s,\cdot)$ to $K_H\delta_0$)
only, and $I_{\fd_P}$ is the integral \eqref{eq:Iea}.

By \eqref{eq:delta0},
it suffices to prove the bound for $\int_{\overline{U}_P}\abs{f(s,v)}\,dv$ (every $f$) with $C_f$ depending on $f(s,\cdot)|_{K_H}$.

Let $v=v(X)\in \overline{U}_P$ and write $v=u\,m(a)\kappa$ as in \eqref{eq:siegelbounds}. By
\eqref{eq:fIwasawa},
\[
f(s,v)=\delta_P^{1/2}(m(a))\,\abs{\det a}^{s-1/2}W_{\kappa}(a).
\]
Since $f(s,\cdot)$ is smooth (independently of $s$ when $f$ is an entire section), the set $\{W_{\kappa}:\kappa\in K_H\}$ 
with $W_\kappa$ as in \eqref{eq:fIwasawa} is a finite subset of
$W_{\psi}(\rho)$, hence by Theorem~\ref{thm:mgpsi} there is $C_f\ge0$ with
$\abs{W_{\kappa}(a)}\le C_f\norm a^{D}$ for all $\kappa\in K_H$ and $a\in\GL_{kc}$. Put
\begin{equation}\label{eq:sigmas}
\sigma_s:=\Re(s)-\tfrac12+\tfrac{kc+\epsilon}2,
\end{equation}
so that $\delta_P^{1/2}(m(a))\abs{\det a}^{\Re(s)-1/2}=\abs{\det a}^{\sigma_s}$; the hypothesis
$\Re(s)>R$ gives $\sigma_s>(2kc-1)D+\fd_P>0$. By \eqref{eq:siegelbounds},
\[
\abs{f(s,v(X))}\ \le\ C_f\,\norm a^{D}\,\abs{\det a}^{\sigma_s}\ \le\ C_f\,\vartheta(X)^{(2kc-1)D}\,\vartheta(X)^{-\sigma_s}
=C_f\,\vartheta(X)^{-(\sigma_s-(2kc-1)D)} ,
\]
and integrating over the $\fd_P$ coordinates of $X$ gives \eqref{eq:lambdabound}, where $I_{\fd_P}(\Re(s)-R+\fd_P)$ is finite by \eqref{eq:Iea} because
$\Re(s)-R+\fd_P=\sigma_s-(2kc-1)D>\fd_P$.
\end{proof}

\begin{remark}\label{rem:lambdaremarks}
The following complements are immediate from the proof.
\begin{enumerate}[leftmargin=*,label=(\arabic*),ref=(\arabic*)]
\item\label{it:lr1} The intertwining operator \cite[(3.2)]{CFK22} was defined by an integral similar to
\eqref{eq:lambda}, without the character $\psi_{U_{P'}}$, and with $\delta_0$ replaced by another representative for a Weyl group element,
which still satisfies \eqref{eq:delta0}. The proof implies that this intertwining operator is also absolutely convergent in $\Re(s)>R$.

\item\label{it:lr2} Let $f$ be an entire section of $V(\rho)$ (see Remark~\ref{rem:sections}).
Then $s\mapsto\lambda(s,c,\rho,\psi)f$ is holomorphic on $\Re(s)>R$. Indeed, the constant $C_f$ of \eqref{eq:lambdabound} is bounded on compact subsets of $\bC$,
so by \eqref{eq:lambdabound} the integral $\int_{U_{P'}}\abs{f(s,\delta_0u)}\,du$ is bounded on compact subsets of $\Re(s)>R$,
and the holomorphy is now well known, see, e.g., \cite{Mat01}.
\item\label{it:lr3} For $H=\GSpin_{2kc}$, the model $W_{\psi}(\rho)$ is twisted by a
quasi-character $\chi$ of $F^\times$ to form a representation of $M_P$; this multiplies the
integrand by $\abs{\chi}$ evaluated at the $\GL_1$-part, which is bounded by a power of
$\norm a$, and shifts $R$ accordingly.
\end{enumerate}
\end{remark}
\begin{remark}\label{rem:lambdaGL}
Let $k_1,k_2\geq1$ and $k=k_1+k_2$.
For $H=\GL_{kc}$ and $P=P_{(k_1c,k_2c)}$, the case
(3) of \cite[\S3]{CFK22}, the representation is induced from a tensor product of $(k_1,c)$ and $(k_2,c)$ models $W_{\psi}(\rho_1)$ and $W_{\psi}(\rho_2)$, and
$\overline{U}_P=\bV_{(k_1c,k_2c)}$; in this case the argument of the proof of Theorem~\ref{thm:jacquetconv} with $d=2$ immediately implies
the absolute convergence of the functional \eqref{eq:lambda} and the intertwining operator \cite[(3.2)]{CFK22} in a half-plane in $s$.
\end{remark}

\subsection{The doubling integral}\label{ssec:doubling}
We turn to the (local) integral $Z(s,\omega,f)$ of the generalized doubling method
\cite[\S2.4]{CFK22}, and fill in the details of its absolute convergence, based on the sketch of \cite[Proposition~2.5]{CFK22}.

\subsubsection{The setting}\label{sssec:doublingsetting}
We follow \cite[\S3.3]{CFK-func}, to which we refer for the general definitions. 
The groups $G$ of the generalized doubling method are $\Sp_c$, $\SO_c$, $\GSpin_c$ and $\GL_c$ (\cite[\S~2]{CFK22}). We
carry out the computations for $G=\Sp_c$ with $c=2\ell$, and $k\ge1$. For $k>1$ this is the case described in
coordinates in \cite[Example~3.1]{CFK-func}; for $k=1$ the group $U$ below is trivial and $\psi_U=1$
(see \cite[\S3.3]{CFK-func}), and the formulas below remain valid with $(k-1)c=0$; the groups $\SO_c$ differ only in
the coordinates of the objects below, which \cite[\S2]{CFK22} listed. For $\GSpin_c$ the proof is similar, see Remark~\ref{rem:Gspin}; the case of $\GL_c$ is a bit more complicated and deferred to the appendix.

Let $\jmath_{2l}=\left(\begin{smallmatrix}&J_l\\-J_l&\end{smallmatrix}\right)$ with $J_l$ as defined in \S~\ref{ssec:lambdaconv}.
Let $G=\Sp_c$, realized in $\GL_c$ as the group of elements $g\in\GL_c$ such that ${}^tg\jmath_cg=\jmath_c$. 
Let $H=\Sp_{2kc}$, similarly realized in $\GL_{2kc}$ using $\jmath_{2kc}$.

Fix $B_G=G\cap B_c$, $K_G=G\cap K_c$, $B_H=H\cap B_{2kc}$ and $K_H=H\cap K_{2kc}$.

We define an embedding of $G\times G$ in $H$ in the stabilizer of a ``generic'' character.
Let $Q=M_Q\ltimes U_Q$ be the standard parabolic subgroup of $H$ with $M_Q\cong\GL_c^{k-1}\times H_0$, $H_0=\Sp_{2c}$
occupying the middle $2c\times 2c$ block, and $U=U_Q$ (from here on, $U$ no longer denotes the group $V_{(c^k)}$ of \S~\ref{ssec:kcdata}), with its character $\psi_U$ of
\cite[Example~3.1]{CFK-func}. The group $G\times G$ is embedded in $H$ in the stabilizer of $\psi_U$ in $M_Q$.
The embedding $\mathfrak e\colon G\times G\to H$ is given by
\begin{equation}\label{eq:embedding}
\mathfrak e(g_1,g_2)=\diag\Big(g_1,\dots,g_1,\begin{pmatrix}g_{1,1}&&g_{1,2}\\&g_2&\\g_{1,3}&&g_{1,4}\end{pmatrix},g_1^*,\dots,g_1^*\Big),
\qquad g_1=\begin{pmatrix}g_{1,1}&g_{1,2}\\g_{1,3}&g_{1,4}\end{pmatrix},
\end{equation}
with $g_{1,i}\in\Mat_\ell$, $k-1$ copies of $g_1$ and of $g_1^*$; we put $\mathfrak e_1(g)=\mathfrak e(g,I_c)$ and
$\mathfrak e_2(g)=\mathfrak e(I_c,g)$, so that $\mathfrak e_1(G)$ and $\mathfrak e_2(G)$ commute. Let $\iota$ be the involution of $G$ given by conjugation by $\left(\begin{smallmatrix}&I_{\ell}\\I_{\ell}\end{smallmatrix}\right)$.

As in \S~\ref{ssec:lambdaconv} let $P=M_P\ltimes U_P$ be the Siegel parabolic subgroup of $H$,
with $M_P=\{m(a):a\in\GL_{kc}\}$, and define the space $V(s,\rho)$.

Let
\begin{equation}\label{eq:delta01}
\delta_0=\begin{pmatrix}&I_{kc}\\-I_{kc}&\end{pmatrix}\in H,\qquad
\delta_1=\diag\Big(I_{(k-1)c},\begin{pmatrix}I_c&I_c\\&I_c\end{pmatrix},I_{(k-1)c}\Big)\in U_P\cap H_0,\qquad
\delta=\delta_0\delta_1.
\end{equation}
Note that ${}^{\delta_0}U_P=\overline{U}_P$. In addition for all $g\in G$ we have ${}^{\delta}\mathfrak e(g,{}^{\iota}g)\in M_P$,
and for $f\in V(s,\rho)$ and $h\in H$,
\begin{equation}\label{eq:diaginvariance}
f\big(s,{}^{\delta}\mathfrak e(g,{}^{\iota}g)\,h\big)=f(s,h).
\end{equation}

Let $U_0=U\cap U_P$. By \cite[\S3.3]{CFK-func} $U={}^{\delta^{-1}}m(V_{(c^k)})\rtimes U_0$, and $\psi_U({}^{\delta^{-1}}m(u))=\psi_k^{-1}(u)$ for
$u\in V_{(c^k)}$.

Let $\pi\in\Irr(G)$, let $\omega(g)=\langle\varphi,\pi^\vee(g)\varphi^\vee\rangle$
be a matrix coefficient of $\pi^\vee$ (the contragredient of $\pi$), with $\varphi$ in the space of $\pi$ and
$\varphi^\vee$ in that of $\pi^\vee$, and let $f\in V(s,\rho)$. The local integral is
\begin{equation}\label{eq:Zdef}
Z(s,\omega,f)=\int_G\int_{U_0}\omega(g)\,f\big(s,\delta u_0\,\mathfrak e_2({}^{\iota}g)\big)\,\psi_U(u_0)\,du_0\,dg .
\end{equation}
We will show that it is absolutely convergent, namely
\begin{equation*}
Z^{\mathrm{abs}}(s,\omega,f):=\int_G\int_{U_0}\abs{\omega(g)}\,\abs{f\big(s,\delta u_0\,\mathfrak e_2({}^{\iota}g)\big)}\,du_0\,dg <\infty.
\end{equation*}

\subsubsection{The passage from \texorpdfstring{$U_0$}{U0} to \texorpdfstring{$U_P$}{UP}}\label{sssec:U0toUP}
In this subsection we consider the inner $du_0$ integration of \eqref{eq:Zdef}. Recall from \S~\ref{ssec:lambdaconv} that $U_P=\{u(Y)\in H\}$, and as explained there for
$\overline{U}_P$, here also $Y\in\Mat_{kc}$ ranges over a $\fd_P$-dimensional $F$-subspace $\fY$, and $\vartheta(Y)$ is the maximum of $1$ and of the absolute values of the coordinates
on $\fY$. The subgroup $U_0$ of $U$ is defined by the elements $u(Y_0)\in H$, for $Y_0\in\Mat_{kc}$ whose bottom left $c\times c$ block is $0$;
$Y_0$ is defined by $\dim U_0$ coordinates,
\[\fd_{0}:=\dim U_0=\fd_P-\tfrac{c(c+1)}2.\]

Let $P_0=P\cap H_0$ be the Siegel parabolic subgroup of $H_0$; its unipotent radical is $U_1=U_P\cap M_Q=U_P\cap H_0$,
the subgroup of $u(Y_1)\in H_0$, with $Y_1\in\Mat_{kc}$
whose nonzero entries can only appear in the bottom left $c\times c$ block; $\dim U_1=\tfrac{c(c+1)}2$. Note that $\delta_1\in U_1$, in fact $\delta_1=u(\widehat{I}_c)$ where
$\widehat{I}_c=\left(\begin{smallmatrix}0&0\\I_c&0\end{smallmatrix}\right)$. Since $U_P$ is abelian, we have $U_P=U_0\times U_1$.

For an element $h\in H$ which normalizes $U_0$, let $\mathrm{mod}_{U_0}(h)$ denote the modular character of the automorphism $u_0\mapsto hu_0h^{-1}$ of $U_0$ defined by
\[\int\limits_{U_0}\varphi(h^{-1}u_0h)\,du_0=\mathrm{mod}_{U_0}(h)\int\limits_{U_0}\varphi(u_0)\,du_0.\]
Similarly define $\mathrm{mod}_{U}$; by definition $\delta_P(h)=\mathrm{mod}_{U_P}(h)$ for $h\in P$.

We start with the following convergence result.
\begin{lemma}\label{lem:U0Alone}
Let $\rho$ be a $(k,c)$ representation of $\GL_{kc}$ and $D=D(\rho)$. There is a real number $R_0$,
depending only on $k$, $c$ and $D$, such that for every $s$ with $\Re(s)>R_0$ and every
$f\in V(s,\rho)$,
\begin{equation*}
\int_{U_0}\abs{f(s,\delta u_0)}\,du_0\ <\ \infty .
\end{equation*}
In particular, the inner integral of \eqref{eq:Zdef} converges absolutely for $\Re(s)>R_0$ and every $g\in G$.
\end{lemma}
\begin{proof}
For $k=1$ the group $U_0$ is trivial and there is nothing to prove, so assume $k>1$. Then $\fd_{0}\ge1$. We prove the assertion with
\begin{equation}\label{eq:RU0}
R_0:=\tfrac12-\tfrac{kc+1}2+(2kc-1)D+\fd_{0},
\end{equation}
which is smaller by $\dim U_1$ than the constant $R$ of \eqref{eq:R0} (here $\epsilon=1$).

Let $u_0=u(Y_0)\in U_0$. By \eqref{eq:delta01},
\[\delta u_0=\delta_0\,u(\widehat{I}_c+Y_0)=v(X)\,\delta_0,\qquad X:=-(\widehat{I}_c+Y_0),\]
with $v(X)\in\overline{U}_P$. The entries of $X$ are, up to a sign, the entries of $Y_0$ outside the bottom left $c\times c$ block, and the entries of $I_c$ inside it. Therefore
\begin{equation}\label{eq:nuXnuy}
\vartheta(X)=\vartheta(Y_0),
\end{equation}
where $\vartheta(Y_0)$ is the maximum of $1$ and of the absolute values of the $\fd_{0}$ coordinates of $Y_0$, and
\[
f(s,\delta u_0)=f(s,v(X)\delta_0).
\]

Let $\sigma_s$ be given by \eqref{eq:sigmas} (with $\epsilon=1$). The assumption $\Re(s)>R_0$ implies $\sigma_s>(2kc-1)D+\fd_{0}>0$.
We write $v(X)=u\,m(a)\kappa$ as in \eqref{eq:siegelbounds}, apply \eqref{eq:fIwasawa} to $\delta_0\cdot f$, then proceed as in the proof of Theorem~\ref{thm:lambdaconv}:
apply Theorem~\ref{thm:mgpsi} and obtain, by \eqref{eq:siegelbounds} and \eqref{eq:nuXnuy} and since $\sigma_s>0$,
\[
\abs{f(s,\delta u_0)}\ \le\ C\,\norm a^{D}\,\abs{\det a}^{\sigma_s}\ \le\ C\,\vartheta(X)^{(2kc-1)D}\,\vartheta(X)^{-\sigma_s}=C\,\vartheta(Y_0)^{-(\sigma_s-(2kc-1)D)} .
\]

Integrating over the $\fd_{0}$ coordinates of $Y_0$, we obtain
\begin{equation*}
\int_{U_0}\abs{f(s,\delta u_0)}\,du_0\ \le\ C\,I_{\fd_{0}}\big(\Re(s)-R_0+\fd_{0}\big)<\infty ,
\end{equation*}
where the finiteness follows from \eqref{eq:Iea} because $\Re(s)-R_0+\fd_{0}=\sigma_s-(2kc-1)D>\fd_{0}$. The last assertion follows.
\end{proof}

With this result at hand we can consider the $du_0$ integral separately from \eqref{eq:Zdef}; it becomes a functional on $V(s,\rho)$ with the following
equivariance property \cite[Corollary~2.3]{CFK22}: for $g_0\in G$ and $h\in H$,
\[
\int_{U_0}f\big(s,\delta u_0\,\mathfrak e(g_0,{}^{\iota}g_0)\,h\big)\psi_U(u_0)\,du_0=\int_{U_0}f(s,\delta u_0h)\,\psi_U(u_0)\,du_0.
\]
This follows from the formal properties of the construction, and the same result holds when $f$ is replaced with its absolute value in the integral:
\begin{lemma}\label{lem:abscor23}
For every $f\in V(s,\rho)$, $g_0\in G$ and $h\in H$,
\begin{equation}\label{eq:abscor23}
\int_{U_0}\abs{f\big(s,\delta u_0\,\mathfrak e(g_0,{}^{\iota}g_0)\,h\big)}\,du_0
=\int_{U_0}\abs{f(s,\delta u_0h)}\,du_0 .
\end{equation}
\end{lemma}
\begin{proof}
The proof of \cite[Corollary~2.3]{CFK22} (see the proof of \cite[Proposition~2.2]{CFK22})
consists of three steps, each of which applies to $\abs f$. First, the function
$u\mapsto f(s,\delta uh)$ on $U$ satisfies $f(s,\delta vu_0h)=f(s,{}^{\delta}v\,\delta u_0h)=\psi_k(u_v)f(s,\delta u_0h)$
for $v={}^{\delta^{-1}}m(u_v)$ with $u_v\in V_{(c^k)}$, by the
$(V_{(c^k)},\psi_k)$-equivariance of the $(k,c)$ model in \eqref{eq:fIwasawa}; hence
$\abs{f(s,\delta uh)}$ is a left ${}^{\delta^{-1}}m(V_{(c^k)})$-invariant function on $U$, and the integral over
$U_0$ is the integral over ${}^{\delta^{-1}}m(V_{(c^k)})\backslash U$. Second, $\mathfrak e(g_0,{}^{\iota}g_0)$
normalizes $U$ and ${}^{\delta^{-1}}m(V_{(c^k)})$, and the change of variables $u\mapsto{}^{\mathfrak e(g_0,{}^{\iota}g_0)}u$
on ${}^{\delta^{-1}}m(V_{(c^k)})\backslash U$ multiplies the measure by a character of $G$ which is trivial for $G=\Sp_c$; third,
$\delta\,\mathfrak e(g_0,{}^{\iota}g_0)={}^{\delta}\mathfrak e(g_0,{}^{\iota}g_0)\,\delta$ and
\eqref{eq:diaginvariance} remove the element ${}^{\delta}\mathfrak e(g_0,{}^{\iota}g_0)$. Thus \eqref{eq:abscor23} follows.
\end{proof}

The proof of convergence is more convenient when we replace the subgroup $U_0$ in \eqref{eq:Zdef} by $U_P$.
\begin{lemma}\label{lem:slice}
Let $K_f$ be a compact open subgroup of $H$ such that $f(s,h\kappa)=f(s,h)$ for all
$h\in H$ and $\kappa\in K_f$, and put $K_{f,1}:=K_f\cap U_1$. Then for every $b\in H$ which normalizes $U_1$,
\begin{equation}\label{eq:slice}
\int_{U_0}\abs{f(s,\delta u_0b)}\,du_0\ \le\ \frac{\mathrm{mod}_{U_1}(b)^{-1}}{\vol(K_{f,1})}\int_{U_P}\abs{f(s,\delta_0ub)}\,du .
\end{equation}
\end{lemma}

\begin{proof}
Let $u_0\in U_0$ and $\kappa_1\in K_{f,1}$. Since $U_P$ is abelian and $b$ normalizes
$U_1$,
\[
f(s,\delta u_0b)
=f\big(s,\delta u_0\,b\,\kappa_1\big)
=f\big(s,\delta u_0\,({}^{b}\kappa_1)\,b\big)
=f\big(s,\delta_0\,(\delta_1\,{}^{b}\kappa_1)\,u_0\,b\big).
\]
Integrating over $\kappa_1\in K_{f,1}$ and $u_0\in U_0$, and substituting $u_1=\delta_1\,{}^{b}\kappa_1$,
which ranges over the coset $\delta_1\,{}^{b}K_{f,1}$ of volume $\mathrm{mod}_{U_1}(b)\vol(K_{f,1})$ in $U_1$,
\[
\mathrm{mod}_{U_1}(b)\vol(K_{f,1})\int_{U_0}\abs{f(s,\delta u_0b)}\,du_0
=\int_{\delta_1{}^{b}K_{f,1}}\int_{U_0}\abs{f(s,\delta_0u_1u_0b)}\,du_0\,du_1
\le\int_{U_P}\abs{f(s,\delta_0ub)}\,du .
\]
The last inequality follows from Tonelli's theorem.
\end{proof}

\begin{remark}[The archimedean case]\label{rem:slicearch}
Over an archimedean field there is no compact open subgroup fixing $f$, but we can obtain the analog of
\eqref{eq:slice} using the theorem of Dixmier and Malliavin
\cite{DM78}, as in \cite[\S7]{Sou95} and in \cite[Corollary~3.21]{CFK-func}; when $f$ is regarded as a section we use the refinement of that theorem which applies to sections, by Gourevitch \cite[Corollary A.3]{CFK22}. We assume, as in \cite[\S~1.1]{CFK22}, that $\rho$ (and thereby $W_{\psi}(\rho)$, with the quotient topology via $\xi\mapsto W_\xi$) is smooth admissible Fr\'{e}chet of moderate growth. Thus $V(s,\rho)$ (with the action by right translation) is a smooth Fr\'echet representation of the abelian Lie group $U_1$.
By \cite{DM78} there are finitely many $f_1,\dots,f_r\in V(s,\rho)$ and $\phi_1,\dots,\phi_r\in C_c^\infty(U_1)$ such that
\[
f(s,h)=\sum_{i=1}^{r}\int_{U_1}f_i(s,hu_1)\,\phi_i(u_1)\,du_1,\qquad h\in H .
\]
For $u_0\in U_0$, since $U_P$ is abelian, and $b$ by assumption normalizes $U_1$,
\[f_i(s,\delta u_0bu_1)=f_i(s,\delta_0(\delta_1{}^{b}u_1)u_0b).\]
After a change of variables
$u_1\mapsto{}^{b}u_1$ and translation by $\delta_1\in U_1$ we obtain
\[
\abs{f(s,\delta u_0b)}\ \le\ \mathrm{mod}_{U_1}(b)^{-1}\sum_{i=1}^{r}\norm{\phi_i}_\infty\int_{U_1}\abs{f_i(s,\delta_0u_1u_0b)}\,du_1 .
\]
Integrating over $u_0\in U_0$ we obtain the analog of \eqref{eq:slice}:
\begin{equation*}
\int_{U_0}\abs{f(s,\delta u_0b)}\,du_0\ \le\ \mathrm{mod}_{U_1}(b)^{-1}\sum_{i=1}^{r}\norm{\phi_i}_\infty\int_{U_P}\abs{f_i(s,\delta_0ub)}\,du ,
\end{equation*}
with the sections $f_i$ and the constants $\norm{\phi_i}_\infty$, which depend on $f$ only, in place of $f$.
\end{remark}

Now we can formulate the lemma which bounds the integral over $U_P$ of a right translate of $f$ by $b\in M_P$ with explicit dependence on $b$.
\begin{lemma}\label{lem:UPtorus}
Let $D=D(\rho)$ and let $\sigma_s=\Re(s)-\frac12+\frac{kc+1}2$ be as in \eqref{eq:sigmas}. There is a real number $R_1$, depending only on $k$, $c$ and $D$,
such that for every $s$ with $\Re(s)>R_1$ and every $f\in V(s,\rho)$ there is a constant
$C_f\ge0$ such that the following holds.
Let $b\in M_P$ and let $a_b\in\GL_{kc}$ be defined by ${}^{\delta_0}b=m(a_b)$. Then
\begin{equation}\label{eq:UPtorusbound}
\int_{U_P}\abs{f(s,\delta_0ub)}\,du\ \le\ C_f\,I_{\fd_P}\big(\sigma_s-(2kc-1)D\big)\,\delta_P(b)\,\abs{\det a_b}^{\sigma_s}\,\norm{a_b}^{D},
\end{equation}
and the right-hand side of \eqref{eq:UPtorusbound} is finite. If $f$ is an entire section, then $C_f$ can be chosen to be bounded when $s$ varies in a compact subset of
$\{\Re(s)>R_1\}$, and $s\mapsto\int_{U_P}f(s,\delta_0ub)\,du$ is holomorphic on $\Re(s)>R_1$.
\end{lemma}

\begin{proof}
We prove the assertion with
\begin{equation}\label{eq:RUP}
R_1:=\tfrac12-\tfrac{kc+1}2+(2kc-1)D+\fd_P,
\end{equation}
which is the constant $R$ of \eqref{eq:R0} for $\epsilon=1$; thus $\Re(s)>R_1$ is equivalent to $\sigma_s-(2kc-1)D>\fd_P$.
Using $\delta_0b=m(a_b)\delta_0$ and ${}^{\delta_0}U_P=\overline{U}_P$ (here, in fact, ${}^{\delta_0}u(Y)=v(-Y)$ by \eqref{eq:delta01}),
\[
\int_{U_P}\abs{f(s,\delta_0ub)}\,du=\delta_P(b)\int_{\overline{U}_P}\abs{f(s,m(a_b)v\delta_0)}\,dv .
\]
Write $v=v(X)=u\,m(a)\kappa$ as in \eqref{eq:siegelbounds}. Since $m(a_b)$
normalizes $U_P$, $m(a_b)u\,m(a)\kappa\delta_0=u'\,m(a_ba)\,\kappa\delta_0$ with $u'\in U_P$, and
\eqref{eq:fIwasawa} gives
$\abs{f(s,m(a_b)v\delta_0)}=\abs{\det(a_ba)}^{\sigma_s}\,\abs{W_{\kappa\delta_0}(a_ba)}$ with
$W_{\kappa\delta_0}\in W_{\psi}(\rho)$. By Theorem~\ref{thm:mgpsi} and \eqref{eq:siegelbounds}, there is $C_f\ge0$ such that
\[
\abs{f(s,m(a_b)v(X)\delta_0)}\ \le\ C_f\,\abs{\det a_b}^{\sigma_s}\norm{a_b}^{D}\,\abs{\det a}^{\sigma_s}\norm a^{D}
\ \le\ C_f\,\abs{\det a_b}^{\sigma_s}\norm{a_b}^{D}\,\vartheta(X)^{-(\sigma_s-(2kc-1)D)} ,
\]
where we used $\sigma_s>0$, and integrating over the $\fd_P$ coordinates of $X$ gives \eqref{eq:UPtorusbound}, where $I_{\fd_P}(\sigma_s-(2kc-1)D)$ is finite by \eqref{eq:Iea}.
For an entire section, by Remark~\ref{rem:sections}, $C_f$ can be chosen to be bounded when $s$ varies in a compact subset of $\{\Re(s)>R_1\}$. Then the left-hand side of \eqref{eq:UPtorusbound} is bounded on compact subsets of $\{\Re(s)>R_1\}$,
and the holomorphy follows.
\end{proof}

\begin{remark}\label{rem:U0Alonetorus}
The proof of Lemma~\ref{lem:U0Alone} also provides an explicit bound for the integral over $U_0$ of a right translate of $f$, without the passage to $U_P$.
Let $b\in M_P$ and assume that $b$ normalizes $U_0$ and $U_1$. As in Lemma~\ref{lem:UPtorus} let $a_b\in\GL_{kc}$ be defined by ${}^{\delta_0}b=m(a_b)$.
Since $\delta_1\in U_1$, we have $b^{-1}\delta_1b=u(\widehat{b_0})$, where $\widehat{b_0}=\left(\begin{smallmatrix}0&0\\b_0&0\end{smallmatrix}\right)$ with $b_0\in\Mat_c$.
Substituting $u_0\mapsto{}^{b}u_0$ and using $\delta_0b=m(a_b)\delta_0$,
\[
\int_{U_0}\abs{f(s,\delta u_0b)}\,du_0=\mathrm{mod}_{U_0}(b)\int_{U_0}\abs{f\big(s,m(a_b)\,\delta_0\,u(\widehat{b_0})\,u_0\big)}\,du_0 .
\]
For $u_0=u(Y_0)$ we have $\delta_0u(\widehat{b_0})u_0=v(X)\delta_0$ with $X=-(\widehat{b_0}+Y_0)$. The entries of $X$ are, up to a sign, the entries of $Y_0$ outside the bottom left
$c\times c$ block, and the entries of $b_0$ inside it, hence $\vartheta(X)\ge\vartheta(Y_0)$, with equality when $b_0\in\Mat_c(\Oo)$ (cf. \eqref{eq:nuXnuy}).
Let $R_0$ be given by \eqref{eq:RU0} and assume $\Re(s)>R_0$, so that $\sigma_s-(2kc-1)D>\fd_0$. Writing $v(X)=u\,m(a)\kappa$ as in \eqref{eq:siegelbounds},
the argument of the proof of Lemma~\ref{lem:UPtorus} implies that there is $C_f\ge0$ such that
\begin{align*}
\abs{f\big(s,m(a_b)v(X)\delta_0\big)}&\le C_f\,\abs{\det a_b}^{\sigma_s}\norm{a_b}^{D}\,\vartheta(X)^{-(\sigma_s-(2kc-1)D)}
\\&\le C_f\,\abs{\det a_b}^{\sigma_s}\norm{a_b}^{D}\,\vartheta(Y_0)^{-(\sigma_s-(2kc-1)D)},
\end{align*}
and integrating over the $\fd_0$ coordinates of $Y_0$ we obtain
\begin{equation}\label{eq:U0torusbound}
\int_{U_0}\abs{f(s,\delta u_0b)}\,du_0\ \le\ C_f\,I_{\fd_{0}}\big(\Re(s)-R_0+\fd_{0}\big)\,\mathrm{mod}_{U_0}(b)\,\abs{\det a_b}^{\sigma_s}\,\norm{a_b}^{D},
\end{equation}
where the right-hand side is finite by \eqref{eq:Iea}.
Since $U_P=U_0\times U_1$ and $b$ normalizes both subgroups, $\mathrm{mod}_{U_0}(b)=\delta_P(b)\,\mathrm{mod}_{U_1}(b)^{-1}$. Thus
\eqref{eq:U0torusbound} is the bound which the combination of Lemmas~\ref{lem:slice} and \ref{lem:UPtorus} provides, without the factor $\vol(K_{f,1})^{-1}$ and in the larger half-plane $\Re(s)>R_0$.
\end{remark}

\subsubsection{Torus elements}\label{sssec:torus}
The proof of convergence of \eqref{eq:Zdef} will be carried out using the Cartan decomposition of $G$ with respect to a maximal torus.
We introduce the relevant notation and collect a few results that we need.

By our choices of Borel subgroups $B_H$ and $B_G$ (see \S~\ref{sssec:doublingsetting}), we have
$B_G=T_G\ltimes N_G$ with $T_G=T_c\cap G$, and $B_H=T_H\ltimes N_H$ with $T_H=T_{2kc}\cap H$.
By \eqref{eq:embedding}, for \[t=\diag(t_1,\dots,t_\ell,t_\ell^{-1},\dots,t_1^{-1})\in T_G,\]
\begin{equation*}
\mathfrak t:=\mathfrak e_2({}^{\iota}t)=\diag\big(I_{(k-1)c+\ell},t_\ell^{-1},\dots,t_1^{-1},t_1,\dots,t_\ell,I_{\ell+(k-1)c}\big)\in T_H.
\end{equation*}
Thus $\mathfrak t=m(a_{\mathfrak t})$ with $a_{\mathfrak t}=\diag(I_{(k-1)c+\ell},t_\ell^{-1},\dots,t_1^{-1})\in\GL_{kc}$, and
${}^{\delta_0}\mathfrak t=m(a_{\mathfrak t}^*)$ (see \eqref{eq:delta01}) with
\begin{equation}\label{eq:at}
a_{\mathfrak t}^*=\diag(t_1,\dots,t_\ell,I_{(k-1)c+\ell}).
\end{equation}
Define
\begin{equation}\label{eq:xtdef}
x(t):=|\det a_{\mathfrak t}^*|=\prod_{i=1}^{\ell}\abs{t_i}.
\end{equation}
Since $\mathfrak t\in T_H$, $\mathfrak t$ normalizes $U_P$, $U_0$ and $U_1$. We have
\begin{equation}\label{eq:moduli}
\delta_P(\mathfrak t)=x(t)^{-(kc+1)},\qquad\mathrm{mod}_{U_0}(\mathfrak t)=x(t)^{-(k-1)c},\qquad \mathrm{mod}_{U_1}(\mathfrak t)=x(t)^{-(c+1)}.
\end{equation}

Also by \eqref{eq:fIwasawa},
\begin{equation}\label{eq:lefttranslationtfrak}
f(s,\delta_0\mathfrak th)=f(s,m(a_{\mathfrak t}^*)\delta_0h)=x(t)^{s-1/2+(kc+1)/2}W_{\delta_0h}(a_{\mathfrak t}^*).
\end{equation}
Note that \[\norm t=\norm{\mathfrak t}=\norm{a_{\mathfrak t}^*}.\]

The cone $T_G^-$ is the set of $t\in T_G$ with $\abs{\alpha(t)}\le1$ for all simple roots $\alpha$ determined by $T_G$ and $B_G$.
The Cartan decomposition of $G$ is
$G=\bigsqcup_tK_GtK_G$, the union over the representatives $t=\diag(\varpi^{e_1},\dots,\varpi^{e_\ell},\varpi^{-e_\ell},\dots,\varpi^{-e_1})$
with $e_1\ge\dots\ge e_\ell\ge0$ of $(T_G\cap K_G)\backslash T_G^-$.
For such a representative $t$,
$\norm t=\max_i\abs{t_i}^{\pm1}=\abs{t_1}^{-1}$ (see \S~\ref{ssec:height} property~\ref{it:h4}), and
\begin{equation}\label{eq:xt}
x(t)=q^{-(e_1+\ldots+e_\ell)}\ \le\ \norm t^{-1}\ \le\ 1 .
\end{equation}

We recall the bound on the volumes of the double cosets appearing in the Cartan decomposition: there is $C_G>0$ with
\begin{equation}\label{eq:volKtK}
\vol(K_GtK_G)\ \le\ C_G\,\delta_{B_G}(t)^{-1}\ \le\ C_G\,\norm t^{2\ell^2},\qquad \forall t\in T_G^- ,
\end{equation}
by \cite[\S I.1, (5)]{Wal03} (each of the $\ell^2$ positive roots $\alpha$ contributes $\abs{\alpha(t)}^{-1}\leq\norm t^{2}$).

Finally, the integral \eqref{eq:Zdef} also involves a matrix coefficient of $\pi^\vee$, and we need a coarse bound on its restriction to $T_G^-$.

\begin{lemma}\label{lem:mcbound}
Let $\pi\in\Irr(G)$ and let $A_\pi\ge0$ bound the absolute values of the real parts of the cuspidal exponents of $\pi$. There is $E\ge0$, depending only on $c$ and $A_\pi$, such that for every matrix coefficient $\omega$ of $\pi^\vee$ there is a constant $C_\omega\ge0$ satisfying
\begin{equation}\label{eq:mcbound}
\abs{\omega(t)}\ \le\ C_\omega\,\norm t^{E},\qquad \forall t\in T_G^-.
\end{equation}
If $\pi$ is unitary, we can take $E=0$.
\end{lemma}

\begin{proof}
The norm on $G$ defined in \cite[\S I.1]{Wal03} through the embedding $G\subset\GL_c$ is the restriction of the height \eqref{eq:height}.
Waldspurger \cite[Corollaire~I.4.4]{Wal03} proved that for every matrix coefficient $\omega$ of an admissible representation of $G$ there are $E\in\bZ_{\ge0}$ and $C_\omega>0$ such that
$\abs{\omega(g)}\le C_\omega\norm g^{E}$ for all $g\in G$, and that for a representation of finite length $E$ can be chosen independently of $\omega$.
Applying this to $\pi^\vee$ we obtain \eqref{eq:mcbound}. The proof of \textit{loc. cit.}, through \cite[Proposition~I.4.3]{Wal03}, showed that $E$ is determined by $c$ and by the absolute values of the real parts of the exponents of the Jacquet modules of $\pi^\vee$ along the standard parabolic subgroups of $G$, namely of the cuspidal exponents of $\pi^\vee$;
the dependence on $c$ enters through the comparison of these characters, and of the modulus characters, with powers of $\norm{\cdot}$. Since $J_{P}(\pi^\vee)\cong J_{\overline{P}}(\pi)^\vee$ for every parabolic subgroup $P$ \cite[Corollaire~I.4.2]{Wal03}, and $\overline{P}$ is conjugate to a standard parabolic subgroup,
the cuspidal exponents of $\pi^\vee$ are the inverses of Weyl conjugates of the cuspidal exponents of $\pi$; hence their real parts are bounded in absolute value by $A_\pi$, and $E$ depends only on $c$ and $A_\pi$. If $\pi$ is
unitary, with invariant inner product $(\cdot,\cdot)$, then $\pi^\vee$ is isomorphic to the complex conjugate of $\pi$, and every matrix coefficient of $\pi^\vee$ is of the form
$g\mapsto\overline{(\pi(g)\varphi_1,\varphi_2)}$; by the Cauchy--Schwarz inequality it is bounded on $G$ by $\norm{\varphi_1}\,\norm{\varphi_2}$, hence we can take $E=0$.
\end{proof}

\subsubsection{The reduction to the torus and the proof}\label{sssec:torusreduction}
Let $\omega$ be a matrix coefficient of $\pi^{\vee}$ and $f\in V(s,\rho)$.
We finally turn to the integral $Z(s,\omega,f)$ itself and prove that it is absolutely convergent, namely
\begin{equation*}
Z^{\mathrm{abs}}(s,\omega,f)=\int_G\int_{U_0}\abs{\omega(g)}\,\abs{f\big(s,\delta u_0\,\mathfrak e_2({}^{\iota}g)\big)}\,du_0\,dg <\infty,
\end{equation*}
in a right half-plane independent of $\omega$ and $f$.

\begin{proposition}\label{prop:torusreduction}
Let $\omega$ and $f$ be as above. There are finitely many matrix coefficients
$\omega_1,\dots,\omega_{r_1}$ of $\pi^\vee$ and finitely many sections $f_1,\dots,f_{r_2}\in V(s,\rho)$, obtained from $\omega$ and $f$ by translations by elements of
$K_G$ and $K_H$, such that
\begin{equation}\label{eq:torussum}
Z^{\mathrm{abs}}(s,\omega,f)\ \le\ \sum_{t\in(T_G\cap K_G)\backslash T_G^-}\vol(K_GtK_G)\,
\max_{1\le j\le r_1}\abs{\omega_j(t)}\,\max_{1\le l\le r_2}\int_{U_0}\abs{f_l(s,\delta u_0\mathfrak t)}\,du_0 .
\end{equation}
\end{proposition}

\begin{proof}
Since $\varphi$ and $\varphi^\vee$ are smooth the functions
$g\mapsto\omega(\kappa_1g\kappa_2)$, $\kappa_1,\kappa_2\in K_G$, form a finite set
$\{\omega_1,\dots,\omega_{r_1}\}$ of matrix coefficients of $\pi^\vee$. Similarly, $f$ is
smooth and the functions $f(s,\cdot\,\kappa)$, $\kappa\in K_H$, form a finite set
$\{f_1,\dots,f_{r_2}\}$. Let $g=\kappa_1t\kappa_2$ with $t\in T_G^-$ and
$\kappa_1,\kappa_2\in K_G$. Then ${}^{\iota}g={}^{\iota}\kappa_1\,{}^{\iota}t\,{}^{\iota}\kappa_2$
with ${}^{\iota}\kappa_i\in K_G$ ($\iota$ preserves $K_G$), and
\[
\mathfrak e_2({}^{\iota}g)=\mathfrak e_2({}^{\iota}\kappa_1)\,\mathfrak t\,\mathfrak e_2({}^{\iota}\kappa_2)
=\mathfrak e(\kappa_1,{}^{\iota}\kappa_1)\,\mathfrak e_1(\kappa_1^{-1})\,\mathfrak t\,\mathfrak e_2({}^{\iota}\kappa_2)
=\mathfrak e(\kappa_1,{}^{\iota}\kappa_1)\,\mathfrak t\,\mathfrak e_1(\kappa_1^{-1})\,\mathfrak e_2({}^{\iota}\kappa_2),
\]
because $\mathfrak e_1(G)$ and $\mathfrak e_2(G)$ commute. The element
$\kappa:=\mathfrak e_1(\kappa_1^{-1})\mathfrak e_2({}^{\iota}\kappa_2)=\mathfrak e(\kappa_1^{-1},{}^{\iota}\kappa_2)$
lies in $K_H$ ($\mathfrak e(K_G\times K_G)\subset K_H$). By Lemma~\ref{lem:abscor23} with $g_0=\kappa_1$ and $h=\mathfrak t\kappa$,

\[
\int_{U_0}\abs{f(s,\delta u_0\mathfrak e_2({}^{\iota}g))}\,du_0
=\int_{U_0}\abs{f(s,\delta u_0\mathfrak t\kappa)}\,du_0
\le\max_l\int_{U_0}\abs{f_l(s,\delta u_0\mathfrak t)}\,du_0 ,
\]
and $\abs{\omega(g)}=\abs{\omega(\kappa_1t\kappa_2)}\le\max_j\abs{\omega_j(t)}$. Both upper bounds depend on $g$ only through $t$; integrating over $K_GtK_G$ and summing over the
double cosets gives \eqref{eq:torussum}.
\end{proof}

\begin{theorem}\label{thm:doublingconv}
Let $G=\Sp_c$, $c=2\ell$, $k\ge1$, let $\pi\in\Irr(G)$ and let $\rho$ be
a $(k,c)$ representation of $\GL_{kc}$. Let $D=D(\rho)$ and let $E$ be an exponent as in Lemma~\ref{lem:mcbound}. There is
a real number $R$, depending only on $k$, $c$, $D$ and $E$, such that the following
holds for every $s$ with $\Re(s)>R$, every matrix coefficient $\omega$ of $\pi^\vee$ and
every $f\in V(s,\rho)$.
\begin{enumerate}[leftmargin=*,label=(\arabic*),ref=(\arabic*)]
\item\label{it:dc1} The sum over the cone with the domain $U_P$ in place of $U_0$ converges:
\begin{equation}\label{eq:strongUP}
\sum_{t\in(T_G\cap K_G)\backslash T_G^-}\vol(K_GtK_G)\,\abs{\omega(t)}\int_{U_P}\abs{f(s,\delta_0u\mathfrak t)}\,du\ <\ \infty .
\end{equation}
\item\label{it:dc2} The integral $Z(s,\omega,f)$ converges absolutely.
\end{enumerate}
For an entire section $f$, $Z(s,\omega,f)$ is holomorphic on $\Re(s)>R$.
\end{theorem}

\begin{proof}
We prove the assertions with
\begin{equation*}
R:=\max\big(R_1,\ \tfrac12-\tfrac{kc+1}2+D+kc+1+2\ell^2+E\big) ,
\end{equation*}
where $R_1$ is given by \eqref{eq:RUP}.
For \ref{it:dc1}, the hypothesis gives $\Re(s)>R_1$, so Lemma~\ref{lem:UPtorus} with $b=\mathfrak t$ and $a_b=a_{\mathfrak t}^*$ (see \eqref{eq:at} and \eqref{eq:xtdef}),
\eqref{eq:volKtK} and \eqref{eq:mcbound} bound the summand of \eqref{eq:strongUP} by
\[
C_GC_\omega C_f\,I_{\fd_P}\big(\sigma_s-(2kc-1)D\big)\,\norm t^{2\ell^2+E+D}\,\delta_P(\mathfrak t)\,x(t)^{\sigma_s}
\ \le\ C_GC_\omega C_f\,I_{\fd_P}\big(\sigma_s-(2kc-1)D\big)\,x(t)^{a},
\]
where $a:=\sigma_s-(kc+1)-2\ell^2-E-D$,
by \eqref{eq:moduli} and $\norm t\le x(t)^{-1}$ \eqref{eq:xt}; the hypothesis gives $a>0$.
The number of representatives $t$ with $x(t)=q^{-e}$ is the number of tuples
$e_1\ge\dots\ge e_\ell\ge0$ with sum $e$, at most $(e+1)^{\ell}$, so \eqref{eq:strongUP} is
bounded by a constant times $\sum_{e\ge0}(e+1)^{\ell}q^{-ae}<\infty$.

For \ref{it:dc2}, Proposition~\ref{prop:torusreduction} bounds $Z^{\mathrm{abs}}(s,\omega,f)$
by the sum \eqref{eq:torussum}, whose summands Lemma~\ref{lem:slice} bounds, for a compact
open subgroup $K_f$ fixing $f_1,\dots,f_{r_2}$, by
\[
\vol(K_{f,1})^{-1}\,\mathrm{mod}_{U_1}(\mathfrak t)^{-1}\,\vol(K_GtK_G)\,\max_j\abs{\omega_j(t)}\,\max_l\int_{U_P}\abs{f_l(s,\delta_0u\mathfrak t)}\,du .
\]
Since $\mathrm{mod}_{U_1}(\mathfrak t)^{-1}=x(t)^{c+1}\le1$ by \eqref{eq:moduli}, the sum is bounded by
$\vol(K_{f,1})^{-1}\sum_{j,l}$ of the sums \eqref{eq:strongUP} for $(\omega_j,f_l)$, which are
finite by \ref{it:dc1}. For an entire section $f$, the sections $f_l$ are entire as well, the constants $C_{f_l}$ of Lemma~\ref{lem:UPtorus} are bounded on compact subsets of $\{\Re(s)>R\}$, and the
subgroup $K_f$ is independent of $s$ (Remark~\ref{rem:sections}); hence $Z^{\mathrm{abs}}(s,\omega,f)$ is bounded on compact subsets of $\{\Re(s)>R\}$, and $Z(s,\omega,f)$ is holomorphic on $\Re(s)>R$ (see, e.g., \cite{Mat01}).
\end{proof}

\begin{remark}\label{rem:CFKorder}
In the proof of Theorem~\ref{thm:doublingconv}, Proposition~\ref{prop:torusreduction} is applied with the domain $U_0$, and the passage to $U_P$ (Lemma~\ref{lem:slice})
is carried out for each representative $t$ separately. This order of the two steps is the natural one, because the equivariance property of Lemma~\ref{lem:abscor23},
which underlies the reduction to the torus, holds for the domain $U_0$ and not for $U_P$. In the proof of \cite[Proposition~2.5]{CFK22} the passage to $U_P$ was
presented before the reduction to the torus (see \cite[(2.10)]{CFK22}), although the intended order was the one described here. One can also avoid the passage to $U_P$ altogether, using
Remark~\ref{rem:U0Alonetorus}.
\end{remark}
\begin{remark}\label{rem:SOeven}
One fact which simplifies the proof is that when we apply the Cartan decomposition we obtain elements $t$ with $x(t)\leq1$, and in fact $x(t)\to0$ along these representatives (see \eqref{eq:xt}).
This means that in \eqref{eq:lefttranslationtfrak} we obtain $x(t)^{s}$ decaying to $0$ fast enough when $\Re(s)$ is large. The situation is similar
when $G=\SO_{2\ell+1}$, but for $G=\SO_{c}$ with $c=2\ell$ the representatives are $\diag(\varpi^{e_1},\dots,\varpi^{e_\ell},\varpi^{-e_\ell},\dots,\varpi^{-e_1})$ with $e_1\ge\dots\ge e_{\ell-1}\ge\abs{e_\ell}$. When $\ell>2$ we have $e_1+\dots+e_{\ell-1}+e_\ell\geq e_1$ and $x(t)\leq\norm t^{-1}$ as in the symplectic case. For $\ell\leq 2$ one
can argue as in \cite{CFK-func} (after \cite[(3.45)]{CFK-func}), using the full orthogonal group $\mathrm{O}_c$, as follows.
Let
\[
w_1=\diag\Big(I_{\ell-1},\left(\begin{smallmatrix}&1\\1&\end{smallmatrix}\right),I_{\ell-1}\Big)\in\mathrm{O}_c(\Oo),\qquad\det w_1=-1.
\]
Conjugation by $w_1$ exchanges the coordinates $t_\ell$ and $t_\ell^{-1}$ of $t\in T_G$, so it maps the representative with the exponents $(e_1,\dots,e_{\ell-1},e_\ell)$ to
the element with the exponents $(e_1,\dots,e_{\ell-1},-e_\ell)$. The embedding \eqref{eq:embedding} extends to a homomorphism $\mathfrak e\colon\mathrm{O}_c\times\mathrm{O}_c\to\mathrm{O}_{2kc}$, and
$\det\mathfrak e(g_1,g_2)=(\det g_1)^{2k-1}\det g_2$; hence $\mathfrak e(g_1,g_2)\in H$ whenever $\det g_1=\det g_2$. In particular $\mathfrak e(g_0,{}^{\iota}g_0)\in H$ for all
$g_0\in\mathrm{O}_c$ (note that for $\SO_{c}$, $\iota$ is the conjugation by $\left(\begin{smallmatrix}&I_\ell\\-I_\ell\end{smallmatrix}\right)\in K_c$ and is a similitude of the form defining $\mathrm{O}_c$).
Lemma~\ref{lem:abscor23} remains valid for $g_0\in\mathrm{O}_c$: the element $\mathfrak e(g_0,{}^{\iota}g_0)$ normalizes $U$ and ${}^{\delta^{-1}}m(V_{(c^k)})$, the measure on ${}^{\delta^{-1}}m(V_{(c^k)})\backslash U$ is invariant under its action,
and ${}^{\delta}\mathfrak e(g_0,{}^{\iota}g_0)$ is still the image in $M_P$ of the diagonal embedding of $g_0$ in $\GL_c^{\triangle}$, so that \eqref{eq:diaginvariance} holds up to the factor
$\eta_\rho(g_0)$, and $\abs{\eta_\rho(g_0)}=\abs{\det g_0}^{e_\rho}=1$.

Now let $g=\kappa_1t\kappa_2$ with $\kappa_1,\kappa_2\in K_G$ and a representative $t$ such that $e_\ell<0$. Write $t=w_1^{-1}t'w_1$, so that the exponents of $t'$ are
$e_1\ge\dots\ge e_{\ell-1}\ge-e_\ell>0$, and put $\kappa_1'=\kappa_1w_1^{-1}$ and $\kappa_2'=w_1\kappa_2$, which belong to $\mathrm{O}_c(\Oo)$ and have determinant $-1$.
Let $\mathfrak t'=\mathfrak e_2({}^{\iota}t')$. The computation in the proof of Proposition~\ref{prop:torusreduction} gives
\[
\mathfrak e_2({}^{\iota}g)=\mathfrak e(\kappa_1',{}^{\iota}\kappa_1')\,\mathfrak t'\,\mathfrak e({\kappa_1'}^{-1},{}^{\iota}\kappa_2'),
\]
where $\kappa:=\mathfrak e({\kappa_1'}^{-1},{}^{\iota}\kappa_2')\in K_H$ because $\det{\kappa_1'}^{-1}=\det{}^{\iota}\kappa_2'=-1$. By Lemma~\ref{lem:abscor23} with $g_0=\kappa_1'$,
\[
\int_{U_0}\abs{f(s,\delta u_0\mathfrak e_2({}^{\iota}g))}\,du_0=\int_{U_0}\abs{f(s,\delta u_0\mathfrak t'\kappa)}\,du_0 .
\]
Thus in \eqref{eq:torussum} the inner integral at $\mathfrak t$ is replaced by the inner integral at $\mathfrak t'$, and $t'$ satisfies $\abs{t'_i}\le1$ for all $i$ and
$x(t')=q^{-(e_1+\,\dots\,+e_{\ell-1}-e_\ell)}\le\norm{t'}^{-1}=\norm t^{-1}$, as in \eqref{eq:xt}. The factors $\abs{\omega(g)}$ and $\vol(K_GtK_G)$ are treated as before, by \eqref{eq:mcbound} and \eqref{eq:volKtK}, since they involve the representative $t\in T_G^-$ and not $t'$. With this modification the proof of
Theorem~\ref{thm:doublingconv} applies. For $\ell=1$ we have $G=T_G\cong F^\times$, the representatives are $\diag(\varpi^{e},\varpi^{-e})$ with $e\in\bZ$, $w_1=\left(\begin{smallmatrix}&1\\1&\end{smallmatrix}\right)$
inverts $t$, and the argument reduces the representatives with $e<0$ to those with $e>0$. For $\ell=2$ it treats, in particular, the representatives with the exponents $(e,-e)$, $e>0$, for which
$x(t)=1$ while $\norm t=q^{e}$.
(In \cite[(3.45)]{CFK-func} the cases of $\SO_4$ and of $\GSpin_4$, which is handled similarly, were overlooked.)
\end{remark}
\begin{remark}\label{rem:Gspin}
The case of $\GSpin_c$ follows along the same lines as the proofs above, since the domain of integration is $C_G^{\bullet}\backslash G$ where $C_G^{\bullet}=F^\times$ is
the identity component of the center of $\GSpin_c$.
\end{remark}

\appendix
\section{The generalized doubling integral for \texorpdfstring{$\GL_c$}{GL(c)}}\label{sec:GLdoubling}
Let $c,k\ge1$ and $G=\GL_c$. In this appendix we prove the analog of Theorem~\ref{thm:doublingconv} for the doubling integral of $G$
\cite[\S2.4]{CFK22} (see Theorem~\ref{thm:GLdoublingconv}).
In contrast with \S~\ref{ssec:doubling}, the representatives for the Cartan decomposition of $G$ are the elements
\begin{equation}\label{eq:GLcone}
t=\diag(\varpi^{e_1},\dots,\varpi^{e_c}),\qquad e_1\ge\dots\ge e_c,
\end{equation}
where the integers $e_i$ are of arbitrary sign, and the section contributes the factor $\abs{\det t}^{\Re(s)}$, which is unbounded as $\Re(s)\to\infty$ when $e_1+\dots+e_c<0$.
We treat this factor in two ways. For the integral over $U_0$, which is the inner integral of the doubling integral, the element $\delta_1$ compensates for it, and
Theorem~\ref{thm:mgpsi} suffices, for all $k\ge1$ (Lemma~\ref{lem:GLU0}). For the integral over $U_P$ the element $\delta_1$ is absent; here we assume $k\ge2$ and
use the Schwartz--Bruhat function $\phi_1$ of \eqref{eq:thegauge2psi}, together with an identity for the Iwasawa coordinates (Lemmas~\ref{lem:GLlevi} and \ref{lem:GLUPtorus}).
For $k=1$ there is no such function, and the statement with the domain $U_P$
fails (see the proof of \cite[Proposition~2.5]{CFK22}); see Remark~\ref{rem:GLk1}.

\subsection{The setting}\label{ssec:GLsetting}
Let $H=\GL_{2kc}$. Define the Borel subgroups $B_H=B_{2kc}=T_H\ltimes N_H$ with $T_H=T_{2kc}$, and $B_{G}=B_c=T_G\ltimes N_G$, and the maximal compact subgroups
$K_H=K_{2kc}$, $K_G=K_c$ (see \S~\ref{ssec:notation}).

Let $Q=M_Q\ltimes U_Q$ be the parabolic subgroup $P_{(c^{k-1},2c,c^{k-1})}$ of $H$, so that
$M_Q\cong\GL_c^{k-1}\times H_0\times\GL_c^{k-1}$ with $H_0=\GL_{2c}$ occupying the middle $2c\times2c$ block. Let $U=U_Q$, with its character $\psi_U$ of
\cite[\S2.2]{CFK22}; for $k=1$, $Q=H$, $U$ is trivial and $\psi_U=1$. The group $G\times G$ is embedded in $H$ in the stabilizer of $\psi_U$ in $M_Q$.
The embedding $\mathfrak e\colon G\times G\to H$ is given by
\begin{equation*}
\mathfrak e(g_1,g_2)=\diag(g_1,\dots,g_1,g_2,g_1,\dots,g_1),
\end{equation*}
where $g_1$ appears $k$ times on the left of $g_2$ and $k-1$ times on the right \cite[\S2.2]{CFK22}.
As in \S~\ref{ssec:doubling}, $\mathfrak e_1(g)=\mathfrak e(g,I_c)$ and $\mathfrak e_2(g)=\mathfrak e(I_c,g)$, so that $\mathfrak e_1(G)$ and $\mathfrak e_2(G)$ commute;
here the involution $\iota$ is trivial.

Let $P=M_P\ltimes U_P$ be the parabolic subgroup $P_{(kc,kc)}$ of $H$, and write
\[
m(a_1,a_2)=\diag(a_1,a_2),\qquad u(Y)=\begin{pmatrix}I_{kc}&Y\\&I_{kc}\end{pmatrix},\qquad v(X)=\begin{pmatrix}I_{kc}&\\X&I_{kc}\end{pmatrix},
\]
for $a_1,a_2\in\GL_{kc}$ and $X,Y\in\Mat_{kc}$. Then $U_P=\{u(Y)\}$, $\overline{U}_P=\{v(X)\}$, the entries of $Y$ and $X$ are coordinates on $U_P$ and $\overline{U}_P$,
the Haar measures $du$ and $dv$ are the product measures, and $\fd_P:=\dim U_P=(kc)^2$. Also
$\delta_P(m(a_1,a_2))=\abs{\det a_1}^{kc}\abs{\det a_2}^{-kc}$ and \[m(a_1,a_2)u(Y)m(a_1,a_2)^{-1}=u(a_1Ya_2^{-1}).\]

Let $\rho$ and $\rho'$ be $(k,c)$ representations of $\GL_{kc}$. In \cite{CFK22} the authors fixed $\rho=\rho_c(\tau_0)$ and $\rho'=\chi^{-1}\rho_c(\tau_0^\vee)$ for
a generic $\tau_0\in\Irr(\GL_k)$ and a quasi-character $\chi$ of $F^\times$. For $s\in\bC$ let $V(s,\rho\otimes\rho')$ be the
space of the smooth functions $f\colon H\to\bC$ such that for each $h\in H$ there is a function $W_h$ in $W_\psi(\rho)\otimes W_\psi(\rho')$, regarded as a function on
$\GL_{kc}\times\GL_{kc}$, such that
\begin{equation}\label{eq:GLsection}
f(s,u\,m(a_1,a_2)h)=\delta_P^{1/2}(m(a_1,a_2))\,\abs{\det a_1}^{s-1/2}\abs{\det a_2}^{-s+1/2}\,W_h(a_1,a_2),\qquad u\in U_P .
\end{equation}
Put
\begin{equation}\label{eq:GLsigmas}
\sigma_s:=\Re(s)-\tfrac12+\tfrac{kc}2,
\end{equation}
so that
\begin{equation*}
\abs{f(s,u\,m(a_1,a_2)h)}=(\abs{\det a_1}/\abs{\det a_2})^{\sigma_s}\,\abs{W_h(a_1,a_2)}.
\end{equation*}

For $Z\in\Mat_c$ define $\widehat{Z}=\left(\begin{smallmatrix}0&0\\Z&0\end{smallmatrix}\right)\in\Mat_{kc}$.
Let
\begin{equation*}
\delta_0=\begin{pmatrix}&I_{kc}\\I_{kc}&\end{pmatrix},\qquad \delta_1=u(\widehat{I}_c)\in U_P\cap H_0,\qquad \delta=\delta_0\delta_1
\end{equation*}
(see \cite[\S2.4]{CFK22}). Note that ${}^{\delta_0}u(Y)=v(Y)$, in particular ${}^{\delta_0}U_P=\overline{U}_P$, and ${}^{\delta_0}m(a_1,a_2)=m(a_2,a_1)$.
In addition, for $g\in G$ we have $\mathfrak e(g,g)=m(g^{\triangle},g^{\triangle})$ with $g^{\triangle}=\diag(g,\dots,g)\in\GL_{kc}$ as in \S~\ref{ssec:kcdata}.
This element commutes with $\delta_0$, and also with $\delta_1$ because $g^{\triangle}\,\widehat{I}_c\,(g^{\triangle})^{-1}=\widehat{I}_c$.
Hence ${}^{\delta}\mathfrak e(g,g)=m(g^{\triangle},g^{\triangle})\in M_P$, and by \eqref{eq:GLsection} and \eqref{eq:etarho}, for $f\in V(s,\rho\otimes\rho')$ and $h\in H$,
\begin{equation}\label{eq:GLdiaginvariance}
f\big(s,{}^{\delta}\mathfrak e(g,g)\,h\big)=\eta_\rho(g)\eta_{\rho'}(g)\,f(s,h).
\end{equation}
This is the analog of \eqref{eq:diaginvariance}; in the setting of \cite{CFK22}, $\eta_\rho\eta_{\rho'}=\chi^{-k}\circ\det$.

Let $U_0=U\cap U_P$. This is the subgroup of the elements $u(Y_0)$ such that the bottom left $c\times c$ block of $Y_0\in\Mat_{kc}$ is $0$, and
\[\fd_0:=\dim U_0=(kc)^2-c^2.\]
As in \S~\ref{sssec:U0toUP} let $U_1=U_P\cap M_Q=U_P\cap H_0$; it is the subgroup of the elements $u(\widehat{Y}_1)$, $Y_1\in\Mat_c$. 
In particular $\delta_1\in U_1$. Since $U_P$ is abelian,
\begin{equation}\label{eq:GLUPsplit}
U_P=U_0\times U_1.
\end{equation}
For $k=1$ the group $U_0$ is trivial and $U_P=U_1$.

Let $\pi\in\Irr(G)$, $\omega$ be a matrix coefficient of $\pi^\vee$ and $f\in V(s,\rho\otimes\rho')$. The local integral is
(\cite[\S2.4]{CFK22})
\begin{equation}\label{eq:GLZdef}
Z(s,\omega,f)=\int_G\int_{U_0}\omega(g)\,f\big(s,\delta u_0\,\mathfrak e_2(g)\big)\,\psi_U(u_0)\,du_0\,dg ,
\end{equation}
where the restriction of $\psi_U$ to $U_0$ plays no role for the purpose of absolute convergence. Let $Z^{\mathrm{abs}}(s,\omega,f)$ be the integral \eqref{eq:GLZdef} with
$\omega$, $f$ and $\psi_U$ replaced by their absolute values.

The Cartan decomposition is $G=\bigsqcup_tK_GtK_G$, where $t$ varies over the elements \eqref{eq:GLcone}.
For $t=\diag(t_1,\dots,t_c)\in T_G$ put
\begin{equation}\label{eq:GLxpm}
x^{+}(t):=(\prod_{\abs{t_i}\le1}\abs{t_i})\le1,\qquad x^{-}(t):=(\prod_{\abs{t_i}>1}\abs{t_i})\ge1,\qquad x^{\mp}(t):=\frac{x^{-}(t)}{x^{+}(t)}=\prod_{i=1}^{c}\max(\abs{t_i},\abs{t_i}^{-1})\ge1,
\end{equation}
so that $\abs{\det t}=x^{+}(t)x^{-}(t)$, $\abs{\det t}^{\pm1}\le x^{\mp}(t)$ and $\norm t\le x^{\mp}(t)$ (see \S~\ref{ssec:height}, property~\ref{it:h4}).
For $t$ as in \eqref{eq:GLcone}, $x^{\mp}(t)=q^{\abs{e_1}+\dots+\abs{e_c}}$. Let
\begin{equation}\label{eq:GLfrakt}
\mathfrak t:=\mathfrak e_2(t)=m\big(I_{kc},a_t\big),\qquad a_t:=\diag(t,I_{(k-1)c}),\qquad {}^{\delta_0}\mathfrak t=m(a_t,I_{kc}) .
\end{equation}
The element $\mathfrak t\in T_H$ normalizes $U_P$, $U_0$ and $U_1$: ${}^{\mathfrak t}u(Y)=u(Ya_t^{-1})$, which multiplies the first $c$ columns of $Y$ by $t^{-1}$.
Hence, with the notation $\mathrm{mod}_{U_1}$ of \S~\ref{sssec:U0toUP},
\begin{equation}\label{eq:GLmoduli}
\delta_P(\mathfrak t)=\abs{\det t}^{-kc},\qquad \mathrm{mod}_{U_1}(\mathfrak t)=\abs{\det t}^{-c}.
\end{equation}

\subsection{The Levi coordinates}\label{ssec:GLlevi}
For the Iwasawa decomposition we will need a variant of Lemma~\ref{lem:iwasawaunip}, which we state and prove in this subsection.

Recall the notation $\extn{r}{g}$, the content $\ct$ and \eqref{eq:GLCB}--\eqref{eq:GLjacobi} of \S~\ref{ssec:height}. For $t\in T_G$, the nonzero $r\times r$ minors of $t$ are the products of $r$ of its nonzero entries, hence
\begin{equation}\label{eq:GLtminors}
\max_{0\le r\le c}\extn{r}{t}=x^{-}(t).
\end{equation}
For $X\in\Mat_{kc}$ and a subset $S\subset\{1,\dots,kc\}$ let $\ct^{*}_S(X)$ be the maximum of $1$ and of the absolute values of all the minors of $X$, of all sizes, whose
columns belong to $S$; recall that $\extn{0}{X}=1$, so that $\ct^{*}_{\emptyset}(X)=1$. Put $\ct^{*}_j(X):=\ct^{*}_{\{1,\dots,j\}}(X)$ for $0\le j\le kc$, and $\ct^{*}(X):=\ct^{*}_{kc}(X)$.
The function $\ct^{*}$ is related to the content as follows: the maximal minors of the $kc\times2kc$ matrix $(X\mid I_{kc})$ are, up to signs, the minors of $X$ of all sizes (the empty one included), hence
$\ct^{*}(X)=\ct((X\mid I_{kc}))=\max_{0\le r\le kc}\extn{r}{X}$ (cf. Example~\ref{ex:twoblock}), and similarly $\ct^{*}_S(X)=\ct((X_S\mid I_{kc}))$, where $X_S$ consists of the columns of $X$ in $S$.
By definition
\begin{equation}\label{eq:GLmumonotone}
1\le\ct^{*}_S(X)\le\ct^{*}_{S'}(X)\le\ct^{*}(X),\qquad S\subset S',\qquad\vartheta(X)\le\ct^{*}(X).
\end{equation}
Also recall that $\vartheta$ was defined for an arbitrary matrix by \eqref{eq:Nx}.
\begin{lemma}\label{lem:GLlevi}
Let $X\in\Mat_{kc}$. Write $v(X)=u\,m(a_1,a_2)\,\kappa$ with $u\in U_P$, $a_1,a_2\in\GL_{kc}$ and $\kappa\in K_H$, and $a_1=u'y\kappa'$ with
$u'\in V_{(c^k)}$, $y=\diag(y_1,\dots,y_k)\in M_{(c^k)}$ and $\kappa'\in K_{kc}$. Then:
\begin{enumerate}[leftmargin=*,label=(\roman*),ref=(\roman*)]
\item\label{it:GLl1} $\abs{\det a_2}=\abs{\det a_1}^{-1}=\ct^{*}(X)$ and $\norm{m(a_1,a_2)}\le\vartheta(X)^{2kc-1}$.
\item\label{it:GLl2} $\prod_{i=1}^{j}\abs{\det y_i}^{-1}=\ct^{*}_{jc}(X)$ for $1\le j\le k$.
\item\label{it:GLl3} $\extn{r}{y_1^{-1}}\le\ct^{*}_c(X)\le\ct^{*}(X)$ for $0\le r\le c$.
\item\label{it:GLl4} If $k\ge2$, then $\extn{r}{y_2}\le1$ for $0\le r\le c$.
\end{enumerate}
\end{lemma}
\begin{proof}
We argue as in the proof of Lemma~\ref{lem:iwasawaunip}\ref{it:iw1}, for the composition $(c^k,kc)$ of $2kc$. Put $v=v(X)$. The decomposition
\[
v=\big(u\,m(u',I_{kc})\big)\,\diag(y_1,\dots,y_k,a_2)\,\big(m(\kappa',I_{kc})\,\kappa\big)
\]
is an Iwasawa decomposition of $v\in\overline{U}_P\subset \bV_{(c^k,kc)}$, with respect to $P_{(c^k,kc)}$. Denote $L:=u\,m(u',I_{kc})\diag(y_1,\dots,y_k,a_2)\in P_{(c^k,kc)}$;
the projection of $L$ into $M_{(c^k,kc)}$ is $\diag(y_1,\dots,y_k,a_2)$.

Fix $1\le l\le k$, $0\le r\le c$ and a set $S_l$ of $r$ rows in the block $l$ (the $l$-th block of $c$ rows of the composition $(c^k,kc)$). Let $S$ be the union of $S_l$ and of the rows of the blocks $l+1,\dots,k$, and let $v_S$ (resp., $L_S$) be the matrix formed by the rows of
$v$ (resp., $L$) in $S$ and in the last block. We compute the content $\ct(v_S)$ in two ways.

First, $\ct(v_S)=\ct(L_S)$ by \eqref{eq:mukappa}. The rows of $L_S$ vanish in the columns of the blocks $1,\dots,l-1$, and in the remaining columns $L_S$ takes the form
\[
\begin{pmatrix}(y_l)_{S_l}&\ast\\0&L'\end{pmatrix},
\]
where $(y_l)_{S_l}$ consists of the rows of $y_l$ in $S_l$, and $L'$ is a block upper triangular square matrix with the diagonal blocks $y_{l+1},\dots,y_k,a_2$.
A maximal minor of $L_S$ vanishes unless it contains all the columns of $L'$ (otherwise the rows of $L'$ are linearly dependent in it), in which case it is the product of an $r\times r$ minor of $(y_l)_{S_l}$ and
$\det L'$. Hence
\begin{equation}\label{eq:GLLS}
\ct(v_S)=\extn{r}{(y_l)_{S_l}}\,\abs{\det a_2}\prod_{j=l+1}^{k}\abs{\det y_j}.
\end{equation}

Second, the rows of $v$ in $S$ are the standard basis vectors $e_i$, $i\in S$, and the last $kc$ rows of $v$ form the matrix $(X\mid I_{kc})$. A maximal minor of $v_S$ vanishes unless it contains the columns
$i\in S$; expanding along the rows $e_i$ and then along the columns of $I_{kc}$, such a minor is, up to a sign, a minor of $X$ whose columns belong to the complement $S^c$ of $S$ in $\{1,\dots,kc\}$,
and every such minor (the empty one included) occurs. Hence
\begin{equation}\label{eq:GLcontentRS}
\ct(v_S)=\ct^{*}_{S^c}(X).
\end{equation}

For $r=0$ the set $S_l$ is empty, $\extn{0}{(y_l)_{S_l}}=1$ and $S^c$ consists of the first $lc$ indices, so \eqref{eq:GLLS} and \eqref{eq:GLcontentRS} give
\begin{equation}\label{eq:GLLR}
\abs{\det a_2}\prod_{j=l+1}^{k}\abs{\det y_j}=\ct^{*}_{lc}(X),\qquad 1\le l\le k.
\end{equation}
In addition, since $\abs{\det v}=1$ and $\abs{\det\kappa}=\abs{\det\kappa'}=1$,
\begin{equation}\label{eq:GLLR0}
\abs{\det a_2}\prod_{j=1}^{k}\abs{\det y_j}=1.
\end{equation}
The case $l=k$ of \eqref{eq:GLLR} is $\abs{\det a_2}=\ct^{*}(X)$, and together with
\eqref{eq:twoblockbounds} (see Example~\ref{ex:twoblock}) we obtain \ref{it:GLl1}. Dividing \eqref{eq:GLLR0} by the case $l=j$ of \eqref{eq:GLLR} we obtain \ref{it:GLl2}.

For the remaining assertions, let $1\le l\le k$ and $0\le r\le c$ be arbitrary. Substituting \eqref{eq:GLLR} in \eqref{eq:GLLS} and comparing with \eqref{eq:GLcontentRS}, we obtain for every set $S_l$ of $r$ rows in the block $l$,
\begin{equation}\label{eq:GLySl}
\extn{r}{(y_l)_{S_l}}\,\ct^{*}_{lc}(X)=\ct^{*}_{S^c}(X).
\end{equation}
The set $S^c$ consists of the indices of the blocks $1,\dots,l-1$ and of the indices of the block $l$ which do not belong to $S_l$; in particular $S^c\subset\{1,\dots,lc\}$, and
$\ct^{*}_{S^c}(X)\le\ct^{*}_{lc}(X)$ by \eqref{eq:GLmumonotone}. Since $\ct^{*}_{lc}(X)\ge1$, we can divide \eqref{eq:GLySl} by $\ct^{*}_{lc}(X)$ and deduce $\extn{r}{(y_l)_{S_l}}\le1$.
Every $r\times r$ minor of $y_l$ is an $r\times r$ minor of the $r\times c$ matrix $(y_l)_{S_l}$, where $S_l$ is the set of the rows of this minor. Therefore
\begin{equation}\label{eq:GLylminors}
\extn{r}{y_l}=\max_{S_l}\extn{r}{(y_l)_{S_l}}\le1,\qquad 1\le l\le k,\quad 0\le r\le c,
\end{equation}
where $S_l$ varies over the sets of $r$ rows in the block $l$. For $k\ge2$ and $l=2$ this is \ref{it:GLl4}.

Let $k\ge1$ again and let $l=1$. Then \eqref{eq:GLylminors} with $c-r$ in place of $r$ reads $\extn{c-r}{y_1}\le1$. By \eqref{eq:GLjacobi} and \ref{it:GLl2} with $j=1$,
\[
\extn{r}{y_1^{-1}}=\abs{\det y_1}^{-1}\extn{c-r}{y_1}\le\abs{\det y_1}^{-1}=\ct^{*}_c(X)\le\ct^{*}(X),
\]
which is \ref{it:GLl3}.
\end{proof}

\subsection{The integral over \texorpdfstring{$U_P$}{UP}}\label{ssec:GLUP}
The following lemma is the analog of Lemma~\ref{lem:UPtorus}. In this subsection $k\ge2$. Let $D=D(\rho)$, $D'=D(\rho')$ and $e_\rho$ be as in \S~\ref{sec:applications}, and put
\begin{equation}\label{eq:GLA}
\beta_0:=(2kc-1)\big(c\abs{e_\rho}+2c(k-1)D+D'\big).
\end{equation}

\begin{lemma}\label{lem:GLUPtorus}
Assume $k\ge2$, let $f\in V(s,\rho\otimes\rho')$ and let $\beta_0$ be given by \eqref{eq:GLA}. There is a real number $R_0$, depending only on $k$, $c$, $D$, $e_\rho$ and $D'$, such that
for every $s$ with $\Re(s)>R_0$ there are a constant $C_f\ge0$ and an integer $l_f\ge0$ such that for all $t\in T_G$,
\begin{equation}\label{eq:GLUPtorusbound}
\int_{U_P}\abs{f(s,\delta_0u\mathfrak t)}\,du\ \le\ C_{f,s}\,I_{\fd_P}(\sigma_s/2-\beta_0)\,\delta_P(\mathfrak t)\,\abs{\det t}^{-D}\,x^{+}(t)^{\sigma_s}\,x^{-}(t)^{-\sigma_s/2},
\qquad C_{f,s}:=C_f\,q^{3l_fc\,\sigma_s/2},
\end{equation}
with $I_{\fd_P}$ as in \eqref{eq:Iea}, and the right-hand side of \eqref{eq:GLUPtorusbound} is finite. If $f$ is an entire section, then $l_f$ can be chosen independently of $s$,
$C_f$ can be chosen to be bounded when $s$ varies in a compact subset of $\{\Re(s)>R_0\}$ (hence so is $C_{f,s}$), and $s\mapsto\int_{U_P}f(s,\delta_0u\mathfrak t)\,du$ is holomorphic on $\Re(s)>R_0$.
\end{lemma}
\begin{proof}
We prove the assertion with
\begin{equation}\label{eq:GLRprime}
R_0:=\tfrac12-\tfrac{kc}2+2\big(\beta_0+\fd_P\big),
\end{equation}
so that $\Re(s)>R_0$ is equivalent to $\sigma_s/2-\beta_0>\fd_P$; in particular $\sigma_s>0$.

Using $\delta_0\mathfrak t=m(a_t,I_{kc})\delta_0$ (see \eqref{eq:GLfrakt}),
\begin{equation}\label{eq:GLsubst}
\int_{U_P}\abs{f(s,\delta_0u\mathfrak t)}\,du=\delta_P(\mathfrak t)\int_{\Mat_{kc}}\abs{f\big(s,m(a_t,I_{kc})\,v(X)\,\delta_0\big)}\,dX .
\end{equation}
Let $X\in\Mat_{kc}$ and write $v(X)=u\,m(a_1,a_2)\kappa$ and $a_1=u'y\kappa'$ as in Lemma~\ref{lem:GLlevi}.
Since $m(a_t,I_{kc})$ normalizes $U_P$, \eqref{eq:GLsection} and Lemma~\ref{lem:GLlevi}\ref{it:GLl1} imply
\begin{equation}\label{eq:GLintegrand}
\abs{f\big(s,m(a_t,I_{kc})v(X)\delta_0\big)}=\abs{\det t}^{\sigma_s}\,\ct^{*}(X)^{-2\sigma_s}\,\abs{W_{\kappa\delta_0}(a_ta_1,a_2)} .
\end{equation}
Since $f(s,\cdot)$ is smooth and $K_H$ is compact, the set $\{W_{\kappa\delta_0}:\kappa\in K_H\}$ is finite, and each of its elements is a finite sum of functions
$W\otimes W'$ with $W\in W_\psi(\rho)$ and $W'\in W_\psi(\rho')$. Fix one such function $W\otimes W'$.

By Theorem~\ref{thm:mgpsi} and Lemma~\ref{lem:GLlevi}\ref{it:GLl1}, there is $C'\ge0$, depending on $W'$, such that 
\begin{equation*}
\abs{W'(a_2)}\le C'\norm{a_2}^{D'}\le C'\vartheta(X)^{(2kc-1)D'}.
\end{equation*}
Since $a_t\in M_{(c^k)}$ normalizes $V_{(c^k)}$, $a_ta_1=({}^{a_t}u')\,\diag(ty_1,y_2,\dots,y_k)\,\kappa'$. By \eqref{eq:thegauge2psi} there are
$\phi_1,\dots,\phi_{k-1}\in\mathcal S(\Mat_c)$, depending on $W$, such that
\[
\abs{W(a_ta_1)}\le\abs{\det t}^{-D}\,\abs{\eta_\rho(y_k)}\prod_{i=1}^{k-1}\abs{\det(y_k^{-1}y_i)}^{-D}\ \phi_1(y_2^{-1}ty_1)\prod_{i=2}^{k-1}\phi_i(y_{i+1}^{-1}y_i),
\]
where we used $k\ge2$. By properties~\ref{it:h1} and \ref{it:h3} of \S~\ref{ssec:height} and by Lemma~\ref{lem:GLlevi}\ref{it:GLl1},
$\norm y\le\norm{a_1}\le\vartheta(X)^{2kc-1}$. Since $\abs{\eta_\rho(g)}=\abs{\det g}^{e_\rho}$ and $\abs{\det y_i}^{\pm1}\le\norm {y_i}^c\le\norm y^c$, we also have $|\eta_\rho(y_k)|\leq\norm y^{c|e_\rho|}$ and $\prod_{i=1}^{k-1}\abs{\det(y_k^{-1}y_i)}^{-D}\le\norm y^{2c(k-1)D}$. 
Hence, since $\phi_1,\ldots,\phi_{k-1}$ are bounded and $\phi_1$, the only one of them whose argument involves $t$, is compactly supported, there are
a constant $C$ and an integer $l\ge0$, depending on $W$ only, such that 
\begin{equation*}
\abs{W(a_ta_1)}\le C\,\abs{\det t}^{-D}\,\vartheta(X)^{(2kc-1)(c\abs{e_\rho}+2c(k-1)D)}\,\mathbf 1_{\varpi^{-l}\Mat_c(\Oo)}(y_2^{-1}ty_1).
\end{equation*}

Assume that $Z:=y_2^{-1}ty_1\in\varpi^{-l}\Mat_c(\Oo)$. Then $t=y_2Zy_1^{-1}$, $\extn{r}{Z}\le q^{lr}$, and by \eqref{eq:GLCB} and
Lemma~\ref{lem:GLlevi}\ref{it:GLl3} and \ref{it:GLl4}, $\extn{r}{t}\le q^{lr}\ct^{*}(X)$ for all $0\le r\le c$. By \eqref{eq:GLtminors},
\begin{equation}\label{eq:GLkey}
x^{-}(t)\le q^{lc}\,\ct^{*}(X).
\end{equation}
Since $\sigma_s>0$, \eqref{eq:GLkey} and $\ct^{*}(X)\ge\vartheta(X)$ (see \eqref{eq:GLmumonotone}) imply
\begin{equation}\label{eq:GLsplitmu}
\ct^{*}(X)^{-2\sigma_s}=\ct^{*}(X)^{-3\sigma_s/2}\ct^{*}(X)^{-\sigma_s/2}\le q^{3lc\,\sigma_s/2}\,x^{-}(t)^{-3\sigma_s/2}\,\vartheta(X)^{-\sigma_s/2}.
\end{equation}
Combining \eqref{eq:GLintegrand}--\eqref{eq:GLsplitmu} with $\abs{\det t}=x^{+}(t)x^{-}(t)$, and summing over the finitely many functions $W\otimes W'$, we obtain $C_f\ge0$ and an integer $l_f\ge0$ such that
\[
\abs{f\big(s,m(a_t,I_{kc})v(X)\delta_0\big)}\le C_f\,q^{3l_f\,c\,\sigma_s/2}\,\abs{\det t}^{-D}\,x^{+}(t)^{\sigma_s}\,x^{-}(t)^{-\sigma_s/2}\,\vartheta(X)^{-(\sigma_s/2-\beta_0)},
\]
where $C_f$ and $l_f$ depend on the functions $W_{\kappa\delta_0}$, $\kappa\in K_H$, only. Substituting this bound in \eqref{eq:GLsubst} and integrating over the $\fd_P$ coordinates
of $X$, we obtain \eqref{eq:GLUPtorusbound}, where $I_{\fd_P}(\sigma_s/2-\beta_0)$ is finite by \eqref{eq:Iea} because $\sigma_s/2-\beta_0>\fd_P$.
For an entire section, by Remark~\ref{rem:sections}, $l_f$ can be chosen independently of $s$, and $C_f$ bounded when $s$ varies in a compact subset of $\{\Re(s)>R_0\}$.
Then the left-hand side of \eqref{eq:GLUPtorusbound} is bounded on compact subsets of $\{\Re(s)>R_0\}$, and the holomorphy follows.
\end{proof}

\subsection{The integral over \texorpdfstring{$U_0$}{U0}}\label{ssec:GLU0}
In this subsection $k\ge1$.
Let $D=D(\rho)$ and $D'=D(\rho')$. The element $\mathfrak t$ normalizes $U_0$, and by \eqref{eq:GLmoduli} and \eqref{eq:GLUPsplit},
\begin{equation}\label{eq:GLmodU0}
\mathrm{mod}_{U_0}(\mathfrak t)=\delta_P(\mathfrak t)\,\mathrm{mod}_{U_1}(\mathfrak t)^{-1}=\abs{\det t}^{-(k-1)c}.
\end{equation}
The following lemma is the analog of Remark~\ref{rem:U0Alonetorus}.
But in the present setting, when $x^{-}(t)>1$, the entries of the conjugate of $\delta_1$ by $\mathfrak t^{-1}$ provide the decay in $x^{-}(t)$.

\begin{lemma}\label{lem:GLU0}
Let $k\ge1$ and $f\in V(s,\rho\otimes\rho')$. There is a real number $R_1$, depending only on $k$, $c$, $D$ and $D'$, such that for every $s$ with $\Re(s)>R_1$
there is a constant $C_{f}\ge0$ such that for all $t\in T_G$,
\begin{equation}\label{eq:GLU0bound}
\int_{U_0}\abs{f(s,\delta u_0\mathfrak t)}\,du_0\ \le\ C_{f}I_{\fd_0}(\sigma_s/2-\beta_1)\,\mathrm{mod}_{U_0}(\mathfrak t)\,\norm t^{D+\beta_1}\,x^{+}(t)^{\sigma_s}\,x^{-}(t)^{-\sigma_s/2},
\end{equation}
where $\beta_1:=(2kc-1)(D+D')$, and the right-hand side of \eqref{eq:GLU0bound} is finite. If $f$ is an entire section, then $C_f$ can be chosen to be bounded when $s$ varies in a compact subset of
$\{\Re(s)>R_1\}$, and $s\mapsto\int_{U_0}f(s,\delta u_0\mathfrak t)\,du_0$ is holomorphic on $\Re(s)>R_1$.
\end{lemma}
\begin{proof}
Recall that $\fd_0=\dim U_0=(kc)^2-c^2$. We prove the assertion with
\begin{equation}\label{eq:GLRU0}
R_1:=\tfrac12-\tfrac{kc}2+2(\beta_1+\fd_0),
\end{equation}
so that $\Re(s)>R_1$ is equivalent to $\sigma_s/2-\beta_1>\fd_0$; in particular $\sigma_s>0$.

Since ${}^{\mathfrak t}u(Y)=u(Ya_t^{-1})$, we have $\mathfrak t^{-1}\delta_1\mathfrak t=u(\widehat{t})$. Substituting $u_0\mapsto{}^{\mathfrak t}u_0$ and using $\delta_0\mathfrak t=m(a_t,I_{kc})\delta_0$ and ${}^{\delta_0}u(Y)=v(Y)$,
\begin{equation}\label{eq:GLU0subst}
\int_{U_0}\abs{f(s,\delta u_0\mathfrak t)}\,du_0=\mathrm{mod}_{U_0}(\mathfrak t)\int_{U_0}\abs{f\big(s,m(a_t,I_{kc})\,v(X)\,\delta_0\big)}\,du_0,\qquad X:=\widehat{t}+Y_0,
\end{equation}
where $u_0=u(Y_0)$.
For $k=1$ we have $X=t$ and there is no integration.

Write $v(X)=u\,m(a_1,a_2)\kappa$ as in Lemma~\ref{lem:GLlevi}. As in \eqref{eq:GLintegrand},
\[
\abs{f\big(s,m(a_t,I_{kc})v(X)\delta_0\big)}=\abs{\det t}^{\sigma_s}\,\ct^{*}(X)^{-2\sigma_s}\,\abs{W_{\kappa\delta_0}(a_ta_1,a_2)} ,
\]
and the functions $W_{\kappa\delta_0}$, $\kappa\in K_H$, are finite sums of functions $W\otimes W'$ from a finite set. Note that
$\norm{a_t}=\norm t$. By Theorem~\ref{thm:mgpsi} and Lemma~\ref{lem:GLlevi}\ref{it:GLl1}, there is $C\ge0$, depending on $W$ and $W'$, such that 
\[
\abs{W(a_ta_1)W'(a_2)}\le C\,\norm{t}^{D}\norm{a_1}^{D}\norm{a_2}^{D'}\le C\,\norm t^{D}\,\vartheta(X)^{\beta_1}.
\]
The entries of $X$ are the entries of $Y_0$ and of $t$, so $\vartheta(X)\le\norm t\,\vartheta(Y_0)$.

We turn to bound $\ct^{*}(X)$ from below. First, $\ct^{*}(X)\ge\vartheta(X)\ge\vartheta(Y_0)$ by \eqref{eq:GLmumonotone}. Second, let $S$ be the set of indices $1\le i\le c$ with $\abs{t_i}>1$.
The submatrix of $X$ on the rows $(k-1)c+i$, $i\in S$, and the columns $i\in S$, lies in the bottom left $c\times c$ block of $X$, so it is the diagonal matrix with the entries $t_i$, $i\in S$,
and its determinant has the absolute value $x^{-}(t)$. Hence
\begin{equation}\label{eq:GLkeyU0}
\ct^{*}(X)\ge\max\big(x^{-}(t),\vartheta(Y_0)\big).
\end{equation}
Since $\sigma_s>0$, \eqref{eq:GLkeyU0} implies $\ct^{*}(X)^{-2\sigma_s}\le x^{-}(t)^{-3\sigma_s/2}\,\vartheta(Y_0)^{-\sigma_s/2}$. Combining the bounds (including
$\vartheta(X)\le\norm t\,\vartheta(Y_0)$) with $\abs{\det t}=x^{+}(t)x^{-}(t)$, and letting $C_f$ be the maximum over the constants $C$ above, we obtain
\[
\abs{f\big(s,m(a_t,I_{kc})v(X)\delta_0\big)}\le C_f\,\norm t^{D+\beta_1}\,x^{+}(t)^{\sigma_s}\,x^{-}(t)^{-\sigma_s/2}\,\vartheta(Y_0)^{-(\sigma_s/2-\beta_1)} .
\]
For $k=1$ this is \eqref{eq:GLU0bound} ($\fd_0=0$ so $I_{\fd_0}\equiv1$). For $k\ge2$, substituting this bound in \eqref{eq:GLU0subst} and integrating over the $\fd_0$ coordinates of $Y_0$,
we obtain \eqref{eq:GLU0bound}, where $I_{\fd_0}(\sigma_s/2-\beta_1)$ is finite by \eqref{eq:Iea} because $\sigma_s/2-\beta_1>\fd_0$.
The assertions on entire sections follow as in the proof of Lemma~\ref{lem:GLUPtorus} (the integer $l_f$ does not occur here).
\end{proof}

\begin{example}\label{ex:GLU0}
Let $k=c=1$, so that $H=\GL_2$, $U_0$ is trivial, $\rho$ and $\rho'$ are quasi-characters of $F^\times$, and $X=t\in F^\times$. Then
\[
\delta\,\mathfrak t=\begin{pmatrix}&1\\1&\end{pmatrix}\begin{pmatrix}1&1\\&1\end{pmatrix}\begin{pmatrix}1&\\&t\end{pmatrix}=\begin{pmatrix}&t\\1&t\end{pmatrix}=m(t,1)\,v(t)\,\delta_0,
\qquad\ct^{*}(t)=\max(1,\abs t).
\]
For $\abs t\le1$ we have $v(t)\in K_H$, and for $\abs t>1$ we have $v(t)=u(t^{-1})\,m(t^{-1},t)\left(\begin{smallmatrix}&-1\\1&t^{-1}\end{smallmatrix}\right)$. Hence
\[
\abs{f(s,\delta\mathfrak t)}=\abs t^{\sigma_s}\max(1,\abs t)^{-2\sigma_s}
C\abs{\rho(ta_1)\rho'(a_2)}=C\min(\abs t,\abs t^{-1})^{\sigma_s}\abs{\rho(ta_1)\rho'(a_2)},
\]
with $(a_1,a_2)=(1,1)$ or $(t^{-1},t)$, and $Z(s,\omega,f)$ is an integral of Tate type.
\end{example}
\begin{example}\label{ex:GLk2c1}
Let $k=2$ and $c=1$, so that $H=\GL_4$. For $u_0=u(Y_0)$,
\[
X=\begin{pmatrix}x_1&x_2\\t&x_3\end{pmatrix},\qquad du_0=dx_1\,dx_2\,dx_3,\qquad \ct^{*}(X)=\max\big(1,\abs t,\abs{x_1},\abs{x_2},\abs{x_3},\abs{x_1x_3-tx_2}\big),
\]
and \eqref{eq:GLkeyU0} reads $\ct^{*}(X)\ge\max(1,\abs t,\abs{x_1},\abs{x_2},\abs{x_3})$.
\end{example}

\subsection{The theorem}\label{ssec:GLtheorem}
For the representatives \eqref{eq:GLcone} we have, as in \eqref{eq:volKtK},
\begin{equation}\label{eq:GLvolKtK}
\vol(K_GtK_G)\le C_G\,\delta_{B_G}(t)^{-1}\le C_G\,x^{\mp}(t)^{c-1},
\end{equation}
because $\delta_{B_G}(t)^{-1}=\prod_{i<j}\abs{t_j/t_i}=q^{\sum_{i=1}^{c}(c+1-2i)e_i}$. As in Lemma~\ref{lem:mcbound}, by \cite{Wal03} there is $E\ge0$, depending only on $c$ and on a bound for the absolute values of the real parts of the cuspidal exponents of $\pi$,
such that every matrix coefficient $\omega$ of $\pi^\vee$ satisfies
\begin{equation}\label{eq:GLmcbound}
\abs{\omega(t)}\le C_\omega\,\norm t^{E}\le C_\omega\,x^{\mp}(t)^{E}
\end{equation}
for all $t$ as in \eqref{eq:GLcone}. If $\pi=\pi_0\abs{\det}^{r}$ with $\pi_0$ unitary and $r\in\bR$, then we can take $E=c\abs{r}$, because
$\omega=\omega_0\abs{\det}^{-r}$ where $\omega_0$ is a matrix coefficient of $\pi_0^\vee$, which is bounded,
and $\abs{\det t}^{\pm1}\le\norm t^{c}$ by property~\ref{it:h4} of \S~\ref{ssec:height}.

\begin{theorem}\label{thm:GLdoublingconv}
Let $G=\GL_c$, $k\ge1$, let $\pi\in\Irr(G)$, and let $\rho$ and $\rho'$ be $(k,c)$ representations of $\GL_{kc}$. There is a real number $R$,
depending only on $k$, $c$, $e_\rho$, $D(\rho)$, $D(\rho')$ and $E$, such that the following holds for every $s$ with $\Re(s)>R$, every matrix coefficient $\omega$ of
$\pi^\vee$ and every $f\in V(s,\rho\otimes\rho')$.
\begin{enumerate}[leftmargin=*,label=(\arabic*),ref=(\arabic*)]
\item\label{it:GLdc1} If $k\ge2$, the sum over the representatives \eqref{eq:GLcone} with the domain $U_P$ converges:
\begin{equation}\label{eq:GLstrongUP}
\sum_{t}\vol(K_GtK_G)\,\abs{\omega(t)}\int_{U_P}\abs{f(s,\delta_0u\mathfrak t)}\,du<\infty .
\end{equation}
\item\label{it:GLdc2} The integral $Z(s,\omega,f)$ is absolutely convergent.
\end{enumerate}
For an entire section $f$, $Z(s,\omega,f)$ is holomorphic on $\Re(s)>R$.
\end{theorem}
\begin{proof}
Let $\beta_0$ be given by \eqref{eq:GLA}, $\beta_1$ be as in Lemma~\ref{lem:GLU0}, and $R_0$ and $R_1$ be given by \eqref{eq:GLRprime} and \eqref{eq:GLRU0}.
We prove the assertions with $R:=R_2+2c$ for $k\ge2$ and $R:=R_3$ for $k=1$, where
\begin{equation*}
\begin{split}
R_2&:=\max\big(R_0,\ \tfrac12-\tfrac{kc}2+2(c-1+kc+D+E)\big),\\
R_3&:=\max\big(R_1,\ \tfrac12-\tfrac{c}2+2(c-1+D+\beta_1+E)\big).
\end{split}
\end{equation*}
Assume $k\ge2$ and let $\Re(s)>R_2$. Since $R_2\ge R_0$, Lemma~\ref{lem:GLUPtorus} applies; in particular $\sigma_s/2-\beta_0>\fd_P$ (see \eqref{eq:GLsigmas} and \eqref{eq:GLRprime}), so that
$I_{\fd_P}(\sigma_s/2-\beta_0)$ is finite. By Lemma~\ref{lem:GLUPtorus}, \eqref{eq:GLmoduli}, \eqref{eq:GLvolKtK}, \eqref{eq:GLmcbound} and $\abs{\det t}^{-1}\le x^{\mp}(t)$, the summand of
\eqref{eq:GLstrongUP} is at most
\[
C\,x^{\mp}(t)^{c-1+E+kc+D}\,x^{+}(t)^{\sigma_s}x^{-}(t)^{-\sigma_s/2}\le C\,x^{\mp}(t)^{-a},\qquad a:=\tfrac{\sigma_s}2-(c-1)-E-kc-D,
\]
because $x^{+}(t)^{\sigma_s}x^{-}(t)^{-\sigma_s/2}\le x^{\mp}(t)^{-\sigma_s/2}$ by \eqref{eq:GLxpm}. Here $C:=C_GC_\omega C_{f,s}\,I_{\fd_P}(\sigma_s/2-\beta_0)<\infty$. 
The inequality $\Re(s)>\tfrac12-\tfrac{kc}2+2(c-1+kc+D+E)$, which holds since $\Re(s)>R_2$, is equivalent to $a>0$. The number of the representatives $t$ with $x^{\mp}(t)=q^{j}$ is at most
$(2j+1)^c$, so \eqref{eq:GLstrongUP} is bounded by a constant times $\sum_{j\ge0}(2j+1)^cq^{-aj}<\infty$. This proves \ref{it:GLdc1} for $\Re(s)>R_2$, in particular for $\Re(s)>R$.

For $k\ge2$, \ref{it:GLdc2} also follows from \ref{it:GLdc1}. First we reduce $Z^{\mathrm{abs}}(s,\omega,f)$ to a sum over the representatives \eqref{eq:GLcone};
this reduction is valid for all $k\ge1$. Let $\kappa_0\in K_G$. The element $\mathfrak e(\kappa_0,\kappa_0)=m(\kappa_0^{\triangle},\kappa_0^{\triangle})$ normalizes $U_0$ and preserves $du_0$:
indeed ${}^{\mathfrak e(\kappa_0,\kappa_0)}u(Y)=u(\kappa_0^{\triangle}Y(\kappa_0^{\triangle})^{-1})$, and the bottom left $c\times c$ block of $\kappa_0^{\triangle}Y(\kappa_0^{\triangle})^{-1}$ is the conjugate by $\kappa_0$ of
the bottom left $c\times c$ block of $Y$. Since $\abs{\eta_\rho(\kappa_0)\eta_{\rho'}(\kappa_0)}=1$, the substitution $u_0\mapsto{}^{\mathfrak e(\kappa_0,\kappa_0)}u_0$ and \eqref{eq:GLdiaginvariance} imply, for all $h\in H$,
\begin{equation}\label{eq:GLabscor}
\int_{U_0}\abs{f(s,\delta u_0\,\mathfrak e(\kappa_0,\kappa_0)h)}\,du_0=\int_{U_0}\abs{f(s,\delta u_0h)}\,du_0 .
\end{equation}
Let $g=\kappa_1t\kappa_2$ with $\kappa_1,\kappa_2\in K_G$ and $t$ as in \eqref{eq:GLcone}. Since $\mathfrak e_1(G)$ and $\mathfrak e_2(G)$ commute,
$\mathfrak e_2(g)=\mathfrak e(\kappa_1,\kappa_1)\,\mathfrak t\,\mathfrak e(\kappa_1^{-1},\kappa_2)$ and $\mathfrak e(\kappa_1^{-1},\kappa_2)\in K_H$. The argument of the proof of
Proposition~\ref{prop:torusreduction}, with \eqref{eq:GLabscor} in place of Lemma~\ref{lem:abscor23}, provides finitely many matrix coefficients $\omega_j$ of $\pi^\vee$
and sections $f_i\in V(s,\rho\otimes\rho')$, which are translates of $\omega$ and $f$ by elements of $K_G$ and $K_H$, such that
\begin{equation}\label{eq:GLtorussum}
Z^{\mathrm{abs}}(s,\omega,f)\le\sum_t\vol(K_GtK_G)\max_j\abs{\omega_j(t)}\max_i\int_{U_0}\abs{f_i(s,\delta u_0\mathfrak t)}\,du_0 .
\end{equation}
Now the proof of Lemma~\ref{lem:slice} applies with the present $U_P$ and $U_0$, by \eqref{eq:GLUPsplit} and the commutativity of $U_P$;
hence the analog of \eqref{eq:slice} holds for $b=\mathfrak t$ and a compact open subgroup $K_f$ fixing the sections $f_i$, with
$\mathrm{mod}_{U_1}(\mathfrak t)^{-1}=\abs{\det t}^{c}\le x^{\mp}(t)^{c}$ by \eqref{eq:GLmoduli}. Combined with the bound $C\,x^{\mp}(t)^{-a}$ obtained above for the pairs $(\omega_j,f_i)$,
the summand of \eqref{eq:GLtorussum} is at most a constant multiple of $x^{\mp}(t)^{c-a}$. If $\Re(s)>R=R_2+2c$, then $a>c$, and the sum
converges as above. This proves \ref{it:GLdc2} for $k\ge2$. If $f$ is an entire section, the holomorphy follows as in the proof of Theorem~\ref{thm:doublingconv}, using Lemma~\ref{lem:GLUPtorus}.

Now let $k=1$ and $\Re(s)>R_3$. Since $R_3\ge R_1$, Lemma~\ref{lem:GLU0} applies, and
by Lemma~\ref{lem:GLU0}, \eqref{eq:GLmodU0}, \eqref{eq:GLvolKtK}, \eqref{eq:GLmcbound} and $\norm t\le x^{\mp}(t)$, the summand of \eqref{eq:GLtorussum} is at most a constant multiple of
\[
x^{\mp}(t)^{c-1+E+D+\beta_1}\,x^{+}(t)^{\sigma_s}x^{-}(t)^{-\sigma_s/2}\le x^{\mp}(t)^{-a'},\qquad a':=\tfrac{\sigma_s}2-(c-1)-E-D-\beta_1,
\]
and $a'>0$ since $\Re(s)>R_3$. The sum converges as above, which proves \ref{it:GLdc2}. The holomorphy follows as above, with Lemma~\ref{lem:GLU0} in place of Lemma~\ref{lem:GLUPtorus}.
\end{proof}

\begin{remark}\label{rem:GLU0allk}
The argument of the last paragraph of the proof of Theorem~\ref{thm:GLdoublingconv} applies to every $k\ge1$, with $R_3:=\max\big(R_1,\ \tfrac12-\tfrac{kc}2+2(c-1+(k-1)c+D+\beta_1+E)\big)$:
by Lemma~\ref{lem:GLU0}, \eqref{eq:GLmodU0}, \eqref{eq:GLvolKtK}, \eqref{eq:GLmcbound}, $\norm t\le x^{\mp}(t)$ and $\abs{\det t}^{-1}\le x^{\mp}(t)$, the summand of \eqref{eq:GLtorussum} is at most a constant multiple of
$x^{\mp}(t)^{-a'}$, $a':=\tfrac{\sigma_s}2-(c-1)-E-(k-1)c-D-\beta_1>0$.
Thus Theorem~\ref{thm:GLdoublingconv}\ref{it:GLdc2} and the holomorphy can be obtained without the passage to $U_P$.
\end{remark}

\begin{remark}\label{rem:GLkey}
The inequality \eqref{eq:GLkey} is the point of the proof of Lemma~\ref{lem:GLUPtorus}: on the support of the integrand, the part $x^{-}(t)$ of $t$ which makes $\abs{\det t}^{\sigma_s}$ large is dominated
by $\ct^{*}(X)=\abs{\det a_2}$, and the section contributes $\ct^{*}(X)^{-2\sigma_s}$, the exponent being $2\sigma_s$ because $s$ appears in the two blocks of $M_P$ with opposite signs.
A bound of the form $\abs{W(a_ta_1)}\le C\norm{a_ta_1}^{D}$, which is what Theorem~\ref{thm:mgpsi} provides, does not suffice for the integral over $U_P$ (in contrast with Lemma~\ref{lem:GLU0}, where \eqref{eq:GLkeyU0} holds for all $t$, not only on the support of the integrand), since the gain required is a power of $\ct^{*}(X)$
proportional to $\sigma_s$. In the proof of \cite[Proposition~2.5]{CFK22} the inner integral was bounded by $\upsilon(t)\abs{\det t}^{\Re(s)d+d'}\int_{U_P}\abs{f(s,\delta_0u)}\,du$,
where $\upsilon$ is a gauge which vanishes unless the coordinates of $t$ are small. This separation of $t$ from $u$ is not available as stated, because the argument of the
$(k,c)$ function is $a_ta_1$ and not $a_t$: by \eqref{eq:GLkey} the coordinates of $t$ are small only relative to $\ct^{*}(X)$. Lemma~\ref{lem:GLUPtorus} is the corrected form of this bound,
and it implies the statement \eqref{eq:GLstrongUP}, which is the content of \cite[(2.10)]{CFK22} for $G=\GL_c$.
\end{remark}

\begin{remark}\label{rem:GLk1}
For $k=1$ the bound \eqref{eq:thegauge2psi} contains no Schwartz--Bruhat function, and
\eqref{eq:GLstrongUP} fails, as noted in the proof of \cite[Proposition~2.5]{CFK22}. For example, let $c=1$: then $U_P=U_1$, and for $\abs t>1$ and $\abs x\le1$ we have
$\abs{f(s,m(t,1)v(x)\delta_0)}=\abs t^{\sigma_s}\abs{\eta_\rho(t)}\abs{f(s,\delta_0)}$ whenever $f(s,\cdot)$ is right-invariant under $u(\Oo)$, and this is unbounded in $t$ for $\Re(s)\gg0$.
The integral $Z(s,\omega,f)$ itself converges by Theorem~\ref{thm:GLdoublingconv}\ref{it:GLdc2}, and $\delta_1$ plays a key role in this, through \eqref{eq:GLkeyU0}; see Example~\ref{ex:GLU0}.
\end{remark}

\bibliographystyle{amsalpha}
\bibliography{lemma314-final}

\end{document}